\documentclass{amsart}
\usepackage{amsfonts}

\usepackage{amsmath}

\usepackage{amssymb}
\usepackage{amsthm}
\usepackage{mathrsfs}
\usepackage{mathtools}
\usepackage{esint}
\newtheorem{theorem}{Theorem}[section]
\newtheorem{corollary}{Corollary}[section]
\newtheorem{lemma}{Lemma}[section]
\newtheorem{proposition}{Proposition}[section]
\theoremstyle{definition}
\newtheorem{definition}{Definition}[section]
\theoremstyle{remarq}
\newtheorem{remarq}{Remark}[section]
\theoremstyle{remark}

\theoremstyle{nota}

\theoremstyle{notation}
\newtheorem{notation}{Notation}[section]
\theoremstyle{exemple}

\theoremstyle{example}
\newtheorem{example}{Example}[section]
\numberwithin{equation}{section}

\def\XXint#1#2#3{{\setbox0=\hbox{$#1{#2#3}{int}$}  
		\vcenter{\hbox{$#2#3$}}\kern-.5\wd0}}

\usepackage{graphicx}
\usepackage{color}
\usepackage{hyperref}

\begin{document}
	\title[Reiterated $\Sigma$-convergence in Orlicz setting and Applications]{Reiterated $\Sigma$-convergence in Orlicz setting and Applications}
	\author{Dongho J.}
	\address{ Dongho Joseph, University of Maroua, Department of Mathematics and Computer Science, P.O. box 814, Maroua, Cameroon}
	\email{joseph.dongho@fs.univ-maroua.cm}
	\author{Fotso T.J.}
	\address{ Fotso Tachago Joel, University of Bamenda, 	Higher Teachers Training College, Department of Mathematics
		P.O. Box 39
		Bambili, Cameroon}
	\email{fotsotachago@yahoo.fr}
	\author{Nnang H.}
	\address{ Nnang Hubert, University of Yaounde I,
		Higher Teachers Trainning College, Department of Mathematics
		P.O. Box 47
		Yaounde, Cameroon}
	\email{hnnang@yahoo.fr}
	\author{Tchinda T.F.A}
	\address{Tchinda Takougoum Franck Arnold, University of Maroua, Department of Mathematics and Computer Science, P.O. box 814, Maroua, Cameroon}
	\email{takougoumfranckarnold@gmail.com}
	
	\date{\today}
	
	\begin{abstract}
		The concept of reiterated $\Sigma$-convergence (and more generally of multiscale $\Sigma$-convergence)  is extended to framework of Orlicz-Sobolev spaces, in order to deals with homogenization of multiscales problems in general deterministic setting and whose solutions leads in this type of spaces. This concept relies on the notion of reiterated homogenization supralgebra that we will assumed being ergodic. An application to the deterministic reiterated  homogenization of nonlinear degenerate elliptic operators with nonstandard growth is given and some concrete homogenization problems following varied structure hypothesis are deduce from this latter.
	\end{abstract}
	
	\maketitle

	
	\section{Introduction} \label{labelintro} 
	
	\noindent  This paper is devoted to the deterministic reiterated homogenization theory in the framework of Orlicz-Sobolev spaces, and more precisely to the reiterated $\Sigma$-convergence method (see, e.g., \cite{nguetseng2003homogenization}) in this type of spaces. 
	
	For the application, we are interested in the limiting behaviour (as $0<\varepsilon \rightarrow 0$) of the sequence of solutions to the following problem:%
	\begin{equation}\label{1.1}
		\left\{ \begin{array}{l}
			\displaystyle 	-{\rm div}\left[ a\left( \frac{x}{\varepsilon },\frac{x}{\varepsilon ^{2}}%
			,u_{\varepsilon },Du_{\varepsilon }\right) \right] =f\text{ \ in }\Omega,
			\\ 
			
			\\
			u_{\varepsilon }\in W_{0}^{1}L^{\Phi }\left( \Omega
			\right) ,
		\end{array} \right.
	\end{equation}%
	where $\varepsilon$ is a small positive parameter.
	
	\subsection{Our hypotheses}\label{hypoproblem}
	
	For the notations and definitions of functions spaces, we refer to Subsection \ref{labelsub1sect2}. Let us specify data in \eqref{1.1}:
	\begin{itemize}
		\item  $\Omega $ a regular bounded open set in $
		\mathbb{R}^{d},d\geq 2,$ $D$ and ${\rm div}$ denoting gradient and divergence operators, respectively.
		\item $\Phi$ is a $N$-function of class   $\Delta_{2}\cap \Delta'$, such that its complementary  $\widetilde{\Phi}$ is also of class $\Delta_{2}\cap \Delta'$. 
		\item $f\in L^{d}\left( \Omega \right) \cap L^{\widetilde{\Phi }}\left( \Omega
		\right) $ is a scalar function and for each $\varepsilon>0$, $u_{\varepsilon }\in W_{0}^{1}L^{\Phi }(\Omega)$ is the unknown scalar function.
		\item $L^{d}(\Omega)$ is a classical Lebesgue space, $L^{\widetilde{\Phi}}(\Omega)$ and $W^{1}_{0}L^{\Phi}(\Omega)$ are the Orlicz space and Orlicz-Sobolev space, respectively.
		\item  $a:=(a_i)_{1\leq i\leq d}:
		\mathbb{R}
		^{d}\times 
		\mathbb{R}
		^{d}\times 
		\mathbb{R}
		\times 
		\mathbb{R}^{d}\rightarrow 
		\mathbb{R}^{d}$ is a vector valued function satisfying the following properties: 
		\item[\textbf{(H$_{\mathbf{1}}$)}] \textit{(Carath\'{e}odory function hypothesis)}  For all $\left( \zeta ,\lambda \right) \in 
		\mathbb{R}
		\times 
		\mathbb{R}
		^{d},$ the function $\left( y,z\right) \longrightarrow a\left( y,z,\zeta
		,\lambda \right) $ from $
		\mathbb{R}
		^{d}\times 
		\mathbb{R}
		^{d}$ into $
		\mathbb{R}^{d}$ is of Carath\'{e}odory type, that is:
		\begin{itemize}
			\item[(i)] 	For each $z\in 
			\mathbb{R}^{d},$ the function $y\longrightarrow a\left( y,z,\zeta ,\lambda \right) $
			is measurable from $%
			\mathbb{R}^{d}$ to $
			\mathbb{R}^{d}$.
			\item [(ii)]  For almost all $y\in 
			\mathbb{R}
			^{d},$ the function $z\longrightarrow a\left( y,z,\zeta ,\lambda \right) $ is
			continuous from $\mathbb{R}^{d}$ to $\mathbb{R}
			^{d}$ with $a\left( \cdot, \cdot,0,\omega \right) \in L^{\infty }\left( 
			\mathbb{R}_{y}^{d}\times 
			\mathbb{R}
			_{z}^{d}\right) ,\omega $ being the origin in $
			\mathbb{R}^{d}.$
		\end{itemize}
		\item[\textbf{(H$_{\mathbf{2}}$)}] \textit{(Nonstandard growth hypothesis)} There are $N-$functions $\Phi ,\Psi :\left[ 0,+\infty
		\right[ \rightarrow \left[ 0,+\infty \right[ ,\Phi ,\Psi $ being twice
		continuously differentiable with%
		\begin{equation}
			1<\rho _{0}\leq \frac{t\psi \left( t\right) }{\Psi \left( t\right) }\leq
			\rho _{1}\leq \frac{t\phi \left( t\right) }{\Phi \left( t\right) }\leq \rho
			_{2} \; \text{ for all }t>0,  \label{1.2}
		\end{equation}%
		where  $\rho _{0},\rho _{1},\rho _{2}$ are constants and $\Phi ,\Psi $
		are odd, increasing homeomorphisms from $
		\mathbb{R}
		$ to $
		\mathbb{R}
		$ such that $\Phi \left( t\right) =\int_{0}^{t}\phi \left( s\right) ds$ and $%
		\Psi \left( t\right) =\int_{0}^{t}\psi \left( s\right) ds\left( t\geq
		0\right) $. Moreover, there exist $c_{1},c_{3}>\frac{1}{2}$ and $c_{2},c_{4}>0$, $%
		\Phi $ dominates $\Psi $ globally (in symbols $\Phi \prec \Psi$) and 
		\begin{equation}\label{1.3}
			\begin{array}{rcl}
				\left\vert a\left( y,z,\zeta ,\lambda \right) -a\left( y,z,\zeta',\lambda'\right) \right\vert & \leq &  c_{1}\widetilde{\Psi }%
				^{-1}\left( \Phi \left( c_{2}\left\vert \zeta -\zeta ^{\prime }\right\vert
				\right) \right)   \\
				&  &  + c_{3}\widetilde{\Phi }^{-1}\left( \Phi \left(
				c_{4}\left\vert \lambda -\lambda' \right\vert \right) \right)  
			\end{array}
		\end{equation}%
		for a.e. $y\in
		\mathbb{R}
		^{d}$ and for all $\left( z,\zeta ,\lambda \right) \in \mathbb{R}
		^{d}\times 
		\mathbb{R}
		\times
		\mathbb{R}
		^{d},$ where $\widetilde{\Phi }\left( t\right) =\int_{0}^{t}\phi ^{-1}\left(
		s\right) ds$ and $\widetilde{\Psi }\left( t\right) =\int_{0}^{t}\psi
		^{-1}\left( s\right) ds\left( t\geq 0\right) $ are the complementary $N$-functions of $\Phi$ and $\Psi$, respectively.
		
		\item[\textbf{(H$_{\mathbf{3}}$)}] \textit{(Degenerate coercivity hypothesis)} There exists a continuous monotone decreasing mapping  
		$h:\left[ 0,+\infty \right[ \rightarrow \left[ 0,1\right[ ,$ with $\underset{%
			t\geq 0}{\min }h\left( t\right) >0$ and unbounded anti-derivative such that
		for any $\left( \zeta ,\lambda \right) \in 
		\mathbb{R}
		\times
		\mathbb{R}^d,$
		\begin{equation}
			a\left( y,z,\zeta ,\lambda \right) \cdot \lambda \geq \widetilde{\Phi }%
			^{-1}\left( \Phi \left( h\left( \left\vert \zeta \right\vert \right)
			\right) \right) \cdot\Phi \left( \left\vert \lambda \right\vert \right) \text{
				a.e. }\left( y,z\right) \text{ in }%
			\mathbb{R}^{d}\times 
			\mathbb{R}^{d}.  \label{1.4}
		\end{equation}%
		\item[\textbf{(H$_{\mathbf{4}}$)}] \textit{(Positivity hypothesis)} For all $\zeta  \in 
		\mathbb{R}
		$ and for all $\lambda, \lambda' \in
		\mathbb{R}
		^{d},$ 
		\begin{align*}\left( a\left( y,z,\zeta ,\lambda \right) -a\left( y,z,\zeta,\lambda'\right) \right) \cdot \left( \lambda -\lambda'\right)  >  0  \hbox{ for a.e. } (y,z)\in
			\mathbb{R}^{d}\times 
			\mathbb{R}^{d}.
		\end{align*}
		\item[\textbf{(H$_{\mathbf{5}}$)}] \textit{(Local continuity hypothesis)} 
		The function $a$ satisfies a local continuity assumption in the first variable, i.e.  
		for each bounded set $\Lambda $ in $
		\mathbb{R}
		^{d}$ and $\eta >0,$ there exists $\rho >0$ such  
		that,  
		\begin{align*} \hbox{ if }\left\vert \xi \right\vert \leq \rho 
			\hbox{ then }\left\vert a\left(
			y-\xi ,z,\zeta ,\lambda \right) -a\left( y,z,\zeta ,\lambda \right)
			\right\vert \leq \eta, 
		\end{align*}
		for all 
		$\left( z,\zeta ,\lambda \right) \in 
		\mathbb{R}^d\times\mathbb{R}\times \mathbb{R}^{d}$ and almost all $y\in 
		\Lambda.
		$
		
	\end{itemize}
	
	\begin{remarq} \textup{
			Apart from the above assumptions on the function $a$, we will add what is called  \textit{abstract structure hypothesis} (see \eqref{tj} in Section \ref{labelsect4}). Recall that, as observed in  \cite{Cle,MR}, \eqref{1.2} guarantee that $\Phi, \Psi$ and their conjugates verify $\Delta_2$-condition. } 
	\end{remarq}
	\color{black}
	
	\subsection{Litterature review}
	
	It is well known that from a physical point of view, equations of the form \eqref{1.1} arise, e.g., from the modeling of electro-rheological fluids (see \cite{ruzi}), the problem of image recovery (see \cite{chen}), or the modeling of non-Newtonian fluids with termo-convective effects (see \cite{anton}). 
	From a mathematical point of view, the reiterated homogenization of partial differential equations of form \eqref{1.1} has been studied by many authors, under different type of function spaces and different structure hypothesis on function $a$. Thus, in \cite{bensoussan1} the reiterated homogenization in classical Lebesgue spaces was first introduced for the study of asymptotic analysis fo periodic structure. In \cite{allair3}, the multiscale convergence in Lebesgue spaces was first introduced permitting to extend the two-scale convergence (see \cite{allair1,nguet1}) for homogenization problems depending of separated multiple scales. It was applied to reiterated homogenization of Neumann problem in perforated domains. In \cite{luka0,luka1,woukeng2010homogenization}, the (reiterated) homogenization of nonlinear (monotone) operators  in Lebesgue spaces  was first studied for usual periodic setting. Latter in \cite{nguetseng2003homogenization}, the concepts of homogenization structure (or homogenization algebra) and  the $\Sigma$-convergence   was first developed and turns out to be exactly the right tool that is needed to systematically extend homogenization theory beyond the classical periodic setting. This permits to work out various outstanding nonperiodic homogenization problems that were out of reach till then for lack of an appropriate mathematical framework: we will then speak of deterministic homogenization. Thus, in \cite{nanguet1} the deterministic homogenization in Lebesgue spaces of nonlinear monotone operators is studied. In \cite{luka}, the deterministic reiterated homogenization was first developed in classical Lebesgue spaces and applied to the nonlinear degenerate monotone operators of form \eqref{1.1}; while in \cite{gabri} the same concept is extend to multiscale setting with the notion of ergodic homogenization supralgebras  and is used for the study of asymptotic analysis of  nonlinear hyperbolic equations. \\
	
	However, it should be noted that all the homogenization problems and methods mentioned above were carried out in the framework of classical Lebesgue spaces. But, many problems involving partial differential equations have no solutions in these spaces (see \cite{chen,Cle,MR,Y,mignon2}), hence the necessity to develop homogenization methods in others type of function spaces such as Orlicz spaces (see \cite{adam}) and exponent variable spaces (see \cite{zhikov2019homogenization}) which are more general since they recovered the classical Lebesgue spaces.
	
	Thus, in \cite{tacha1} the periodic homogenization was first considered in Orlicz spaces and their applications to the homogenization of PDEs and integral functionals can been found in \cite{tacha5,tacha6,martin,ttchin1}. In \cite{ttchin4,tacha2,tacha4}, the reiterated (and multiscale) homogenization in Orlicz setting of some integral functionals and nonlinear operators of form \eqref{1.1} is studied for usual periodic setting. In \cite{nnang2014deterministic}, the deterministic homogenization in Orlicz setting of nonlinear elliptic operators of form \eqref{1.1} is studied. For periodic homogenization in variable exponent spaces, we can refer to \cite{amaz1,mmr,zhikov2019homogenization}.

	\subsection{Problem statement and objectives}
	
	Provided the differential operator $u_{\varepsilon} \rightarrow - \text{div }a^{\varepsilon}(\cdot,\cdot, u_{\varepsilon}, Du_{\varepsilon})$, $u_{\varepsilon} \in W^{1}_{0}L^{\Phi}(\Omega)$, (where for each $\varepsilon>0$, $a^{\varepsilon}(\cdot,\cdot, u_{\varepsilon}, Du_{\varepsilon})$ stands for the function $x\to a(\frac{x}{\varepsilon}, \frac{x}{\varepsilon^{2}}, u_{\varepsilon}, Du_{\varepsilon})$ from $\Omega$ to $\mathbb{R}^{d}$) is well defined and has suitable properties (see Proposition \ref{propalphabeta}), it is classical matter to prove an existence and uniqueness result for \eqref{1.1} (see, e.g., \cite{Y}). 
	
	Thus, we have a generalized sequence $(u_{\varepsilon})_{\varepsilon>0}$ at our disposal, and the main objective is to study, under a suitable condition on $a(y,z, \zeta,\lambda)$ (for fixed $(\zeta,\lambda)$) so-called a \textit{abstract structure hypothesis}, the limiting behaviour of $u_{\varepsilon}$ (the solution of \eqref{1.1}) as $\varepsilon \to 0$  and then several examples considered in various concrete settings are presented by way of illustration. This lies within the class of so-called \textit{deterministic reiterated homogenization problems}. Recall that this \textit{abstract structure hypothesis} cover a great number of concrete behavious including the classical periodicity hypothesis (see Example \ref{cb}). But, first of all we have need to extend so-called \textit{reiterated $\Sigma$-convergence} (see, e.g, \cite{luka,gabri}) to the framework of Orlicz-Sobolev spaces.
	
	\subsection{Main results}
	
	In the present paper, a first novelty is concerned the generalization of the compactness results of reiterated $\Sigma$-convergence in classical Sobolev spaces to a class of Orlicz-Sobolev spaces.
	At the same time, the solution of the  homogenization problem \eqref{1.1} under consideration seems to be more general than the case considered in \cite{luka}.
	Thus, considering notations in Section \ref{labelsect2} and \ref{labelsect3}, we extend the result in \cite[Theorem 3.5]{gabri} to Orliz-Sobolev spaces as follows:
	
	\begin{theorem}\label{ti}
		Let $\Omega$ be an open subset in $\mathbb{R}^{d}$. Let $\Phi$ be an $N$-function of class $\Delta_{2}$ and let $A= A_{y}\odot A_{z}$ be an $RH$-supralgebra where $A_{y}$ (resp. $A_{z}$) is an ergodic $H$-supralgebra on $\mathbb{R}^{d}_{y}$ (resp. $\mathbb{R}^{d}_{z}$). Assume that  $A_{y}$ and $A_{z}$ are translation invariant, and moreover that their elements are uniformly continuous. Finally, let $(u_{\varepsilon})_{\varepsilon\in E}$ be a bounded sequence in $W^{1}_{0}L^{\Phi}(\Omega)$. There exist a subsequence $E^{\prime}$ from $E$ and a triple $ (u_{0}, u_{1}, u_{2}) \in W^{1}_{0}L^{\Phi}(\Omega) \times L^{1}(\Omega; W^{1}_{\#}\mathcal{X}^{\Phi}_{A_{y}})\times L^{1}(\Omega; \mathcal{X}^{1}_{A_{y}}(\mathbb{R}^{d}_{y} ; W^{1}_{\#}\mathcal{X}^{\Phi}_{A_{z}}))$ such that, as $E^{\prime} \ni \varepsilon \to 0$,
		\begin{equation}\label{pa}
			u_{\varepsilon} \rightharpoonup u_{0} \quad in \;\, W^{1}_{0}L^{\Phi}(\Omega)\textup{-}weak 
		\end{equation}
		and 
		\begin{equation}\label{pb}
			\dfrac{\partial u_{\varepsilon}}{\partial x_{j}} \rightharpoonup \dfrac{\partial u_{0}}{\partial x_{j}} + \dfrac{\overline{\partial} u_{1}}{\partial y_{j}} + \dfrac{\overline{\partial} u_{2}}{\partial z_{j}} \quad in \;\, L^{\Phi}(\Omega)\textup{-}weak \; R\, \Sigma \quad (1 \leq j \leq d).
		\end{equation}
		If in addition $\widetilde{\Phi} \in \Delta^{\prime}$ then $u_{1} \in L^{\Phi}(\Omega; W^{1}_{\#}\mathcal{X}^{\Phi}_{A_{y}})$ and $u_{2} \in  L^{\Phi}(\Omega; \mathcal{X}^{\Phi}_{A_{y}}(\mathbb{R}^{d}_{y} ; W^{1}_{\#}\mathcal{X}^{\Phi}_{A_{z}}))$.
	\end{theorem} 
	\begin{remarq}
		\textup{	It should be note that, we have generalized Theorem \ref{ti} to the multiscale setting (see Proposition \ref{multiscal}) which then extend to Orlicz-Sobolev spaces result in \cite[Theorem 3.5]{gabri}. }
	\end{remarq}
	
	Our second (main) and third results in this paper are about the deterministic reiterated homogenization problem \eqref{1.1}.
	However, in order to get uniqueness of the  solutions in \eqref{1.1} and \eqref{3.60},  in the same spirit of \cite[(2.3.40)]{pankov} one can assume that there exists $c_5>0$ such that
	\begin{itemize}
		\item[\textbf{(H$_{6}$)}]  \textit{(strict monotonicity hypothesis)}  For all $\zeta, \zeta' \in 
		\mathbb{R}
		$ and for all $\lambda, \lambda' \in
		\mathbb{R}
		^{d},$
		\begin{align*}\left( a\left( y,z,\zeta ,\lambda \right) -a\left( y,z,\zeta',\lambda'\right) \right) \cdot \left( \lambda -\lambda'\right)  > c_5\Phi \left( \left\vert \lambda -\lambda' \right\vert
			\right) 
		\end{align*} a.e in $\left( y,z\right) $ in $
		\mathbb{R}^{d}\times 
		\mathbb{R}^{d}$. 
	\end{itemize}
	Thus, considering notations in Section \ref{labelsect3} and \ref{labelsect4}, we extend the results in \cite[Theorem 1.1 and Theorem 1.2]{tacha4}  to the general deterministic setting as follows:
	\begin{theorem}\label{mainresult}
		Let \eqref{1.1} be the problem under hypotheses in Subsection \ref{hypoproblem}, with $a$ and $f$ satisfying \textbf{(H$_{1}$)}-\textbf{(H$_{5}$)}. Suppose that $A=A_{y}\odot A_{z}$ is an $RH$-supralgebra on $\mathbb{R}^{d}_{y}\times \mathbb{R}^{d}_{z}$ such that the \textit{abstract structure hypothesis} \textbf{(H$_{7}$)} holds.
		We also assume that $A_{y}$ (resp. $A_{z}$) is ergodic, translation invariant, and moreover that their elements are uniformly continuous.
		
		For each $\varepsilon
		>0$, let $u_{\varepsilon }$ be a solution of \eqref{1.1}. Then there exists a not relabeled subsequence  and $\mathbf{u}:=\left( u_{0},u_{1},u_{2}\right)\in \mathbb{F}_{0}^{1,\Phi }:=  W^{1}_{0}L^{\Phi}(\Omega) \times L^{\Phi}(\Omega; W^{1}_{\#}\mathcal{X}^{\Phi}_{A_{y}})\times L^{\Phi}(\Omega; \mathcal{X}^{\Phi}_{A_{y}}(\mathbb{R}^{d}_{y} ; W^{1}_{\#}\mathcal{X}^{\Phi}_{A_{z}}))$ such that
		\begin{equation}\label{3.58}
			u_{\varepsilon} \rightharpoonup u_{0} \quad in \;\, W^{1}_{0}L^{\Phi}(\Omega)\textup{-}weak 
		\end{equation}
		\begin{equation}\label{3.59}
			\dfrac{\partial u_{\varepsilon}}{\partial x_{j}} \rightharpoonup \dfrac{\partial u_{0}}{\partial x_{j}} + \dfrac{\overline{\partial} u_{1}}{\partial y_{j}} + \dfrac{\overline{\partial} u_{2}}{\partial z_{j}} \quad in \;\, L^{\Phi}(\Omega)\textup{-}weak \; R\, \Sigma \quad (1 \leq j \leq d).
		\end{equation}
		and $\mathbf{u}$ 	solves the variational homogenized problem
		\begin{equation}
			\left\{ 
			\begin{tabular}{l}
				$\displaystyle \int_{\Omega } \pounds_{A}\iint_{\mathbb{R}^{2d}} b\left( u_{0},Du_{0}+ \overline{D}_{y} u
				_{1} +\overline{D}_{z} u
				_{2} \right)\cdot \left(Dv_{0}+ \overline{D}_{y} v
				_{1} +\overline{D}_{z} v
				_{2} \right) dx dy dz
				$ \\
				\\ 
				$\displaystyle  =\int_{\Omega }fv_{0}dx,$ \quad for all $v=\left( v_{0},v_{1},v_{2}\right) \in 
				\mathbb{F}_{0}^{1,\Phi }$. \\ 
			\end{tabular}%
			\right. . \label{3.60}
		\end{equation}
		where $b=(b_{i})$, $1\leq i\leq d$ is defined as in \eqref{tp}.
	\end{theorem}
	
	\begin{theorem}\label{maincor} 
		Under hypothesis of Theorem \ref{mainresult}, for every $\varepsilon >0$, let \eqref{1.1} be  such that $a$ and $f$ satisfy \textbf{(H$_{1}$)}-\textbf{(H$_{6}$)}. 
		Let $u_0 \in W_{0}^{1}L^{\Phi }(\Omega)$ be the solution  defined by means of \eqref{3.60}. Then, it is the unique solution of the
		macroscopic homogenized problem 
		\begin{equation*}
			\left\{ \begin{array}{l}
				-{\rm div}\, q\left(u_0, Du_{0}\right) =f\text{ in }\Omega,  \\
				
				\\
				u_{0}\in W_{0}^{1}L^{\Phi
				}\left( \Omega \right),  
			\end{array}\right.
		\end{equation*}
		where $q$ is defined as follows.
		For $(r,\xi) \in \mathbb R\times \mathbb R^d$
		\begin{align}
			\label{q}
			q\left( r,\xi \right) =\pounds_{A_{y}}\int_{\mathbb{R}^{d}} h\left(y,r, \xi +\overline{D}_{y}\pi
			_{1}\left( r,\xi \right) \right) dy,
		\end{align}  
		where, for a.e. $y \in \mathbb{R}^{d}_{y}$, and for any $(\alpha, \xi) \in \mathbb R \times \mathbb R^d$,  
		\begin{equation}\label{h}
			h\left(y,r, \xi \right) :=\pounds_{A_{z}}\int_{\mathbb{R}^{d}} b_{i}\left( r ,\xi
			+\overline{D}_{z}\pi _{2}\left(y,r, \xi \right) \right) dz,
		\end{equation}
		
		where for a.e.  $y \in \mathbb{R}^{d}_{y}$, and every $(\alpha,\xi) \in \mathbb R \times \mathbb R^d$, $\pi_2(s, \alpha,\xi)$,  is the solution of the following variational cell problem:
		\begin{equation} \label{3.67}
			\left\{ 
			\begin{tabular}{l}
				$\displaystyle \hbox{find } \pi _{2}\left(y,r, \xi \right) \in W_{\#}^{1}\mathcal{X}^{\Phi}_{A_{z}} $ \hbox{such that} \\ 
				$\displaystyle \pounds_{A_{z}} \int_{\mathbb{R}^{d}} b\left( r,\xi +\overline{D}_{z}\pi _{2}\left( y,r,\xi \right) \right)
				\cdot \overline{D}_{z}\theta dz=0$ for all $\theta \in W_{\#}^{1}\mathcal{X}^{\Phi}_{A_{z}}.$%
			\end{tabular}%
			\right. 
		\end{equation}
		and $\pi _{1}\left(\alpha, \xi \right) \in
		W_{\#}^{1}\mathcal{X}^\Phi_{A_{y}} $ is the unique solution of the variational problem 
		\begin{equation}\label{3.69b}
			\left\{ 
			\begin{tabular}{l}
				$\displaystyle	\hbox{ find } \pi _{1}\left(r, \xi \right) \in W_{\#}^{1}\mathcal{X}^{\Phi}_{A_{y}} $ \hbox{ such that } \\ 
				$\displaystyle \pounds_{A_{y}} \int_{\mathbb{R}^{d}} h\left(y, r, \xi +\overline{D}_{y}\pi _{1}\left(r, \xi \right) \right) \cdot \overline{D}_{y}\theta
				dy=0$ for all $\theta \in W_{\#}^{1}\mathcal{X}^{\Phi}_{A_{y}}. $%
			\end{tabular}%
			\right.  
		\end{equation}
	\end{theorem}

	\begin{remarq}
		\textup{	It should be noted that in this study, we investigate the homogenization of \eqref{1.1} not  under the periodicity hypothesis as in the \cite{tacha4}, but in a general deterministic setting including the periodicity, almost periodicity, weakly almost periodicity, convergence at infinity hypotheses and more others (see Section \ref{labelsect5} for more details). }
	\end{remarq}

	\subsection{Organization of paper}
	
	The paper is divided into sections each revolving around a specific aspect: Section \ref{labelsect2} dwells on prelimaries about  Orlicz-Sobolev spaces and homogenization supralgebra.  Section \ref{labelsect3} is devoted to the extension of concept of reiterated $\Sigma$-convergence to  Orlicz-Sobolev spaces. In Section \ref{labelsect4} we study the deterministic reiterated homogenization   problem \eqref{1.1}. The periodicity hypothesis stated in \cite{tacha4} is here replaced by an \textit{abstract structure hypothesis}. Section \ref{labelsect5} is concerned with a few concrete examples of homogenization problem \eqref{1.1} under various concrete structure and we show how each of them can reduce to the \textit{abstract structure hypothesis} in Section \ref{labelsect4}.   Finally in Section \ref{labelappendix} we give in \textbf{Appendix} for the reader's convenience, the results about the traces $a\left(x/\varepsilon, x/\varepsilon^{2}, v, \mathbf{v}\right)$ $(x\in \Omega)$ and justify the well-posedness of problem \eqref{1.1}. 

	
	\section{Prelimaries on  Orlicz-Sobolev spaces and reiterated homogenization supralgebra} \label{labelsect2} 
	
	In what follows, except otherwise stated, all the vector spaces are considered over the complex field, $\mathbb{C}$, and scalar functions assume complex values.
	
	$X$ and	$V$ denote a locally compact space and a Banach space, respectively, and
	$\mathcal C(X; V)$ stands for the space of continuous functions from $X$ into $V$, and
	$\mathcal B(X; V)$  stands for those functions in $\mathcal C(X; V)$ that are bounded.
	
	The space $\mathcal B(X; V)$ is endowed with the supremum norm $\|u\|_{\infty} = \sup_{x\in X}
	\|u(x)\|$, where
	$\|\cdot\|$ denotes the norm in $V$, (in particular, given an open set $A\subset \mathbb R^d$ by $\mathcal B(A)$ we denote the space of real valued continuous and bounded functions defined in $A$).
	
	Likewise the spaces $L^p(X; V)$ (resp. $L^\Phi(X; V)$) and $L^p_{\rm loc}(X; V)$ (resp. $L^\Phi_{\rm loc}(X; V)$)
	($1\leq p \leq \infty$, $\Phi$ an $N$-function and $X$ provided with a positive Radon measure) are denoted by $L^p(X)$ and
	$L^p_{\rm loc}(X)$, respectively, when $V = \mathbb C$. 
	
	In the sequel we denote by $Y$ and $Z$ two identical copies of the cube $(-1/2,1/2)^d$.  
	In order to enlighten the space variable under consideration we will adopt the notation $\mathbb R^d_x, \mathbb R^d_y$, or $\mathbb R^d_z$ to indicate where $x,y $ or $z$ belong to.
	
	The family of open subsets in $\mathbb R^d_x$ will be denoted by $\mathcal A(\mathbb R^d_x)$.
	For any subset $E$ of $\mathbb R^m$, $m \in \mathbb N$, by $\overline E$, we denote its closure in the relative topology and by $|E|$, we denote its Lebesgue measure.

	
	
	In this section, we recall some basic results about Orlicz-Sobolev spaces and homogenization supralgebras.
	
	\subsection{Orlicz-Sobolev spaces on bounded open sets} \label{labelsub1sect2} 
	
	All definitions and results recalled here are classical and can be found in \cite{adam,adams}. 
	
	\subsubsection{$N$-functions}\label{labelsubsub1sect2} 
	
	Let $\Phi : [0,+\infty) \rightarrow  [0,+\infty)$ be a $N$-function, that is, 
	$\Phi$ is continuous, convex ,
	$\Phi(t) > 0$ for $t > 0$,
	$\frac{\Phi(t)}{t} \to 0$ as $t\to 0$ and $\frac{\Phi(t)}{t} \to \infty$ as $t\to \infty$. Then
	$\Phi$ has an integral representation under form
	$
	\Phi(t) = \int_{0}^{t} \phi(\tau)\, d\tau,
	$
	where $\phi : [0, \infty) \rightarrow  [0, \infty)$ is nondecreasing, right continuous, with $\phi(0) = 0$, $\phi(t) > 0$ if $t>0$ and $\phi(t)\rightarrow \infty$ if $t\rightarrow \infty$. We denote $\widetilde{\Phi}$
	the Fenchel's conjugate (or complementary function) of the $N$-function $\Phi$, that is, $\widetilde{\Phi}(t) = \sup_{s\geq 0} \left( st - \Phi(s) \right) \; (t\geq 0)$, then $\widetilde{\Phi}$ is also a $N$-function. Given two $N$-functions $\Phi$ and $\Psi$, we say that $\Phi$ dominates $\Psi$ (denoted by $\Phi \succ \Psi$ or $\Psi \prec \Phi$) near infinity if there are $k>1$ and $t_{0}>0$ such that 
	\begin{equation*}
		\Psi(t) \leq \Phi(kt), \quad \forall t> t_{0}.
	\end{equation*}
	With this in mind, it is well known that if $\displaystyle \lim_{t\to +\infty} \Phi(t) / \Psi(t) = +\infty$ then $\Phi$ dominates $\Psi$ near infinity. Let us recall some important properties about $N$-functions.  
	Let $\Phi$ be a $N$-function. $\Phi$ is said to satisfy the $\Delta_{2}$-condition or $\Phi$ belongs to the class $\Delta_{2}$ at $\infty$, which is written as  $\Phi \in \Delta_{2}$, if there exist constants  $t_{0} > 0$, $k > 2$ such that  
	\begin{equation*}
		\Phi(2t) \leq k \Phi(t),
	\end{equation*}
	for all $t \geq t_{0}$.
	$\Phi$ is said to satisfy the $\Delta'$-condition (or $\Phi$ belongs to the class $\Delta'$) denoted by $\Phi \in \Delta'$ if there exists $\beta > 0$ such that  
	\begin{equation*} 
		\Phi(ts) \leq \beta\, \Phi(t)\Phi(s), \quad \forall t,s \geq 0.
	\end{equation*}
	Let $t\rightarrow \Phi(t) = \int_{0}^{t} \phi(\tau) d\tau$ be a $N$-function, and let $\widetilde{\Phi}$ its complementary. Then one has
	\begin{equation}\label{lem10} 
		\left\{ \begin{array}{l}
			\dfrac{t\,\phi(t)}{\Phi(t)} \geq 1 \quad (\textup{resp.} \, > 1 \; \textup{if} \, \phi \, \textup{is} \, \textup{strictly} \; \textup{increasing})  \\
			
			\\
			\widetilde{\Phi}(\phi(t)) \leq t\,\phi(t) \leq \Phi(2t)
		\end{array}\right.
	\end{equation}
	for all $t >0$.
	
	We give now some examples of $N$-functions.
	\begin{example}
		\textup{	The function $t \rightarrow \frac{t^{p}}{p}$, ($p > 1$) is a $N$-function which satisfy $\Delta_{2}\cap\Delta'$-condition and property (\ref{lem10}). Its complementary is the $N$-function $t \rightarrow \frac{t^{q}}{q}$, where $\frac{1}{p} + \frac{1}{q} = 1$. The function $t \rightarrow t^{p}\ln(1+t)$, ($p \geq 1$) is a $N$-function that satisfies $\Delta_{2}\cap\Delta'$-condition and property (\ref{lem10}), while the $N$-functions $t \rightarrow t^{\ln t}$ and $t \rightarrow e^{t^{r}} - 1$, ($r >0$) are not of class $\Delta_{2}$. However, for $r=1$, the $N$-function $t \rightarrow e^{t} - 1$ satisfy $\Delta'$-condition.   
		}
	\end{example}
	
	\subsubsection{Basic notions on Orlicz-Sobolev spaces}\label{labelsubsub2sect2} 
	
	Let $\Omega$ be a bounded open set in $\mathbb{R}^{d}$ (integer $d\geq 1$), and let $\Phi$ be a $N$-function. The Orlicz space $L^{\Phi}(\Omega)$ is defined to be the vector space of all measurable functions $u: \Omega\rightarrow \mathbb{R}$ such that 
	\begin{equation*} 
		\int_{\Omega} \Phi\left(\dfrac{|u(x)|}{\delta} \right)dx < +\infty,
	\end{equation*}
	for some $\delta=\delta(u) >0$. \\
	
	On the space $L^{\Phi}(\Omega)$ we define the Luxemburg norm,
		\begin{equation*}
			\lVert u \rVert_{L^{\Phi}(\Omega)} = 	\lVert u \rVert_{\Phi,\Omega}= \inf \left\{ \delta>0 \; : \; \int_{\Omega} \Phi\left(\dfrac{|u(x)|}{\delta} \right)dx  \leq 1 \right\}, \quad \forall u \in L^{\Phi}(\Omega),
		\end{equation*}
	which makes it a Banach space.  \\
	Now we will give some properties of Orlicz spaces and we will refer to \cite{adam} for more details.
	Assume that $\Phi \in \Delta_{2}$. Then:
	\begin{itemize}
		\item[\textbf{(i)}] $\mathcal{D}(\Omega)$ is dense in $L^{\Phi}(\Omega)$ ;
		\item[\textbf{(ii)}] $L^{\Phi}(\Omega)$ is separable and reflexive whenever $\widetilde{\Phi}\in \Delta_{2}$ ;
		\item[\textbf{(iii)}] the dual of $L^{\Phi}(\Omega)$ is identified with $L^{\widetilde{\Phi}}(\Omega)$, and the dual norm on $L^{\widetilde{\Phi}}(\Omega)$ is equivalent to $\lVert \cdot \rVert_{L^{\widetilde{\Phi}}(\Omega)}$ ;
		\item[\textbf{(iv)}] given $u \in L^{\Phi}(\Omega)$ and $v \in L^{\widetilde{\Phi}}(\Omega)$ the product $uv$ belongs to $L^{1}(\Omega)$ with the generalized H\"{o}lder's inequality 
		\begin{equation*}
			\left| \int_{\Omega} u(x)v(x)dx \right| \leq 2\, \lVert u \rVert_{L^{\Phi}(\Omega)} \, \lVert v \rVert_{L^{\widetilde{\Phi}}(\Omega)} ;
		\end{equation*}
		\item[\textbf{(v)}] given $v \in L^{\Phi}(\Omega)$ the linear functional $L_{v}$ on $L^{\widetilde{\Phi}}(\Omega)$ defined by 
		\begin{equation*}
			L_{v}(u) = \int_{\Omega} u(x)v(x) dx, \quad u \in L^{\widetilde{\Phi}}(\Omega) ;
		\end{equation*} 
		belongs to the dual $[L^{\widetilde{\Phi}}(\Omega)]'$ with $\lVert v \rVert_{L^{\widetilde{\Phi}}(\Omega)} \leq \lVert L_{v} \rVert_{[L^{\widetilde{\Phi}}(\Omega)]'} \leq 2 \lVert v \rVert_{L^{\widetilde{\Phi}}(\Omega)}$ ;
		\item[\textbf{(vi)}] given two $N$-functions $\Phi$ and $\Psi$, we have the continuous embedding \\ $L^{\Phi}(\Omega) \hookrightarrow L^{\Psi}(\Omega)$ if and only if $\Phi \succ \Psi$ near infinity ;
		\item[\textbf{(vii)}] property $\displaystyle \lim_{t\to +\infty} \frac{\Phi(t)}{t}= +\infty$ implies $L^{\Phi}(\Omega) \hookrightarrow L^{1}(\Omega) \hookrightarrow	L^{1}_{\textup{loc}}(\Omega) \hookrightarrow \mathcal{D'}(\Omega)$,
		\item[\textbf{(viii)}] the product space $L^{\Phi}(\Omega)^{d}= L^{\Phi}(\Omega)\times L^{\Phi}(\Omega) \times \cdots \times L^{\Phi}(\Omega)$, ($d$-times), is endowed with the norm 
		\begin{equation*}
			\lVert \textbf{v} \rVert_{L^{\Phi}(\Omega)^{d}} = \sum_{i=1}^{d} \lVert v_{i} \rVert_{L^{\Phi}(\Omega)}, \quad \textbf{v}=(v_{i}) \in L^{\Phi}(\Omega)^{d} ;
		\end{equation*}
		\item[\textbf{(ix)}] \label{property9}  If $\Omega_{1} \subset \mathbb{R}^{d_{1}}$ and $\Omega_{2} \subset \mathbb{R}^{d_{2}}$ are two bounded open sets with $d_{1}+d_{2}=d$, and if $u \in L^{\Phi}(\Omega_{1}\times \Omega_{2})$,
		then for almost all $x_{1}\in \Omega_{1}$, $u(x_{1}, \cdot) \in L^{\Phi}(\Omega_{2})$. If in addition $\widetilde{\Phi} \in \Delta'$ associate with a constant $\beta$, then the function $u$ belongs to $L^{\Phi}(\Omega_{1}, L^{\Phi}(\Omega_{2}))$, with	
		\begin{equation}\label{be1}
			\|u\|_{L^{\Phi}(\Omega_{1}, L^{\Phi}(\Omega_{2}))} \leq \iint_{\Omega_{1}\times \Omega_{2}} \Phi(|u(x_{1},x_{2})|) dx_{1}dx_{2} + \beta.
		\end{equation}
	\end{itemize}

	Analogously to the case of Lebesgue spaces, one can define the Orlicz-Sobolev function space as follows: 
	\begin{equation*}
		W^{1}L^{\Phi}\left( \Omega \right) =\left\{ u\in L^{\Phi}\left( \Omega \right) :%
		\frac{\partial u}{\partial x_{i}}\in L^{\Phi}\left( \Omega \right),1\leq i\leq
		d\right\},
	\end{equation*}
	where derivatives are taken in the distributional sense on $%
	\Omega.$ Endowed with the norm
	\begin{equation*}
		\left\Vert u\right\Vert _{W^{1}L^{\Phi}\left(
			\Omega \right) }=\left\Vert u\right\Vert _{\Phi,\Omega }+\sum_{i=1}^{d}
		\left\Vert \frac{\partial u}{\partial x_{i}}\right\Vert _{\Phi,\Omega }, \quad u\in
		W^{1}L^{\Phi}\left( \Omega \right) ,
	\end{equation*}
	$W^{1}L^{\Phi}\left( \Omega \right) $ is
	a Banach space, reflexive when $\Phi\in \Delta_{2}$ and we have the compact embedding $W^{1}L^{\Phi}\left( \Omega \right) \subset L^{\Phi}\left( \Omega \right)$ when $\partial\Omega$ is Lipschitzian (see, e.g., \cite{adams}). We denote by $W_{0}^{1}L^{\Phi}\left( \Omega \right)
	, $ the closure of $\ \mathcal{D}\left( \Omega \right) $ (space of test functions on $\Omega$) in $%
	W^{1}L^{\Phi}\left( \Omega \right) $ and the semi-norm $u\rightarrow \left\Vert
	u\right\Vert _{W_{0}^{1}L^{\Phi}\left( \Omega \right) }$ defined by 
	\begin{equation*}
		\left\Vert
		u	\right\Vert _{W_{0}^{1}L^{\Phi}\left( \Omega \right) }=\left\Vert
		Du\right\Vert _{\Phi,\Omega }=\sum_{i=1}^{d} \left\Vert \frac{\partial u}{%
			\partial x_{i}}\right\Vert _{\Phi,\Omega }
	\end{equation*}
	is a norm on $W_{0}^{1}L^{\Phi}\left(
	\Omega \right) $ equivalent to $\left\Vert \cdot \right\Vert _{W^{1}L^{\Phi}\left(
		\Omega \right) }.$
	

	\subsection{Orlicz-Sobolev spaces associated to supralgebras} \label{labelsub2sect2} 
	
	\subsubsection{Basic notions on homogenization supralgebras}\label{labelsub2sub1sect2} 
	
	This concept has just been
	defined in a more recent paper \cite{gabri}. It is more general than those (homogenization algebra and algebra with mean value)
	defined in the papers \cite{nguetseng2003homogenization,zhikov1983averaging} because we do not need the algebra
	to be separable (as in \cite{nguetseng2003homogenization}), or to consist of functions that are
	uniformly continuous (as in \cite{zhikov1983averaging}). Before we go any further, we
	need to give some preliminaries. Let $\mathcal{H}=(H_{\varepsilon
	})_{\varepsilon >0}$ be either of the following two actions of $\mathbb{R}_{+}^{\ast }$ (the
	multiplicative group of positive real numbers) on the numerical space $%
	\mathbb{R}^{d}$ ($d=N\text{ or }m$), defined as follows: 
	\begin{equation}
		H_{\varepsilon }(x)=\frac{x}{\varepsilon _{1}}\;\;(x\in \mathbb{R}^{N})
		\label{2.1}
	\end{equation}%
	\begin{equation}
		H_{\varepsilon }(x)=\frac{x}{\varepsilon _{2}}\;\;(x\in \mathbb{R}^{m})
		\label{2.1.1}
	\end{equation}%
	where $\varepsilon _{1}$ and $\varepsilon _{2}$ are two well-separated functions of $\varepsilon $ tending to
	zero with $\varepsilon $, that is, $0< \varepsilon _{1}, \varepsilon _{2}, \varepsilon _{2}/\varepsilon _{1} \rightarrow 0$. For given $\varepsilon >0$, let 
	\begin{equation*}
		u^{\varepsilon }(x)=u(H_{\varepsilon }(x))\;\;(x\in \mathbb{R}%
		^{d}).\;\;\;\;\;\;
	\end{equation*}%
	Note that, for $u\in L_{\text{loc}}^{1}(\mathbb{R}_{y}^{d})$ (as usual, $\mathbb{R}%
	_{y}^{d}$ denotes the numerical space $\mathbb{R}^{d}$ of variables $%
	y=(y_{1},...,y_{d})$), $u^{\varepsilon }$ lies in $L_{\text{loc}}^{1}(%
	\mathbb{R}_{x}^{d})$; see, e.g., in \cite{gabri}. 
	
	However, a function $u\in \mathcal{B}(\mathbb{R}_{y}^{d})$ (the $\mathcal{C}$%
	*-algebra of bounded continuous complex functions on $\mathbb{R}_{y}^{d}$)
	is said to have a mean value for $\mathcal{H}$, if there exists a complex
	number $M(u)$ such that $u^{\varepsilon }\rightarrow M(u)$ in $L^{\infty }(%
	\mathbb{R}_{x}^{d})$-weak $\ast $ as $\varepsilon \rightarrow 0$. The
	complex number $M(u)$ is called the mean value of $u$ (for $\mathcal{H}$).
	
	
	We summarize below a few basic notions and results concerning the homogenization supralgebras. We refer to \cite{gabri,sango2} for further details.
	
	\begin{definition}
		\label{d2.1} By a homogenization supralgebra  (or $H$%
		-supralgebra, in short) on $\mathbb{R}^{d}$ for $\mathcal{H}$ 
		we mean any closed subalgebra of $\mathcal{B}(\mathbb{R}^{d})$
		which contains the constants, is closed under complex conjugation and whose
		elements possess a mean value for $\mathcal{H}$.
	\end{definition}
	
	
	Let $A_{y}$ be an $H$-supralgebra on $\mathbb{R}^{d}_{y}$ (for $\mathcal{H}$).  We denote by $\Delta (A_{y})$ (a subset of the topological dual $A^{\prime }_{y}$ of $A_{y}$) the spectrum
	of $A_{y}$ and by $\mathcal{G}$ the Gelfand transformation on $A_{y}$. We recall that the spectrum $\Delta (A_{y})$ is a compact topological
	space, and the Gelfand transformation $\mathcal{G}$ is an isometric
	isomorphism identifying $A_{y}$ with $\mathcal{C}(\Delta (A_{y}))$ (the continuous
	functions on $\Delta (A_{y})$) as $\mathcal{C}$*-algebras.
	Moreover, there exists a Radon measure $\beta_{y} $ (of
	total mass $1$) in $\Delta (A_{y})$, called the $M$\textit{-measure} for $A_{y}$ 
	such that  
	\begin{equation*}
		M(u)=\int_{\Delta (A_{y})}\mathcal{G}(u)d\beta \;\, \text{\ for }u\in A_{y}.
	\end{equation*}
	\label{p2.0}
	
	We give now some examples of $H$-supralgebras (more precisely of $H$-algebras).
	\begin{example}\cite{wou} 	\label{p2.5}  
		\begin{itemize}
			\item[(1)] \textup{We denote by $\mathcal{C}_{\text{\emph{per}}}(Y)$\  the
				$H$-algebra of $Y$-periodic continuous functions on $\mathbb{R}_{y}^{d}$\ ($Y=(-%
				\frac{1}{2},\frac{1}{2})^{d}$). }
			\item[(2)] \textup{We denote by $AP(\mathbb{R}_{y}^{d})$  the set of all almost periodic continuous
				functions on $\mathbb{R}_{y}^{d}$ defined as the vector space consisting of
				all functions defined on $\mathbb{R}_{y}^{d}$ that are uniformly
				approximated by finite linear combinations of the functions in the set $%
				\{\exp (2i\pi k\cdot y):k\in \mathbb{R}^{d}\}$. Then $AP(\mathbb{R}_{y}^{d})$ is an $H$-algebra.}
			\item[(3)] \textup{ Let the $H$-algebra $\mathcal{B}_{\infty}(\mathbb{R}^{d}_{y})$, denoting the space of all function $u\in \mathcal{C}(\mathbb{R}^{d}_{y})$ with $\lim_{|y|\to +\infty} u(y)= \zeta \in \mathbb{C}$ ($\zeta$ depending on $u$), where $|y|$ denotes the euclidean norm of $y$ in $\mathbb{R}^{d}_{y}$.  }
			\item[(4)] \textup{ Let $\mathcal{B}_{\infty,\mathbb{Z}^{d}}(\mathbb{R}^{d}_{y})$ denote the closure in $\mathcal{B}(\mathbb{R}^{d}_{y})$ of the space of all finite sums $\sum_{\text{finite}} \varphi_{i}u_{i}$ with $\varphi_{i} \in \mathcal{B}_{\infty}(\mathbb{R}^{d}_{y})$, $u_{i} \in \mathcal{C}_{per}(Y)$.  Then  $\mathcal{B}_{\infty,\mathbb{Z}^{d}}(\mathbb{R}^{d}_{y})$ is an $H$-algebra.}
			\item[(5)] \textup{ We denote by $WAP(\mathbb{R}_{y}^{d})$  the set of all weakly almost periodic continuous
				functions on $\mathbb{R}_{y}^{d}$ defined as the vector space consisting of
				all functions $u$ defined on $\mathbb{R}_{y}^{d}$ such that the set of translates $\{ \tau_{b}\, : \, b \in \mathbb{R}_{y}^{d}\}$ is relatively weakly compact in $\mathcal{B}(\mathbb{R}_{y}^{d})$. Then $WAP(\mathbb{R}_{y}^{d})$ is an $H$-algebra.}
			\item[(6)] \textup{ We denote by $FS(\mathbb{R}^{d}_{y}) $ the Fourier-Stieltjes algebra on $\mathbb{R}^{d}_{y}$ is defined as the closure in $\mathcal{B}(\mathbb{R}^{d}_{y})$ of the space }
			\begin{eqnarray*}
				FS_{\ast}(\mathbb{R}^{d}_{y}) = \left\{ f: \mathbb{R}^{d}_{y} \to \mathbb{R}, \;\; f(x)= \int_{\mathbb{R}^{d}_{y}} \exp(ix\cdot y) d\nu(y) \right. \nonumber \\ 
				\textup{ for some }\;\; \nu \in \mathcal{M}_{\ast}(\mathbb{R}^{d}_{y}) \bigg\},
			\end{eqnarray*}
			\textup{	where $\mathcal{M}_{\ast}(\mathbb{R}^{d}_{y})$ denotes the space of complex valued measures $\nu$ with finite total variation: $|\nu|(\mathbb{R}^{d}_{y}) < \infty$. }
		\end{itemize}
	\end{example}
	
	Now, we recall some notions about the reiterated $H$-supralgebra will be very useful in this study. For that, we define the product action $\mathcal{H}^{\ast}$ of the preceding two actions (\ref{2.1}) and (\ref{2.1.1}) by 
	\begin{eqnarray*}
		\mathcal{H}^{\ast} = (H_{\varepsilon}^{\ast})_{\varepsilon>0}: \\
		H_{\varepsilon}^{\ast}(x,x') = \left(\dfrac{x}{\varepsilon_{1}}, \dfrac{x'}{\varepsilon_{2}}\right) \quad ((x,x') \in \mathbb{R}^{d}\times\mathbb{R}^{m}).
	\end{eqnarray*} 
	In the sequel, action (\ref{2.1.1}) will be denoted by $\mathcal{H}' = (H_{\varepsilon}^{'})_{\varepsilon>0}$, that is $H_{\varepsilon}^{'} = x/\varepsilon_{2}$ $(x\in \mathbb{R}^{m})$.
	
	\begin{remarq}
		\textup{	Following \cite[Page 118]{tacha3}, if $u \in \mathcal{C}(\Omega\times \mathbb{R}^{d}_{y} \times \mathbb{R}^{d}_{z})$ (resp. in $ \mathcal{C}(\overline{\Omega}; \mathcal{B}( \mathbb{R}^{d}_{y} \times \mathbb{R}^{d}_{z}))$) then $u^{\varepsilon}(x)= u(x, H^{\ast}_{\varepsilon}(x,x)) \in \mathcal{C}(\Omega)$ (resp. in $\mathcal{B}(\Omega)$). Let $\Phi$ be a $N$-function. Also, if $u \in L^{\Phi}(\Omega; V)$ where $V$ is a closed vector subspace of $\mathcal{B}( \mathbb{R}^{d}_{y} \times \mathbb{R}^{d}_{z})$ then $u^{\varepsilon}(x)= u(x, H^{\ast}_{\varepsilon}(x,x')) \in L^{\Phi}(\Omega)$.  }
	\end{remarq}
	
	\begin{proposition}\cite{gabri}\\
		Let  $A_{1}$ and $A_{2}$) two $H$-supralgebras on $\mathbb{R}^{d}$ and $\mathbb{R}^{m}$), for $\mathcal{H}$ and $\mathcal{H}'$,  respectively. We set 
		\begin{equation*}
			A_{1}\otimes A_{2} = \left\{ \sum_{\text{finite}}u_{i}\otimes v_{i} : u_{i} \in A_{1} \text{ and } v_{i}\in A_{2}\right\}
		\end{equation*}
		and 
		\begin{equation*}
			A_{1}\odot A_{2} = \overline{A_{1}\otimes A_{2}},
		\end{equation*}
		i.e., the closure in $\mathcal{B}(\mathbb{R}^{d}\times\mathbb{R}^{m})$ of the tensor product $A_{1}\otimes A_{2}$.
		Then $A_{1}\odot A_{2}$ is an $H$-supralgebra on $\mathbb{R}^{d}\times\mathbb{R}^{m}$ for $\mathcal{H}^{\ast}$, called \textit{reiterated $H$-supralgebra} (or $RH$-subalgebra, shortly). 
	\end{proposition}
	
	
	In all that follows, $A= A_{y}\odot A_{z}$ is the $RH$-supralgebra on $\mathbb{\mathbb{R}}_{y}^{d}\times \mathbb{\mathbb{R}}_{z}^{m}$, $ A_{y}$ and $A_{z}$ being  $H$-supralgebras on $\mathbb{\mathbb{R}}_{y}^{d}$ and $\mathbb{\mathbb{R}}_{z}^{m}$ respectively. We use the same
	letter $\mathcal{G}$ to denote the Gelfand transformation on $ A_{y}$, $A_{z}$ and $A$ as well. Points in $%
	\Delta (A_{y})$ (resp. $%
	\Delta (A_{z})$) are denoted by $s$ (resp. $r$). We still denote by $M$ the mean value on $%
	\mathbb{\mathbb{R}}^{d}$ for  $\mathcal{H}$ and for $\mathcal{H}'$, and on $\mathbb{\mathbb{R}}^{d+m}$ for $\mathcal{H}^{\ast}$ as well. The
	compact space $\Delta (A_{y})$ (resp. $\Delta (A_{z})$) is equipped with the $M$-measure $\beta_{y} $ (resp. $\beta_{z} $) for $A_{y}$ (resp. $A_{z}$)%
	. It is fundamental to recall that we have $\Delta(A) = \Delta (A_{y})\times \Delta (A_{z})$ and further the $M$-measure for $A$, with which $\Delta(A)$ is equiped, is precisely the product measure $\beta=\beta_{y}\times\beta_{z}$ (see \cite{nguetseng2003homogenization}).  \\
	
	We give now some examples of  $RH$-supralgebras. 
	
	\begin{example}\label{cb}\cite{gabri} 
		\begin{itemize}
			\item[(1)] \textup{ If $A_{y}=\mathcal{C}_{per}(Y)$ and $A_{z}= \mathcal{C}_{per}(Z)$ then   }
			\begin{equation*}
				A= A_{y}\odot A_{z}= \mathcal{C}_{per}(Y\times Z), \text{ where } Y= Z=(0,1)^{d}.
			\end{equation*}
			\item[(2)] \textup{ If $A_{y}= AP(\mathbb{R}^{d}_{y})$ and $A_{z}= AP(\mathbb{R}^{d}_{z})$ then  }
			\begin{equation*}
				A = A_{y}\odot A_{z} = AP(\mathbb{R}^{d}_{y}\times \mathbb{R}^{d}_{z}).
			\end{equation*}
			\item[(3)] \textup{ If $A_{y}= WAP(\mathbb{R}^{d}_{y})$ and $A_{z}= WAP(\mathbb{R}^{d}_{z})$ then  }
			\begin{equation*}
				A = A_{y}\odot A_{z} = WAP(\mathbb{R}^{d}_{y}\times \mathbb{R}^{d}_{z}).
			\end{equation*}
			\item[(4)] \textup{ If $A_{y} = \mathcal{B}_{\infty}(\mathbb{R}^{d}_{y})$ and $A_{z}= \mathcal{B}_{\infty}(\mathbb{R}^{d}_{z})$, then  }
			\begin{equation*}
				A= A_{y}\odot A_{z} = \mathcal{B}_{\infty}(\mathbb{R}^{d}_{y} \times \mathbb{R}^{d}_{z}).
			\end{equation*}
			\item[(5)] \textup{ If $A_{y} = AP(\mathbb{R}^{d}_{y})$ and $A_{z}= \mathcal{B}_{\infty}(\mathbb{R}^{d}_{z})$, then   }
			\begin{equation*}
				A= A_{y}\odot A_{z} = \mathcal{B}_{\infty}(\mathbb{R}^{d}_{z}; AP(\mathbb{R}^{d}_{y})).
			\end{equation*}
			\item[(6)] \textup{ If $A_{y} = \mathcal{C}_{per}(Y)$ and $A_{z}= \mathcal{B}_{\infty}(\mathbb{R}^{d}_{z})$, then   }
			\begin{equation*}
				A= A_{y}\odot A_{z} = \mathcal{B}_{\infty}(\mathbb{R}^{d}_{z}; \mathcal{C}_{per}(Y)).
			\end{equation*}
			\item[(7)] \textup{ If $A_{y} = \mathcal{B}_{\infty,\mathbb{Z}^{d}}(\mathbb{R}^{d}_{y})$ and $A_{z}= \mathcal{B}_{\infty}(\mathbb{R}^{d}_{z}) $, then  }
			\begin{equation*}
				A= A_{y}\odot A_{z} = \mathcal{B}_{\infty}(\mathbb{R}^{d}_{z}; \mathcal{B}_{\infty,\mathbb{Z}^{d}}(\mathbb{R}^{d}_{y})).
			\end{equation*}
			\item[(8)] \textup{ If $A_{y} = \mathcal{B}_{\infty,\mathbb{Z}^{d}}(\mathbb{R}^{d}_{y})$ and $A_{z}= \mathcal{C}_{per}(Z)$, then  }
			\begin{equation*}
				A= A_{y}\odot A_{z} =  \mathcal{C}_{per}(Z; \mathcal{B}_{\infty,\mathbb{Z}^{d}}(\mathbb{R}^{d}_{y})).
			\end{equation*}
			\item[(9)] \textup{ If $A_{y} = AP(\mathbb{R}^{d}_{y})$ and $A_{z}= WAP(\mathbb{R}^{d}_{z})$, then  }
			\begin{equation*}
				A= A_{y}\odot A_{z} =  AP(\mathbb{R}^{d}_{y}; WAP(\mathbb{R}^{d}_{z})).
			\end{equation*}
			\item[(10)] \textup{ If $A_{y}= FS(\mathbb{R}^{d}_{y})$ and $A_{z}= FS(\mathbb{R}^{d}_{z})$ then  }
			\begin{equation*}
				A = A_{y}\odot A_{z} = FS(\mathbb{R}^{d}_{y}\times \mathbb{R}^{d}_{z}).
			\end{equation*}
		\end{itemize}
	\end{example}

	Next, let $A$ be an $H$-supralgebra on $\mathbb{R}^{d}$. The partial derivative of index $i$ ($1\leq i\leq d$) on $\Delta (A)$
	is defined to be the mapping $\partial _{i}=\mathcal{G}\circ \partial
	/\partial y_{i}\circ \mathcal{G}^{-1}$ (usual composition) of $\mathcal{D}%
	^{1}(\Delta (A))=\{\varphi \in \mathcal{C}(\Delta (A)):\mathcal{G}%
	^{-1}(\varphi )\in A^{1}\}$ into $\mathcal{C}(\Delta (A)),$ where $%
	A^{1}=\{\psi \in \mathcal{C}^{1}(\mathbb{R}^{d}):$ $\psi ,\partial \psi
	/\partial y_{i}\in A$ ($1\leq i\leq d$)$\}$. Higher order derivatives are
	defined analogously, and one also defines the space $A^{m}$ (integers $m\geq
	1$) to be the space of all $\psi \in \mathcal{C}^{m}(\mathbb{R}_{y}^{d})$
	such that $D_{y}^{\alpha }\psi =\frac{\partial ^{\left\vert \alpha
			\right\vert }\psi }{\partial y_{1}^{\alpha _{1}}\cdot \cdot \cdot \partial
		y_{d}^{\alpha _{d}}}\in A$ for every $\alpha =(\alpha _{1},...,\alpha
	_{d})\in \mathbb{N}^{d}$ with $\left\vert \alpha \right\vert \leq m$, and we
	set $A^{\infty }=\cap _{m\geq 1}A^{m}$. At the present time, let $\mathcal{D}%
	(\Delta (A))=\{\varphi \in \mathcal{C}(\Delta (A)):$ $\mathcal{G}%
	^{-1}(\varphi )\in A^{\infty }\}.$ Endowed with a suitable locally convex
	topology, $A^{\infty }$ (resp. $\mathcal{D}(\Delta (A))$) is a Fr\'{e}chet
	space and further, $\mathcal{G}$ viewed as defined on $A^{\infty }$ is a
	topological isomorphism of $A^{\infty }$ onto $\mathcal{D}(\Delta (A))$.
	We denote  by $%
	\mathcal{D}^{\prime }(\Delta (A))$ the space of distributions on $\Delta (A)$.
	
	\subsubsection{Orlicz spaces relies on spectrum of supralgebra}\label{labelsub2sub2sect2} 
	
	We recall here some notions about the  Orlicz spaces associated to an any $H$%
	-supralgebra $A$. We refer to \cite{nnang2014deterministic} for further details.
	
	Let $\Phi$ be an $N$-function, and $u\in A$. According to \cite[Lemma 2.5]{nnang2014deterministic}, we have $\Phi(|u|) \in A$ and  
	\begin{equation}\label{gelfandphi1}
		\mathcal{G}(\Phi(|u|)) = \Phi(|\mathcal{G}(u)|).
	\end{equation}
	
	This being so, for $u\in \mathcal{C}(\Delta(A))$ the function $s \rightarrow \Phi(|u(s)|)$ makes sense and lies in $\mathcal{C}(\Delta(A); \mathbb{R})$. We define the Orlicz space 
	\begin{eqnarray*}
		L^{\Phi}(\Delta(A)) = \bigg\{ u : \Delta(A) \rightarrow \mathbb{C} \; \textup{measurable}; \;\;  \int_{\Delta (A)} \Phi\left(\dfrac{|u|}{\delta}\right) d\beta < +\infty  \nonumber \\
		 \text{ for some }\delta>0 \bigg\}
	\end{eqnarray*}
	which is a Banach space for the Luxemburg norm 
	\begin{equation*}
		\|u\|_{\Phi,\Delta(A)} = \inf  \left\{ \delta > 0 \; : \;\;  \int_{\Delta (A)} \Phi\left(\dfrac{|u|}{\delta}\right) d\beta \leq 1 \right\}.
	\end{equation*}
	
	We have need the following definition.
	\begin{definition}
		An $H$-supralgebra is said being of class $\mathcal{C}^{\infty}$ if $A^{\infty}$ is dense in $A$.
	\end{definition}
	From now on we assume that  $A$ is of class $\mathcal{C}^{\infty}$
	(this is the case when, e.g., $A$ is translation invariant and moreover each
	element of $A$ is uniformly continuous; see \cite[Proposition 2.3]{woukeng2010homogenization})
	then $L^{\Phi}(\Delta (A))$ is a subspace of $\mathcal{D%
	}^{\prime }(\Delta (A))$ (with continuous embedding), so that one may define
	the Orlicz-Sobolev spaces on $\Delta (A)$ as follows. 
	\begin{equation*}
		W^{1}L^{\Phi}(\Delta (A))=\{u\in L^{\Phi}(\Delta (A)):\text{ }\partial _{i}u\in
		L^{\Phi}(\Delta (A))\text{ (}1\leq i\leq d\text{)}\}
	\end{equation*}%
	where the derivative $\partial _{i}u$ is taken in the distribution sense on $%
	\Delta (A)$. We equip $W^{1}L^{\Phi}(\Delta (A))$ with the norm
	\begin{equation*}
		\begin{array}{l}
			\displaystyle 	||u||_{W^{1}L^{\Phi}} = ||u||_{\Phi,\Delta(A)} +  \sum_{i=1}^{d}||\partial _{i}u||_{\Phi,\Delta(A)}  \text{ \thinspace }\left( u\in W^{1}L^{\Phi}(\Delta (A))\right) , 
		\end{array}%
	\end{equation*}%
	which makes it a Banach space. To that space are attached some other spaces
	such as
	$$W^{1}L^{\Phi}(\Delta (A))/\mathbb{C}=\left\{u\in W^{1}L^{\Phi}(\Delta
	(A)):\int_{\Delta (A)}ud\beta =0  \right\},$$
	provided with the seminorm
	\begin{equation*}
		\begin{array}{l}
			\displaystyle 	||u||_{W^{1}L^{\Phi}(\Delta (A))/\mathbb{C}} =   \sum_{i=1}^{d}||\partial _{i}u||_{\Phi,\Delta(A)}  \text{ \thinspace }\left( u\in W^{1}L^{\Phi}(\Delta (A))\right) , 
		\end{array}%
	\end{equation*}
	and its separated completion $%
	W^{1}_{\#}L^{\Phi}(\Delta (A))$; we refer to \cite{nnang2014deterministic} for a documented
	presentation of these spaces. We denote by $J$ the canonical mapping of $W^{1}L^{\Phi}(\Delta (A))/\mathbb{C}$ into its separated completion. Furthermore, the derivative $\partial_{i}$ viewed as a mapping of $W^{1}L^{\Phi}(\Delta (A))/\mathbb{C}$ into $L^{\Phi}(\Delta(A))$ extends to a unique continuous linear mapping, still denoted by $\partial_{i}$, of $W^{1}_{\#}L^{\Phi}(\Delta (A))$ into $L^{\Phi}(\Delta(A))$ such that $\partial_{i} J(v) = \partial_{i} v$ for $v \in W^{1}L^{\Phi}(\Delta (A))/\mathbb{C}$ and  
	\begin{equation*}
		\begin{array}{l}
			\displaystyle 	||u||_{W^{1}_{\#}L^{\Phi}} =   \sum_{i=1}^{d}||\partial _{i}u||_{\Phi,\Delta(A)}  \text{ \thinspace }\left( u\in W^{1}_{\#}L^{\Phi}(\Delta (A))\right). 
		\end{array}%
	\end{equation*}
	We also have need the following definition.
	\begin{definition}
		An $H$-supralgebra $A$ is termed $\Phi$-total if $\mathcal{D}(\Delta(A))$ is dense in $W^{1}L^{\Phi}(\Delta (A))$.
	\end{definition}
	\begin{remarq}
		\textup{	Suppose that $A$ is $\Phi$-total. Then the following assertions hold true (see, \cite[Proposition 2.7]{nnang2014deterministic}): }
		\begin{itemize}
			\item[\textup{(i)}] \textup{ $J(\mathcal{D}(\Delta(A))/\mathbb{C})$ is dense in $W^{1}_{\#}L^{\Phi}(\Delta (A))$, where $\mathcal{D}(\Delta(A))/\mathbb{C} = \mathcal{D}(\Delta(A)) \cap  W^{1}L^{\Phi}(\Delta (A))/\mathbb{C}$.  }
			\item[\textup{(ii)}] \textup{ $\displaystyle \int_{\Delta (A)} \partial_{i} u\, d\beta = 0$ ($1\leq i\leq d$) for any $u \in W^{1}_{\#}L^{\Phi}(\Delta (A))$.}
		\end{itemize}
	\end{remarq}

	\subsubsection{An appropriate representation   of elements in $L^{\Phi}(\Delta(A))$ and $W^{1}L^{\Phi}(\Delta (A)$}\label{labelsub2sub3sect2} 
	
	We can define the appropriate representation of  Orlicz-Sobolev spaces associated to an $H$%
	-supralgebra $A$. 
	
	Following \cite[Subsect. 2.3.2]{nnang2014deterministic}), we denote by $\Xi^{\Phi}(\mathbb{R}^{d}_{y})\equiv \Xi^{\Phi}$ the space of all $u\in L^{\Phi}_{\textup{loc}}(\mathbb{R}^{d}_{y})$ for which the sequence $(u^{\varepsilon})_{0<\varepsilon\leq 1}$ is bounded in $L^{\Phi}_{\textup{loc}}(\mathbb{R}^{d}_{y})$ (where $u^{\varepsilon}(x) = u(H_{\varepsilon}(x))$, $x\in \mathbb{R}^{d}$, $H_{\varepsilon}$ defined in (\ref{2.1})). Provided with the norm 
	\begin{equation*}
		\|u\|_{\Xi^{\Phi}} = \sup_{0< \varepsilon\leq 1}  \|u^{\varepsilon}\|_{\Phi,B_{d}}, \quad u \in \Xi^{\Phi},
	\end{equation*}
	where $B_{d}$ denotes the open unit ball in $\mathbb{R}^{d}_{x}$, $\Xi^{\Phi}(\mathbb{R}^{d}_{y})$ is a Banach space. We define the Banach space $\mathfrak{X}^{\Phi}_{A_{y}}(\mathbb{R}^{d}_{y})\equiv \mathfrak{X}^{\Phi}_{A_{y}}$ to be the closure of $A_{y}$ in $\Xi^{\Phi}(A_{y})$, where $A_{y}$ is an $H$-supralgebra on $\mathbb{R}^{d}_{y}$.
	
	Now, following \cite[Subsect. 2.2]{tacha3}, we define $\Xi^{\Phi}(\mathbb{R}^{d}_{y} ; \mathcal{B}(\mathbb{R}^{d}_{z}))\equiv \Xi^{\Phi}$ (if there is not confusion)  the following space
	\begin{eqnarray*}
		\Xi^{\Phi}(\mathbb{R}^{d}_{y} ; \mathcal{B}(\mathbb{R}^{d}_{z})) = \left\{ u \in L_{loc}^{\Phi}(\mathbb{R}^{d}_{y} ; \mathcal{B}(\mathbb{R}^{d}_{z})) \; : \; \text{for every } U \in \mathcal{A}(\mathbb{R}^{d}_{x}), \right.  \\
		\left.	\sup_{0< \varepsilon\leq 1} \inf \left\{ k > 0 \; : \; \int_{U} \Phi\left(\dfrac{\left\|u\left(\frac{x}{\varepsilon_{1}}, \cdot\right) \right\|_{L^{\infty}}}{k}\right) dx \leq 1  \right\} < \infty \right\}.	\nonumber
	\end{eqnarray*}
	Hence putting 
	\begin{equation*}
		\|u\|_{\Xi^{\Phi}(\mathbb{R}^{d}_{y} ; \mathcal{B}(\mathbb{R}^{d}_{z})} = \sup_{0< \varepsilon\leq 1} \inf \left\{ k > 0  : \, \int_{B_{d}(0,1)} \Phi\left(\dfrac{\left\|u\left(\frac{x}{\varepsilon_{1}}, \cdot\right) \right\|_{L^{\infty}}}{k}\right) dx \leq 1  \right\},
	\end{equation*}
	with $B_{d}(0,1)$ being the unit ball of $\mathbb{R}^{d}_{x}$ centered at the origin, we have a norm on $\Xi^{\Phi}(\mathbb{R}^{d}_{y} ; \mathcal{B}(\mathbb{R}^{d}_{z})$ which makes it a Banach space.
	
	For $A = A_{y}\odot A_{z}$, we denote by $\mathfrak{X}^{\Phi}_{A}(\mathbb{R}^{d}_{y} ; \mathcal{B}(\mathbb{R}^{d}_{z}) \equiv \mathfrak{X}^{\Phi}_{A}$ the closure of $A$ in $\Xi^{\Phi}(\mathbb{R}^{d}_{y} ; \mathcal{B}(\mathbb{R}^{d}_{z})$.

	Now let $A$ be a $H$-supralgebra. For any $u \in A$, one has 
	\begin{equation}
		M\left(\Phi\left(\dfrac{|u|}{\delta}\right)\right) = \int_{\Delta(A)} \mathcal{G}\left(\Phi\left(\dfrac{|u|}{\delta}\right)\right) d\beta =\int_{\Delta (A)} \Phi\left(\dfrac{|\hat{u}|}{\delta}\right) d\beta \geq 0 \label{meannew1}
	\end{equation}%
	for all $\delta>0$, where $\hat{u} = \mathcal{G}(u)$ (see, (\ref{gelfandphi1})). Naturally, (\ref{meannew1}) holds when $u\in \mathfrak{X}^{\Phi}_{A}$. So, instead of the $\Xi^{\Phi}$-norm, we endow $\mathfrak{X}^{\Phi}_{A}$ with the seminorm 
	\begin{equation*} 
		\|u\|_{\Phi,A} = \inf \left\{ \delta>0 \; : \; 	M\left(\Phi\left(\dfrac{|u|}{\delta}\right)\right) \leq 1 \right\}, \quad u \in \mathfrak{X}^{\Phi}_{A},
	\end{equation*}	
	which, by (\ref{meannew1}) yields that 
	\begin{equation*}
		\|\mathcal{G}(u)\|_{\Phi,\Delta(A)} = \|u\|_{\Phi,A} \quad  u \in \mathfrak{X}^{\Phi}_{A}.	
	\end{equation*}
	Hence, endowed with the seminorm $\|\cdot\|_{\Phi,A}$, $\mathfrak{X}^{\Phi}_{A}$ is complete,
	verifying $\mathfrak{X}^{\Phi}_{A}\subset \mathfrak{X}^{\Psi}_{A}$
	when $\Phi\succ \Psi$. In other, by we have
	\begin{equation*}
		\|u\|_{\Phi,A} \leq C_{1} \|u\|_{\Xi^{\Phi}} \text{ for all } u\in A,
	\end{equation*}
	with $C_{1}>0$.
	
	\begin{notation}
		\textup{	We set $\mathfrak{X}^{\Phi,\infty}_{A_{y}}=\mathfrak{X}^{\Phi}_{A_{y}}\cap L^{\infty }(\mathbb{R}_{y}^{d})$ and for $A=A_{y}\odot A_{z}$, we set $\mathfrak{X}^{\Phi,\infty}_{A}=\mathfrak{X}^{\Phi}_{A}\cap L^{\infty}(\mathbb{R}^{d}_{y}; \mathcal{B}(\mathbb{R}^{d}_{z}))$.  }
	\end{notation}
	
	We recall that for an $H$-supralgebra $A$, the spaces $\mathfrak{X}^{\Phi}_{A}$ as defined above are not in general Fr\'{e}chet spaces since they are not separated in
	general. The following properties are worth noticing (see, \cite{sango2} for the Lebesgue spaces and \cite{nnang2014deterministic} for the last property):
	\begin{itemize}
		\item[(\textbf{1})] The Gelfand transformation $\mathcal{G}:A\rightarrow 
		\mathcal{C}(\Delta (A))$ extends by continuity to a unique continuous linear
		mapping, still denoted by $\mathcal{G}$, of $\mathfrak{X}^{\Phi}_{A}$ into $L^{\Phi}(\Delta
		(A))$. Furthermore if $u\in \mathfrak{X}^{\Phi,\infty}_{A}$
		then $\mathcal{G}(u)\in L^{\infty }(\Delta (A))$ and $\left\Vert \mathcal{G}%
		(u)\right\Vert _{L^{\infty }(\Delta (A))}\leq \left\Vert u\right\Vert
		_{L^{\infty }}$
		
		\item[(\textbf{2})] The mean value $M$ viewed as defined on $A$, extends by
		continuity to a positive continuous linear form (still denoted by $M$) on $%
		\mathfrak{X}^{\Phi}_{A}$ satisfying $M(u)=\int_{\Delta (A)}\mathcal{G}(u)d\beta $ ($u\in
		\mathfrak{X}^{\Phi}_{A}$). Furthermore, $M(\tau _{a}u)=M(u)$ for each $u\in \mathfrak{X}^{\Phi}_{A}$ and
		all $a\in \mathbb{R}^{N}$.
		
		\item[(\textbf{3})] Let $\Phi$ be an $N$-function and $\widetilde{\Phi}$ its complementary. The topological dual of $\mathfrak{X}^{\Phi}_{A}$ coincides with $\mathfrak{X}^{\widetilde{\Phi}}_{A}$. Hence, the usual multiplication $A\times
		A\rightarrow A$; $(u,v)\mapsto uv$, extends by continuity to a bilinear form 
		$\mathfrak{X}^{\Phi}_{A}\times \mathfrak{X}^{\widetilde{\Phi}}_{A}\rightarrow L^{1}(\Delta(A))$ with 
		\begin{equation*}
			\int_{\Delta (A)} |\mathcal{G}(u) \mathcal{G}(v)| d\beta \leq \left\| u\right\|_{\Phi,A}\left\| v\right\|_{\widetilde{\Phi},A}%
			\text{\ for }(u,v)\in \mathfrak{X}^{\Phi}_{A}\times \mathfrak{X}^{\widetilde{\Phi}}_{A}.
		\end{equation*}
		
		\item[(\textbf{4)}] \label{p2.3} Assume each
		element of $A$ is uniformly continuous and moreover $A$ is translation
		invariant.
		Then $A^{\infty }$ is dense in $\mathfrak{X}^{\Phi}_{A}$.
	\end{itemize}
	
	\begin{lemma}
		For every $0<\varepsilon\leq 1$ and $u \in \mathfrak{X}^{\Phi}_{A}(\mathbb{R}^{d}_{y} ; \mathcal{B}(\mathbb{R}^{d}_{z})$, there exists a constant $C\in \mathbb{R}^{+}$ such that  
		\begin{equation*}
			\|u^{\varepsilon}\|_{\Phi, B_{d}(0,1)} \leq C \|u\|_{\Phi,A},
		\end{equation*}
		where $u^{\varepsilon}(x) = u\left(\dfrac{x}{\varepsilon_{1}}, \dfrac{x}{\varepsilon_{2}}\right)$.
	\end{lemma}
	
	\begin{proof}
		See \cite[Lemma 2.1]{tacha3} where $ \mathfrak{X}^{\Phi}_{per}(\mathbb{R}^{N}_{y} ; \mathcal{C}_{b})$ is replaced by $\mathfrak{X}^{\Phi}_{A}(\mathbb{R}^{d}_{y} ; \mathcal{B}(\mathbb{R}^{d}_{z}))$.
	\end{proof}
	
	Now, for an $H$-supralgebra $A$, let $u\in \mathfrak{X}^{\Phi}_{A}$, $\|u\|_{\Phi,A} = 0$ if and only if $\mathcal{G}(u) = 0$ (see, e.g., in \cite{nnang2014deterministic}). Unfortunately, the mapping $\mathcal{G}$ (defined on $%
	\mathfrak{X}^{\Phi}_{A}$) is not in general injective. So, we denote the space of (class of) functions by  
	\begin{equation*}
		\mathcal{X}_{A}^{\Phi} = \mathfrak{X}^{\Phi}_{A}/Ker\mathcal{G},\;\;\;\;\;\;\;\;\;
	\end{equation*}%
	where $Ker\mathcal{G}$ is the kernel of $\mathcal{G}$.
	Endowed with the norm 
	\begin{equation*}
		\left\| \varrho u\right\| _{\mathcal{X}_{A}^{\Phi}}=\left\| u\right\|
		_{\Phi,A}, \quad \;\;(u\in \mathfrak{X}^{\Phi}_{A}),\;\;
	\end{equation*}%
	where $\varrho$ is the canonical mapping of $\mathfrak{X}^{\Phi}_{A}$ onto $\mathcal{X}_{A}^{\Phi}$, $\mathcal{X}_{A}^{\Phi}$ is a Banach space. Moreover, according with \cite[Theorem 2.8]{nnang2014deterministic}, The mapping $\mathcal{G}:\mathfrak{X}^{\Phi}_{A}\rightarrow L^{\Phi}(\Delta (A))$
	induces an isometric isomorphism $\mathcal{G}_{1}$ of $\mathcal{X}_{A}^{\Phi}$
	onto $L^{\Phi}(\Delta (A))$ defined by 
	\begin{equation} \label{isomorphismeG1}
		\mathcal{G}_{1}(\varrho u) = \mathcal{G}(u), \quad u \in \mathfrak{X}^{\Phi}_{A}.
	\end{equation}
	So, for any $u \in L^{\Phi}(\Delta (A))$ there exists a unique $u_{0} \in \mathcal{X}_{A}^{\Phi}$ such that $\mathcal{G}_{1}(u_{0}) = u$; that is, $u_{0}$ is the representation of $u$.
	
	next, we extend the mean value on $\mathcal{X}^{\Phi}_{A}$ by 
	\begin{equation*} 
		M_{1}(\varrho u) = M(u), \quad (u\in \mathfrak{X}^{\Phi}_{A})	
	\end{equation*}
	and then  the following results hold true for any $N$-function of $\Delta_{2}$-class (see, e.g., \cite[Corollary 2.10]{nnang2014deterministic}):	
	\begin{itemize}  \label{c2.1}
		\item[\textbf{(i)}] The spaces $\mathcal{X}^{\Phi}_{A}$ is reflexive;
		
		\item[\textbf{(ii)}] The topological dual of the space $\mathcal{X}^{\Phi}_{A}$ coincides with the space $\mathcal{X}^{\widetilde{\Phi}}_{A}$, the duality being given by 
		\begin{equation*}
			\begin{array}{l}
				\displaystyle 	\left\langle \varrho u, \varrho v\right\rangle_{\mathcal{X}^{\Phi}_{A},\mathcal{X}^{\widetilde{\Phi}}_{A}}=M(uv)=\int_{\Delta (A)}\mathcal{G}%
				_{1}(\varrho u) \mathcal{G}_{1}(\varrho v) d\beta \\ 
				\text{for }u\in \mathfrak{X}^{\Phi}_{A} \text{ and }v\in \mathfrak{X}^{\widetilde{\Phi}}_{A}\text{.}%
			\end{array}%
		\end{equation*}
	\end{itemize}
	
	\begin{remarq}
		\textup{	\label{r2.1'}The space $\mathcal{X}^{\Phi}_{A}$ is the separated
			completion of $\mathfrak{X}^{\Phi}_{A}$ and the canonical mapping of $\mathfrak{X}^{\Phi}_{A}$%
			into $\mathcal{X}^{\Phi}_{A}$ is just the canonical surjection
			of $\mathfrak{X}^{\Phi}_{A}$ onto $\mathcal{X}^{\Phi}_{A}$; see once more %
			\cite[Chap. II, Sect. 3, no 7]{bourbaki2007topologie} for the theory of completion. }
	\end{remarq}
	
	\begin{notation}
		\textup{	Let $A= A_{1}\odot A_{z}$ an $RH$-supralgebra.	By the formal integral representation of $u\in \mathfrak{X}^{\Phi}_{A}$, we mean the complex number 
			denoted and defined by }
		\begin{equation*}
			\pounds_{A} \iint_{\mathbb{R}^{2d}} u(y,z) dydz \equiv  \pounds_{A} \iint u(y,z) dy dz = M_{1}(\varrho u) = M(u).	
		\end{equation*}
		When $u\in \mathfrak{X}^{\Phi}_{A_{y}}$, we denote it by 
		\begin{equation*}
			\pounds_{A_{y}} \int_{\mathbb{R}^{d}} u(y) dy \equiv  \pounds_{A_{y}} \int u(y) dy  = M_{1}(\varrho u) = M(u).	
		\end{equation*}
	\end{notation}
	
	We recall now some notions about ergodic $H$-supralgebra which we will use in the next section.
	\begin{definition}
		\label{d2.2} We say that an $H$-supralgebra $A_{y}$ on $\mathbb{R}^{d}_{y}$ 
		is ergodic if for any $u\in A_{y}$, 
		\begin{equation*}
			\lim_{r\rightarrow +\infty }\frac{1}{\left| B_{r}\right| }%
			\int_{B_{r}}u(x+y)dx=M(u)\text{\ uniformly with respect to }y.  
		\end{equation*}
	\end{definition}
	
	The following result whose proof can be found in \cite[Lemma 2.29]{nnang2014deterministic} give us
	an equivalent property for the ergodic $H$-supralgebras.
	
	\begin{lemma}\cite{nnang2014deterministic}
		\label{p2.4}An $H$-supralgebra $A_{y}$\ on $\mathbb{R}^{d}_{y}$\ is ergodic\ if and
		only if 
		\begin{equation*}
			\lim_{r\rightarrow +\infty }\left\| \frac{1}{\left| B_{r}\right| }%
			\int_{B_{r}}u(\cdot +y)dy-M(u)\right\| _{\Phi,A_{y}}=0\text{\ for all }u\in \mathfrak{X}^{\Phi}_{A_{y}}
		\end{equation*}	
	\end{lemma}
	
	\begin{remarq}\label{algegood}
		\textup{As mentioned in \cite{gabri}, each $RH$-supralgebra of Example \ref{cb} is ergodic. Furthermore, $A_{y}$ and $A_{z}$ are of class $\mathcal{C}^{\infty}$ and are translation
			invariant, and moreover each of their elements is uniformly continuous}
	\end{remarq}
	
	Our goal now is to define another spaces attached to $\mathcal{X}_{A_{y}}^{\Phi}$ for an $H$-supralgebra $A_{y}$ on $\mathbb{R}^{d}_{y}$.
	For that, let $u\in L^{\Phi }(\Delta (A_{y}))$ and $1\leq i\leq d$. We know that $\partial _{i }u\in \mathcal{D}^{\prime }(\Delta
	(A_{y}))$ exists and is defined by 
	\begin{equation*}
		\left\langle \partial _{i }u,\varphi \right\rangle =-\left\langle u,\partial _{i }\varphi \right\rangle \text{\ for
			any }\varphi \in \mathcal{D}(\Delta (A_{y}))\text{.}  
	\end{equation*}%
	This leads to the following definition.
	
	\begin{definition}\label{formderiva1}
		\label{d2.3} By a formal derivative of index $1\leq i\leq d$ of a function $u \in \mathcal{X}_{A_{y}}^{\Phi}$
		is meant the unique element $\overline{\partial }u/\partial y_{i}$ of $\mathcal{X}_{A_{y}}^{\Phi}$ (if it exists) such that
		\begin{equation*}
			\mathcal{G}_{1}(\overline{\partial }u/\partial y_{i}) = \partial_{i} \mathcal{G}_{1}(u).  
		\end{equation*}
	\end{definition}	
	
	Due to (\ref{isomorphismeG1}), we put $(\mathcal{G}_{1})^{-1}(W^{1}L^{\Phi}(\Delta(A_{y}))) = W^{1}\mathcal{X}^{\Phi}_{A_{y}}$, that is, 
	\begin{equation*}
		W^{1}\mathcal{X}^{\Phi}_{A_{y}} = \left\{ u \in  \mathcal{X}^{\Phi}_{A_{y}}\, : \; \partial_{i} \mathcal{G}_{1}(u) \in L^{\Phi}(\Delta(A_{y})), \; 1\leq i \leq d \right\}.
	\end{equation*}
	Since for any $u \in W^{1}\mathcal{X}^{\Phi}_{A_{y}}$, $\partial_{i} \mathcal{G}_{1}(u) \in L^{\Phi}(\Delta(A_{y}))$ ($1\leq i \leq d$) and, once more taking into account (\ref{isomorphismeG1}), there exists a unique $u_{i} \in \mathcal{X}^{\Phi}_{A_{y}}$ such that $\partial_{i} \mathcal{G}_{1}(u) = \mathcal{G}_{1}(u_{i})$. 
	Hence, the characterization of $W^{1}\mathcal{X}^{\Phi}_{A_{y}}$ is 
	
	\begin{equation*} 
		W^{1}\mathcal{X}^{\Phi}_{A_{y}} = \left\{ u \in  \mathcal{X}^{\Phi}_{A_{y}}\, : \; \dfrac{\overline{\partial}u}{\partial y_{i}} \in \mathcal{X}^{\Phi}_{A_{y}}, \; 1\leq i \leq d \right\}.
	\end{equation*}
	\begin{notation}
		\textup{	In order to simplify the representation, from now on we will use the same letter $u$ (if there is no danger of confusion) to denote the equivalence class $\varrho u \in \mathcal{X}^{\Phi}_{A_{y}}$ and its representative $u \in \mathfrak{X}^{\Phi}_{A_{y}}$. }
	\end{notation}
	
	This being so, first given $u$ in $W^{1}\mathfrak{X}^{\Phi}_{A_{y}} = \left\{ u \in \mathfrak{X}^{\Phi}_{A_{y}} \; : \; \frac{\partial u}{\partial y_{i}} \in \mathfrak{X}^{\Phi}_{A_{y}} \;\;  (1\leq i \leq d) \right\}$, we have 
	\begin{equation*}
		\mathcal{G}_{1}\left(\varrho\left(\dfrac{\partial u}{\partial y_{i}}\right)\right) = \mathcal{G}\left(\dfrac{\partial u}{\partial y_{i}}\right) = \partial_{i} \mathcal{G}(u)  = \partial_{i} \mathcal{G}_{1}(\varrho(u)) \; \substack{\textup{Definition \ref{formderiva1}}  \\ =} \; \mathcal{G}_{1}\left(\dfrac{\overline{\partial} \varrho(u)}{\partial y_{i}}\right);
	\end{equation*}
	so that 
	\begin{equation*} 
		\dfrac{\overline{\partial} }{\partial y_{i}} \circ \varrho = \varrho \circ \dfrac{\partial}{\partial y_{i}} \quad \textup{on} \;\, W^{1}\mathfrak{X}^{\Phi}_{A_{y}}, \quad 1 \leq i \leq d.
	\end{equation*}
	Second, we have	
	\begin{equation}\label{tc}
		\pounds_{A_{y}}\int \dfrac{\overline{\partial} u}{\partial y_{i}} dy = 0 \quad \textup{for \; all} \;\, u \in  W^{1}\mathfrak{X}^{\Phi}_{A_{y}}, \quad 1 \leq i \leq d. 
	\end{equation}
	Third, we endow $W^{1}\mathcal{X}^{\Phi}_{A_{y}}$ with the norm
	\begin{equation*}
		\|u\|_{W^{1}\mathcal{X}^{\Phi}_{A_{y}}} = \|u\|_{\Phi, A_{y}} + \sum_{i=1}^{d} \left\|\dfrac{\overline{\partial} u}{\partial y_{i}} \right\|_{_{\Phi, A_{y}}}, \quad u \in  W^{1}\mathcal{X}^{\Phi}_{A_{y}},
	\end{equation*}
	which makes it a Banach space. Besides, the canonical mapping $\mathcal{G}_{1}$ (see (\ref{isomorphismeG1})) is an isometric isomorphism of $W^{1}\mathcal{X}^{\Phi}_{A_{y}}$ onto $W^{1}L^{\Phi}(\Delta(A_{y}))$, there exists a unique $u_{0} \in W^{1}\mathcal{X}^{\Phi}_{A_{y}}$ such that 
	\begin{equation*}
		\mathcal{G}_{1}(u_{0}) = u \quad \textup{and} \quad \mathcal{G}_{1}\left(\dfrac{\overline{\partial} u}{\partial y_{i}}\right) = \partial_{i}u, \quad 1\leq i \leq d.	
	\end{equation*} 
	We will be concerned with the space 
	\begin{equation*}
		W^{1}\mathcal{X}^{\Phi}_{A_{y}}/\mathbb{C} = \left\{ u \in W^{1}\mathcal{X}^{\Phi}_{A_{y}} \; :\; \pounds_{A_{y}}\int u dy = 0 \right\}
	\end{equation*}
	equipped with the seminorm
	\begin{equation*}
		\|u\|_{W^{1}\mathcal{X}^{\Phi}_{A_{y}}/\mathbb{C}} =  \sum_{i=1}^{d} \left\|\dfrac{\overline{\partial} u}{\partial y_{i}} \right\|_{_{\Phi, A_{y}}}, \quad u \in  W^{1}\mathcal{X}^{\Phi}_{A_{y}}/\mathbb{C},
	\end{equation*}
	which makes it a locally convex topological space in general nonseparated and noncomplete. We denote by $W^{1}_{\#}\mathcal{X}^{\Phi}_{A_{y}}$  the separated completion of $W^{1}\mathcal{X}^{\Phi}_{A_{y}}/\mathbb{C}$ (for the above seminorm), and by $J_{1}$ the canonical mapping from $W^{1}\mathcal{X}^{\Phi}_{A_{y}}/\mathbb{C}$ onto $W^{1}_{\#}\mathcal{X}^{\Phi}_{A_{y}}$. Therefore, there is a unique isometric isomorphism $\overline{\mathcal{G}}_{1}$ of $W^{1}_{\#}\mathcal{X}^{\Phi}_{A_{y}}$ onto $W^{1}_{\#}L^{\Phi}(\Delta(A_{y}))$; otherwise the restriction $\frac{\overline{\partial}}{\partial y_{i}} : W^{1}\mathcal{X}^{\Phi}_{A_{y}}/\mathbb{C} \to \mathcal{X}^{\Phi}_{A_{y}}$ ($1\leq i \leq d$) extends by continuity to a continuous linear mapping, still denoted by $\frac{\overline{\partial}}{\partial y_{i}}$, from $W^{1}_{\#}\mathcal{X}^{\Phi}_{A_{y}}$ into $\mathcal{X}^{\Phi}_{A_{y}}$ with 
	\begin{equation*}
		\dfrac{\overline{\partial}}{\partial y_{i}}\circ J_{1} = \dfrac{\overline{\partial}}{\partial y_{i}} \quad \textup{in} \;\,  W^{1}\mathcal{X}^{\Phi}_{A_{y}}/\mathbb{C}, \;\, 1\leq i \leq d	
	\end{equation*} 
	\begin{equation*}
		\overline{\mathcal{G}}_{1}\circ J_{1} = J_{1} \circ \mathcal{G}_{1} \quad \textup{in} \;\,  W^{1}\mathcal{X}^{\Phi}_{A_{y}}/\mathbb{C},
	\end{equation*} 
	\begin{equation*}
		\partial_{i}\circ J \circ \mathcal{G}_{1}  = \mathcal{G}_{1}\circ \dfrac{\overline{\partial}}{\partial y_{i}} \quad \textup{in} \;\,  W^{1}\mathcal{X}^{\Phi}_{A_{y}}/\mathbb{C}, \;\; \textup{or}
	\end{equation*} 
	\begin{equation*}
		\partial_{i}\circ  \overline{\mathcal{G}}_{1}  = \mathcal{G}_{1}\circ \dfrac{\overline{\partial}}{\partial y_{i}} \quad \textup{in} \;\,  W^{1}_{\#}\mathcal{X}^{\Phi}_{A_{y}},
	\end{equation*}
	\begin{equation*}
		\|J_{1}u\|_{W^{1}_{\#}\mathcal{X}^{\Phi}_{A_{y}}}  = \|u\|_{W^{1}\mathcal{X}^{\Phi}_{A_{y}}/\mathbb{C}} \quad  u \in   W^{1}\mathcal{X}^{\Phi}_{A_{y}}/\mathbb{C}, \quad \textup{and}
	\end{equation*}
	\begin{equation*}
		\|u\|_{W^{1}_{\#}\mathcal{X}^{\Phi}_{A_{y}}} =  \sum_{i=1}^{d} \left\|\dfrac{\overline{\partial} u}{\partial y_{i}} \right\|_{_{\Phi, A_{y}}}, \quad u \in  W^{1}_{\#}\mathcal{X}^{\Phi}_{A_{y}}.
	\end{equation*}
	Furthermore, as $J_{1}(W^{1}\mathcal{X}^{\Phi}_{A_{y}}/\mathbb{C})$ is
	dense in $W^{1}_{\#}\mathcal{X}^{\Phi}_{A_{y}}$ (this is classical), it follows that if $%
	A^{\infty }_{y}$ is dense in $A_{y}$, then $(J_{1}\circ \varrho )(A^{\infty }_{y}/\mathbb{C})$ is dense
	in $W^{1}_{\#}\mathcal{X}^{\Phi}_{A_{y}}$, where $A^{\infty }_{y}/\mathbb{C}=\{u\in A^{\infty
	}_{y}\,:\;M(u)=0\}$.
	
	\begin{notation}
		\textup{	For an $H$-supralgebra $A_{y}$ on $\mathbb{R}^{d}_{y}$, we note $\mathcal{X}^{1}_{A_{y}}= \mathcal{B}^{1}_{A_{y}}$ (the Besicovitch space, see, e.g. \cite[Subsect. 2.2]{gabri}). }
	\end{notation}

	
	\section{Reiterated $\Sigma$-convergence in Orlicz spaces} \label{labelsect3} 
	
	In this section we extend the concept of reiterated $\Sigma $-convergence (see \cite{gabri}) to Orlicz setting. 
	In all that follows, $\Omega$ is an open subset of $\mathbb{\mathbb{R}}^{d}$ (integer $d\geq 1$)
	and $A= A_{y}\odot A_{z}$ is an $RH$-supralgebra on $\mathbb{\mathbb{R}}_{y}^{d}\times \mathbb{\mathbb{R}}_{z}^{d}$, $ A_{y}$ and $A_{z}$ being  $H$-supralgebras on $\mathbb{\mathbb{R}}_{y}^{d}$ and $\mathbb{\mathbb{R}}_{z}^{d}$ respectively. We use the same
	letter $\mathcal{G}$ to denote the Gelfand transformation on $ A_{y}$, $A_{z}$ and $A$ as well. Points in $%
	\Delta (A_{y})$ (resp. $%
	\Delta (A_{z})$) are denoted by $s$ (resp. $r$). We still denote by $M$ the mean value on $%
	\mathbb{\mathbb{R}}^{d}$ for  $\mathcal{H}$ and for $\mathcal{H}'$, and on $\mathbb{\mathbb{R}}^{2d}$ for $\mathcal{H}^{\ast}$ as well. The
	compact space $\Delta (A_{y})$ (resp. $\Delta (A_{z})$) is equipped with the $M$-measure $\beta_{y} $ (resp. $\beta_{z} $) for $A_{y}$ (resp. $A_{z}$)%
	. We recall that we have $\Delta(A) = \Delta (A_{y})\times \Delta (A_{z})$ and further the $M$-measure for $A$, with which $\Delta(A)$ is equipped, is precisely the product measure $\beta=\beta_{y}\times\beta_{z}$ (see \cite{nguetseng2003homogenization}). 
	Assume that  the $N$-function $\Phi$ and its conjugate $\widetilde{\Phi}$ are of $\Delta_{2}$-class. Thanks to the identification $\mathcal{G}_{1}(\mathcal{X}^{\Phi}_{A}) = L^{\Phi}(\Delta(A))$ we will merely work on the spaces $\mathcal{X}^{\Phi}_{A}$. For $f\in A$, we will use the same letter $f$ (if there is no danger of confusion) to denote the equivalence class $\varrho f \in \mathcal{X}^{\Phi}_{A}$ and its representative $u \in \mathfrak{X}^{\Phi}_{A}$. Finally, let $\varepsilon _{1}$ and $%
	\varepsilon _{2}$ be two well separated functions of $\varepsilon $ tending
	towards zero with $\varepsilon $, that is, $0<\varepsilon _{1},\varepsilon
	_{2},\varepsilon _{2}/\varepsilon _{1}\rightarrow 0$ as $\varepsilon
	\rightarrow 0$, and such that the functions $x\mapsto x/\varepsilon _{1}$
	and $x\mapsto x/\varepsilon _{2}$ define two actions of $\mathbb{R}%
	_{+}^{\ast }$ on $\mathbb{\mathbb{R}}^{d}$.
	
	\subsection{(Weak and Strong) reiterated $\Sigma$-convergence}
	
	In this subsection, we start by define the concept of weak and strong  reiterated $\Sigma$-convergence and we give some of her properties. Furthermore, we state and prove the compactness theorem. We end this subsection by some concrete examples for the case of ergodic $H$-supralgebra.
	
	\par 	Setting 
	\begin{eqnarray*}
		L^{\Phi}(\Omega; \mathcal{X}^{\Phi}_{A}) = \bigg\{ u : \Omega\times\mathbb{R}_{y}^{d}\times\mathbb{R}_{z}^{d} \rightarrow \mathbb{C} \; \textup{measurable}; \;\; \hspace{3cm}  \\ \hspace{2cm}  \int_{ \Omega }\pounds_{A} \iint_{\mathbb{R}^{2d}}  \Phi\left(\dfrac{|u|}{\delta}\right) dxdydz  < +\infty \text{ for some }\delta>0 \bigg\},
	\end{eqnarray*}
	we define for the $RH$-supralgebra $A$, an Orlicz space endowed with the Luxemburg norm
	\begin{equation*}
		\|u\|_{L^{\Phi}(\Omega; \mathcal{X}^{\Phi}_{A})} = \inf \left\{\delta>0 \; : \; \int_{ \Omega }\pounds_{A} \iint_{\mathbb{R}^{2d}} \Phi\left(\dfrac{|u|}{\delta}\right) dxdydz  \leq 1 \right\}.
	\end{equation*}
	
	Throughout the paper, the letter $E$ will denote any ordinary sequence $E=(\varepsilon_{n})$ (integers $n\geq 0$) with $0 \leq \varepsilon_{n} \leq 1$ and $\varepsilon_{n} \rightarrow 0$ as $n\rightarrow \infty$. Such a sequence will be referred to as a \textit{fundamental sequence}.
	
	\begin{definition}
		\label{d3.1} A  sequence $(u_{\varepsilon })_{\varepsilon \in E} \subset L^{\Phi}(\Omega )$%
		is said to be 
		\begin{enumerate}
			\item[\textbf{(i)}]  weakly reiteratively $\Sigma $-convergent  in $L^{\Phi}(
			\Omega )$  to some $u_{0}\in L^{\Phi}(\Omega ;\mathcal{X}^{\Phi}_{A}) \equiv \mathcal{G}_{1}^{-1}(L^{\Phi}(\Omega\times\Delta(A)))$ if  
			\begin{equation}
				\lim_{E\ni \varepsilon \to 0}	\int_{\Omega } u_{\varepsilon }(x) f\left(x, \dfrac{x}{\varepsilon_{1}}, \dfrac{x}{\varepsilon_{2}}\right) dx
				= \int_{ \Omega }\pounds_{A} \iint_{\mathbb{R}^{2d}} u_{0}(x, y, z) f(x, y, z) dx dy dz  \;\;\;\;\;\;\;\;  \label{3.1}
			\end{equation}%
			for every $f\in L^{\widetilde{\Phi}}(\Omega; A)$.
			\item[\textbf{(ii)}] strongly reiteratively $\Sigma $-convergent  in $L^{\Phi}(
			\Omega )$  to some $u_{0}\in L^{\Phi}(\Omega ;\mathcal{X}^{\Phi}_{A})$ if the following property is verified: 
			\begin{equation}\label{c4eq4}
				\begin{array}{l}
					\textup{Given}\; \eta>0\; \textup{and}\; f \in L^{\Phi}(\Omega; A)\; \textup{with}\; \|u_{0}-  f\|_{L^{\Phi}( \Omega ;\mathcal{X}_{A}^{\Phi})} \leq \frac{\eta}{2}\;  \\
					\textup{there exists}\; \rho>0\; \textup{such that} 	
					\|u_{\varepsilon}- f^{\varepsilon}\|_{\Phi,\Omega} \leq \eta\;\; \textup{provided}\; E \ni \varepsilon < \rho,
				\end{array}
			\end{equation}
			where  $f^{\varepsilon}$ is defined on $\Omega$ by $f^{\varepsilon}(x) = f\left(x, \dfrac{x}{\varepsilon_{1}}, \dfrac{x}{\varepsilon_{2}}\right)$.
		\end{enumerate}
	\end{definition}
	
	\begin{notation}
		\textup{	When \eqref{3.1} follows, we will express in short by: $%
			u_{\varepsilon } \rightharpoonup u_{0}$  in $L^{\Phi}( \Omega
			)$-weak $R\, \Sigma$, and when \eqref{c4eq4} follows we express by $%
			u_{\varepsilon } \rightarrow u_{0}$  in $L^{\Phi}( \Omega
			)$-strong $R\, \Sigma$. The letter ``$R$'' stands for \textit{reiteratively}. 	}
	\end{notation}
	
	\begin{remarq} \label{tb} 
		\textup{	From the above definition we have the following:}
		\begin{itemize}
			\item[(1)] 	\textup{ We may the weak/strong reiteratively $\Sigma $-convergence in $L^{\Phi}(\Omega)$ as a generalization of usual  weak/strong  reiteratively $\Sigma $-convergence in $L^{p}(\Omega)$ (see, e.g., \cite[Definition 3.1]{gabri} or \cite[Definition 2.2]{luka}) when $\Phi(t)=t^{p} / p$ ($t\geq 0$ and $1<p<\infty$).}
			\item[(2)]	\textup{ The weak/strong reiteratively $\Sigma $-convergence in $L^{\Phi}(\Omega)$ is a generalization to deterministic setting of weak/strong reiteratively two-scale convergence in $L^{\Phi}(\Omega)$ (see, e.g., \cite[Definition 2.4]{tacha3}). }
			\item[(3)]	\textup{ The above convergence results \eqref{3.1} and \eqref{c4eq4} strongly relies on the following property (see \cite[Prop. 2.4]{woukeng2010homogenization}): For each $\psi$ in $A$ we have, as $\varepsilon \to 0$, }
			\begin{equation*}
				\psi^{\varepsilon} \rightarrow M(\psi) \quad \textup{in} \; L^{\infty}(\mathbb{R}^{d}_{x})\textup{-weak} \, \ast,
			\end{equation*}
			\textup{	where $\psi^{\varepsilon}$ is defined in an obvious way by $\psi^{\varepsilon}(x) = \psi(x/\varepsilon_{1}, x/\varepsilon_{2})$ for $x\in \mathbb{R}^{d}$, and $M$ is the mean value on $\mathbb{R}^{2d} = \mathbb{R}^{d}\times \mathbb{R}^{d}$ for the product action $\mathcal{H}^{\ast}$.}
			\item[(4)]	\textup{ The above Definition \ref{d3.1} extends in a canonical way, arguing in components, to vector valued functions.}
			\item[(5)]	\textup{ The uniqueness of the limit $u_{0}$ in \eqref{3.1} and \eqref{c4eq4} is ensured. Moreover, as in \cite[Prop. 2.8]{tacha3} (for the reiterated periodic case), if we suppose that the sequence $(u_{\varepsilon })_{\varepsilon \in E} \subset L^{\Phi}(\Omega )$	is	 weakly reiteratively $\Sigma $-convergent  in $L^{\Phi}(
				\Omega )$  to some $u_{0}\in L^{\Phi}(\Omega ;\mathcal{X}^{\Phi}_{A})$, then we have that: }
			\begin{itemize}
				\item[(a)]	\textup{ $u_{\varepsilon} \rightharpoonup u^{\ast}_{0}$ in $L^{\Phi}(
					\Omega )$-weak $\Sigma$ (in the sense of \cite[Definition 2.13]{nnang2014deterministic}), where $u^{\ast}_{0}(x,y) \in L^{\Phi}(\Omega ;\mathcal{X}^{\Phi}_{A_{y}})$ is given by }
				\begin{equation*}
					u^{\ast}_{0}(x,y) = \pounds_{A_{z}} \int_{\mathbb{R}^{d}} u_{0}(x, y, z) dz, \text{ for } (x,y) \in \Omega\times\mathbb{R}^{d}_{y}.
				\end{equation*}
				\item[(b)]	\textup{ $u_{\varepsilon} \rightharpoonup \widetilde{u_{0}}$ in $L^{\Phi}(
					\Omega )$-weak as $E \ni \varepsilon \to 0$, where $\widetilde{u_{0}}(x) \in L^{\Phi}(\Omega)$ is given by }
				\begin{equation*}
					\widetilde{u_{0}}(x) = \pounds_{A}\iint_{\mathbb{R}^{2d}} u_{0}(x, y, z) dy dz, \text{ for } x\in \Omega.
				\end{equation*}
			\end{itemize}
		\end{itemize}		
	\end{remarq}
	
	\begin{proposition}\label{strongweak1}
		If $(u_{\varepsilon })_{\varepsilon \in E} \subset L^{\Phi}(\Omega )$	is	 strongly reiteratively $\Sigma $-convergent  in $L^{\Phi}(
		\Omega )$  to some $u_{0}\in L^{\Phi}(\Omega ;\mathcal{X}^{\Phi}_{A})$, then:
		\begin{itemize}
			\item[(i)] $u_{\varepsilon} \rightharpoonup u_{0}$ in $L^{\Phi}( \Omega
			)$-weak $R\, \Sigma$,
			\item[(ii)] $\|u_{\varepsilon}\|_{L^{\Phi}(\Omega)} \rightarrow \|u_{0}\|_{L^{\Phi}(\Omega ;\mathcal{X}^{\Phi}_{A})}$ as $\varepsilon\to 0$.
		\end{itemize}
		Furthermore, if $v\in L^{\Phi}(\Omega; A)$ then $v^{\varepsilon} \rightharpoonup  v$ in $L^{\Phi}( \Omega
		)$-strong $R\, \Sigma$, where $v^{\varepsilon}(x) = v(x, x/\varepsilon_{1}, x/\varepsilon_{2})$ for $x\in \Omega$. 
	\end{proposition}
	
	\begin{proof}
		See \cite[Remark 2.11]{tacha3} where $\mathcal{C}_{per}(Y\times Z)$ is replaced by $A$.
	\end{proof}
	
	\begin{proposition}\label{weaktest}
		Suppose that a sequence $(u_{\varepsilon })_{\varepsilon \in E} \subset L^{\Phi}(\Omega )$	is	 weakly reiteratively $\Sigma $-convergent  in $L^{\Phi}(
		\Omega )$  to some $u_{0}\in L^{\Phi}(\Omega ;\mathcal{X}^{\Phi}_{A})$. Then \eqref{3.1} holds for $f\in \mathcal{C}(\overline{\Omega}; \mathfrak{X}_{A}^{\Phi,\infty})$, where $\mathfrak{X}_{A}^{\Phi,\infty} = \mathfrak{X}_{A}^{\Phi}\cap L^{\infty}(\mathbb{R}^{d}_{y}; \mathcal{B}(\mathbb{R}^{d}_{z}))$.
	\end{proposition}
	
	\begin{proof}
		See \cite[Proposition 2.17]{nnang2014deterministic}. There is no difficult in showing that \eqref{3.1} holds for $f \in \mathcal{C}(\overline{\Omega})\otimes \mathfrak{X}_{A}^{\Phi,\infty}$, and the desired result follows by the density of $\mathcal{C}(\overline{\Omega})\otimes \mathfrak{X}_{A}^{\Phi,\infty}$ in $\mathcal{C}(\overline{\Omega}; \mathfrak{X}_{A}^{\Phi,\infty})$.
	\end{proof}
	
	The next result is a direct consequence of above Proposition \ref{weaktest}. 
	\begin{corollary}
		Let $v \in \mathcal{C}(\overline{\Omega} ; \mathfrak{X}_{A}^{\Phi,\infty})$. Then $v^{\varepsilon} \rightharpoonup  v$ in $L^{\Phi}(\Omega)$-weak $R\, \Sigma$.
	\end{corollary}
	
	The following result is the starting point of generalized compactness results for the concept of reiterated $\Sigma$-convergence.
	
	\begin{theorem}\label{compaorlicz}
		Let $\Phi$ be an $N$-function such that $\Phi, \widetilde{\Phi} \in \Delta_{2}$ and $A$ a $RH$-supralgebra on $\mathbb{R}^{2d}$ (integer $d\geq 1$). 
		Given a bounded sequence $(u_{\varepsilon })_{\varepsilon \in E} \subset L^{\Phi}(\Omega )$, one can extract a not relabelled subsequence such that $(u_{\varepsilon })_{\varepsilon \in E}$ is weakly reiteratively $\Sigma$-convergent  in $L^{\Phi}(\Omega)$.
	\end{theorem}
	
	\begin{proof}
		As in case of Lebesgue spaces, the proof of this theorem is relies on \cite[Proposition 3.2]{gabri} where only the reflexivity of spaces is required.
		For this, we set $Y=L^{\widetilde{\Phi}}(\Omega ; \mathcal{X}_{A}^{\widetilde{\Phi}})$, $X=L^{\widetilde{\Phi}}(\Omega ; A)$. Let us define the mapping $\Gamma_{\varepsilon }$ by 
		\begin{equation*}
			\Gamma_{\varepsilon }(f)=\int_{ \Omega }u_{\varepsilon
			}f^{\varepsilon }dx \quad (f \in L^{\widetilde{\Phi}}(\Omega ; A)).
		\end{equation*}%
		where $f^{\varepsilon }(x )=f(x, x/\varepsilon _{1}
		,x/\varepsilon _{2})$ for $x \in  \Omega $. Then 
		\begin{equation*}
			\underset{\varepsilon }{\lim \sup }\left\vert \Gamma_{\varepsilon }(f
			)\right\vert \leq c\left\Vert  f\right\Vert _{L^{\widetilde{\Phi}}(\Omega; \mathcal{X}_{A}^{\widetilde{\Phi}})}\text{\ for all } f\in X.
		\end{equation*}%
		Indeed one has the inequality $\left\vert \gamma_{\varepsilon }(f)\right\vert
		\leq c\left\Vert f^{\varepsilon }\right\Vert _{L^{\widetilde{\Phi}}(
			\Omega )}$ and thus, as $\varepsilon \rightarrow 0$, $\left\Vert
		f^{\varepsilon }\right\Vert _{L^{\widetilde{\Phi}}( \Omega )}\rightarrow
		\left\Vert f\right\Vert _{L^{\widetilde{\Phi}}( \Omega ; \mathcal{X}_{A}^{\widetilde{\Phi}})}$ (see Proposition \ref{strongweak1}). We therefore apply \cite[Proposition 3.2]{gabri} with the above notation to get the existence of a subsequence $%
		E^{\prime }$ of $E$ and of a unique $u_{0}\in \left[L^{\widetilde{\Phi}}( \Omega ; \mathcal{X}_{A}^{\widetilde{\Phi}})\right]^{\prime} = L^{\Phi}( \Omega ; \mathcal{X}_{A}^{\Phi})$ (by reflexivity, since $\widetilde{\Phi}\in \Delta_{2}$) such that as $E^{\prime} \ni \varepsilon \to 0$, 
		\begin{equation*}
			\int_{ \Omega }u_{\varepsilon }f^{\varepsilon }dx \rightarrow
			\int_{ \Omega} \pounds_{A} \iint_{\mathbb{R}^{2d}} u_{0}(x, y, z)  f
			(x, y, z) dx dy dz.
		\end{equation*}%
		\ for all $f\in X$. 
	\end{proof}
	
	Next, we define the concept of $W^{1}L^{\Phi}$-properness  for reiterated $H$-supralgebras and we study the particular cases of periodic $RH$-algebra and ergodic $H$-supralgebras. 
	
	\subsection{$W^{1}L^{\Phi}$-Properness of $RH$-supralgebras}
	
	In all that follows, $\Phi$ and $\widetilde{\Phi}$ are of $\Delta_{2}$-class.
	Before defining the $W^{1}L^{\Phi}$-Properness of $RH$-supralgebras, we define firstly the concept of $R\,\Sigma$-reflexivity of Orlicz-sobolev spaces for the $RH$-supralgebras.
	
	\subsubsection{ $R\,\Sigma$-reflexivity of Orlicz-Sobolev spaces}
	
	Let us begin by define the  reiterated $\Sigma$-convergence in Orlicz-Sobolev spaces. Let $A= A_{y}\odot A_{z}$ be a $RH$-supralgebra in $\mathbb{R}^{2d}$.
	We set 
	\begin{equation*}
		W^{1}L^{\Phi}(\Omega;\mathcal{X}_{A}^{\Phi}) = \left\{ u \in L^{\Phi}(\Omega;\mathcal{X}_{A}^{\Phi})\, : \; \dfrac{\partial u}{\partial x_{i}} \in L^{\Phi}(\Omega;\mathcal{X}_{A}^{\Phi}), \, 1 \leq i \leq d \right\}.
	\end{equation*}
	Equipped with the norm 
	\begin{equation*}
		\|u\|_{W^{1,x}L^{\Phi}(\Omega;\mathcal{X}_{A}^{\Phi})} = \|u\|_{L^{\Phi}(\Omega;\mathcal{X}_{A}^{\Phi})} + \sum_{i=1}^{d} \left\|\dfrac{\partial u}{\partial x_{i}} \right\|_{L^{\Phi}(\Omega;\mathcal{X}_{A}^{\Phi})},
	\end{equation*}
	$W^{1}L^{\Phi}(\Omega;\mathcal{X}_{A}^{\Phi})$ is a Banach space.
	\begin{definition}\label{sigmaSobolev}
		A sequence $(u_{\varepsilon})_{\varepsilon\in E}\subset 	W^{1}L^{\Phi}(\Omega)$ is said to be weakly reiteratively $\Sigma$-convergent (for $A$) in $	W^{1}L^{\Phi}(\Omega)$ to some $u_{0} \in 		W^{1}L^{\Phi}(\Omega;\mathcal{X}_{A}^{\Phi}) $ if there exist two  functions $u_{1} \in L^{1}(\Omega; W^{1}_{\#}\mathcal{X}^{\Phi}_{A_{y}})$  and $u_{2}\in L^{1}(\Omega; \mathcal{X}^{1}_{A_{y}}(\mathbb{R}^{d}_{y} ; W^{1}_{\#}\mathcal{X}^{\Phi}_{A_{z}}))$ such that as $E \ni \varepsilon \to 0$, we have 
		\begin{itemize}
			\item[\textbf{(i)}] $u_{\varepsilon} \rightharpoonup u_{0}$ in $L^{\Phi}(\Omega)$-weak $R\,\Sigma$,
			\item [\textbf{(ii)}] $\dfrac{\partial u_{\varepsilon}}{\partial x_{i}} \rightharpoonup \dfrac{\partial u_{0}}{\partial x_{i}} + \dfrac{\overline{\partial} u_{1}}{\partial y_{i}} + \dfrac{\overline{\partial} u_{2}}{\partial z_{i}}$ in $L^{\Phi}(\Omega)$-weak $R\,\Sigma$, $1 \leq i \leq d$.
		\end{itemize}
		We then express by: $u_{\varepsilon} \rightarrow u_{0}$ in $W^{1}L^{\Phi}(\Omega)$-weak $R\,\Sigma$.
	\end{definition}
	
	\begin{proposition}\label{ta}
		Suppose that $A_{y}$ and $A_{z}$ are $\Phi$-total. If $(u_{\varepsilon})_{\varepsilon\in E}$ is weakly reiteratively $\Sigma$-convergent in $W^{1}L^{\Phi}(\Omega)$ to $u_{0} \in 	W^{1}L^{\Phi}(\Omega;\mathcal{X}_{A}^{\Phi})$, then $u_{\varepsilon} \rightharpoonup \widetilde{u}_{0}$ in $W^{1}L^{\Phi}(\Omega)$-weak as $E \ni \varepsilon \to 0$, where $\widetilde{u}_{0}(x) = \pounds_{A}\iint_{\mathbb{R}^{2d}}  u_{0}(x,y,z) dydz$ for $x \in \Omega$.
	\end{proposition}
	
	\begin{proof}
		Assume that the sequence	$(u_{\varepsilon})_{\varepsilon\in E}$ is weakly reiteratively $\Sigma$-convergent in $W^{1}L^{\Phi}(\Omega)$ to $u_{0} \in 	W^{1}L^{\Phi}(\Omega;\mathcal{X}_{A}^{\Phi})$. Then there exist $u_{1} \in L^{1}(\Omega; W^{1}_{\#}\mathcal{X}^{\Phi}_{A_{y}})$  and $u_{2}\in L^{1}(\Omega; \mathcal{X}^{1}_{A_{y}}(\mathbb{R}^{N}_{y} ; W^{1}_{\#}\mathcal{X}^{\Phi}_{A_{z}}))$ such that (i) and (ii) of Definition \ref{sigmaSobolev} follows. Taking into account Remark \ref{tb}[see 5.b] yields that $u_{\varepsilon} \rightharpoonup \widetilde{u}_{0}$ and  $\frac{\partial u_{\varepsilon}}{\partial x_{i}} \rightharpoonup \widetilde{\frac{\partial u_{0}}{\partial x_{i}}} + \widetilde{\frac{\overline{\partial} u_{1}}{\partial y_{i}}} + \widetilde{\frac{\overline{\partial} u_{2}}{\partial z_{i}}}$ $(1\leq i \leq d)$ in $L^{\Phi}(\Omega)$-weak as $E \ni \varepsilon \to 0$. But, according to \eqref{tc}, $\frac{\overline{\partial} u_{1}}{\partial y_{i}} = 0$ and $\frac{\overline{\partial} u_{2}}{\partial y_{i}} = 0$, and there is no difficult in proving that $\widetilde{\frac{\partial u_{0}}{\partial x_{i}}} = \frac{\partial \widetilde{u_{0}}}{\partial x_{i}}$. This end the proof.
	\end{proof}
	
	This proposition has two fundamental corollaries.
	
	\begin{corollary}
		Let hypotheses be as in Proposition \ref{ta}. Then $u_{0} \in 	W^{1}L^{\Phi}(\Omega)$.
	\end{corollary}
	
	\begin{corollary}\label{coro2}
		Let hypotheses be as in Proposition \ref{ta}, and let us assume further that each $u_{\varepsilon} \in W^{1}_{0}L^{\Phi}(\Omega)$. Then $u_{0} \in 	W^{1}_{0}L^{\Phi}(\Omega)$.
	\end{corollary}
	
	\begin{definition}
		Let $A$ be an $RH$-supralgebra.	Given a bounded open set $\Omega$ in $\mathbb{R}^{d}_{x}$, the Orlicz-Sobolev space $W^{1}L^{\Phi}(\Omega)$ is said to be $R\,\Sigma$-reflexive (for $A$) if the following holds: Given a bounded sequence $(u_{\varepsilon})_{\varepsilon\in E}$ in $W^{1}L^{\Phi}(\Omega)$, a subsequence $E^{\prime}$ can be extracted such that the sequence $(u_{\varepsilon})_{\varepsilon\in E^{\prime}}$ is weakly reiteratively $\Sigma$-convergent (for $A$) in $W^{1}L^{\Phi}(\Omega)$.
	\end{definition}
	\begin{remarq}
		\textup{	When $\Phi(t)= t^{p}/p$ ($1< p< \infty$), one has the identification $W^{1}L^{\Phi}(\Omega) = W^{1,p}(\Omega)$ and the notion of $R\,\Sigma$-reflexivity (for $W^{1,p}(\Omega)$) stated in \cite[Definition 8.3]{nguetseng2025homogenization} follows.}
	\end{remarq}
	
	We now define the concept of $W^{1}L^{\Phi}$-Properness of an $RH$-supralgebra.
	
	\begin{definition}\label{defproperness}
		The $RH$-supralgebra $A=A_{y}\odot A_{z}$  on $\mathbb{R}^{d}_{y}\times \mathbb{R}^{d}_{z}$ is said to be $W^{1}L^{\Phi}(\Omega)$-proper (resp. $W^{1}_{0}L^{\Phi}(\Omega)$-proper) if $A_{y}$ and $A_{z}$ are of class $\mathcal{C}^{\infty}$, and if the following two conditions are satisfied: 
		\begin{itemize}
			\item[\textbf{(i)}] $A_{y}$ is $\Phi$-total, i.e., $\mathcal{D}(\Delta(A_{y}))$ is dense in $W^{1}L^{\Phi}(\Delta(A_{y})$ 
			\item[\textbf{(ii)}] $A_{z}$ is $\Phi$-total, i.e., $\mathcal{D}(\Delta(A_{z}))$ is dense in $W^{1}L^{\Phi}(\Delta(A_{z})$ 
			\item[\textbf{(iii)}]  $W^{1}L^{\Phi}(\Omega)$ (resp. $W^{1}_{0}L^{\Phi}(\Omega)$) is $R\,\Sigma$-reflexive (for $A$).
		\end{itemize}
	\end{definition}
	The above definition extends to Orlicz spaces the concept of $W^{1,p}$-properness of $RH$-supralgebras given in \cite{nguetseng2025homogenization}, where a great number of $W^{1}L^{\Phi}(\Omega)$-proper $RH$-algebras for $\Phi(t)= t^{p}/p$ ($1 < p< \infty $) are available.\\
	
	Next, we now discuss about two cases : the periodic $RH$-algebras and the ergodic $H$-supralgebras. Let us begin by the first case. 
	
	\subsubsection{$W^{1}L^{\Phi}$-properness of periodic $RH$-algebras}
	
	Let $A = \mathcal{C}_{\textup{per}}(Y)\odot \mathcal{C}_{\textup{per}}(Z) \equiv \mathcal{C}_{\textup{per}}(Y\times Z)$ be the classical $RH$-algebra of $Y\times Z$-periodic complex continuous functions  on $\mathbb{R}_{y}^{d}\times \mathbb{R}_{z}^{d}$ where $Y= Z= (0,1)^{d}$ (the open unit cube in $\mathbb{R}^{d}$). It is known that $\mathcal{C}_{\textup{per}}(\square)$, $\square = Y$ or $Z$, is of class $\mathcal{C}^{\infty}$, see \cite{nguetseng2003homogenization}. We set 
	
	\begin{eqnarray*}
		L^{\Phi}_{\textup{per}}(\square) = \left\{v \in L^{\Phi}_{\textup{loc}}(\mathbb{R}^{d}) \; : \; v \; \textup{is} \; \square\textup{-periodic}  \right\},  \\
		W^{1}L^{\Phi}_{\textup{per}}(\square) = \left\{v \in W^{1}L^{\Phi}_{\textup{loc}}(\mathbb{R}^{d}) \; : \; v \; \textup{is} \; \square\textup{-periodic}  \right\}, \\
		W^{1}_{\#}L^{\Phi}_{\textup{per}}(\square) = \left\{v \in W^{1}L^{\Phi}_{\textup{per}}(\square) \; : \; \int_{Y} v dy = 0  \right\}.
	\end{eqnarray*}
	
	$L^{\Phi}_{\textup{per}}(\square)$ is a Banach space under the $L^{\Phi}$-norm (denoted $\|\cdot\|_{\Phi,\square}$). Also each of spaces $W^{1}L^{\Phi}_{\textup{per}}(\square)$ and $W^{1}_{\#}L^{\Phi}_{\textup{per}}(\square)$ is a Banach space (see, \cite{tacha1}) under norms 
	\begin{equation*}
		\|u\|_{W^{1}L^{\Phi}(\square)} = \|u\|_{\Phi,\square} + \sum_{i=1}^{d} \left\|\dfrac{\partial u}{\partial y_{i}} \right\|_{\Phi,\square}, \quad u \in W^{1}L^{\Phi}_{\textup{per}}(\square)
	\end{equation*}
	and 
	\begin{equation*}
		\|u\|_{W^{1}_{\#}L^{\Phi}(\square)} =  \sum_{i=1}^{d} \left\|\dfrac{\partial u}{\partial y_{i}} \right\|_{\Phi,\square}, \quad u \in W^{1}_{\#}L^{\Phi}_{\textup{per}}(\square),
	\end{equation*}
	respectively. Then, as mentioned in \cite{nnang2014deterministic}, $\mathcal{C}_{\textup{per}}(\square)$ is dense in $L^{\Phi}_{\textup{per}}(\square)$, and $\mathcal{C}^{\infty}_{\textup{per}}(\square)= \mathcal{C}^{\infty}(\mathbb{R}^{d})\cap \mathcal{C}_{\textup{per}}(\square)$ is dense in $W^{1}L^{\Phi}_{\textup{per}}(\square)$. We have also $ \mathfrak{X}^{\Phi}_{\mathcal{C}_{\textup{per}}(\square)} \equiv L^{\Phi}_{\textup{per}}(\square)$, $ \mathfrak{X}^{\Phi}_{\mathcal{C}_{\textup{per}}(Y\times Z)} \equiv L^{\Phi}_{\textup{per}}(Y\times Z)$ and $\pounds_{\mathcal{C}_{\textup{per}}(Y\times Z)}\iint_{\mathbb{R}^{2d}} u(y,z) dydz = \iint_{ Y\times Z} u(y,z) dydz$ ($u\in L^{\Phi}_{\textup{per}}(Y\times Z)$).
	
	We are now in position to recall the properness of the $RH$-algebra $\mathcal{C}_{\textup{per}}(Y\times Z)$.
	
	\begin{proposition}\label{te}
		The periodic $RH$-algebra $\mathcal{C}_{\textup{per}}(Y\times Z)$ is $W^{1}L^{\Phi}(\Omega)$-proper and $W^{1}_{0}L^{\Phi}(\Omega)$-proper.
		Moreover if $\widetilde{\Phi} \in \Delta'$, then $u_{1} \in L^{\Phi}(\Omega; W^{1}_{\#}L^{\Phi}_{per}(Y))$ and $u_{2} \in  L^{\Phi}(\Omega; L^{\Phi}_{per}(Y ; W^{1}_{\#}L^{\Phi}_{per}(Z)))$.
	\end{proposition}
	\begin{proof}
		Provides of \cite[Proposition 2.12]{tacha3}, Corollary \ref{coro2} and \eqref{be1}.	
	\end{proof}
	
	\begin{remarq}
		\textup{	If we take the $N$-function $\Phi(t) = t^{p}/p$ ($1<p<\infty$) in Proposition \ref{te}, then we found the $W^{1,p}$-properness result as in \cite[Proposition 8.3]{nguetseng2025homogenization}. }
	\end{remarq}

	\subsubsection{$W^{1}L^{\Phi}$-properness of ergodic $RH$-supralgebras }
	
	Now we assume in the sequel that the $H$-supralgebras  $A_{y}$ and $A_{z}$ are translation
	invariant and moreover each of their elements is uniformly continuous.
	The next result requires some
	preliminaries. Let $a = (a_{1}, a_{2})\in \mathbb{\mathbb{R}}^{d}\times \mathbb{\mathbb{R}}^{d}$. Since $A = A_{y}\odot A_{z}$ is translation
	invariant, the translation operator $\tau _{a}:A\rightarrow A$ extends by
	continuity to a unique translation operator still denoted by $\tau
	_{a}:\mathfrak{X}_{A}^{\Phi}\rightarrow \mathfrak{X}_{A}^{\Phi}$. Indeed $\tau _{a}$
	is bijective and $\left\| \tau _{a}u\right\| _{\Phi,A}=\left\| u\right\| _{\Phi,A}$
	since $M(\Phi(\left| \tau _{a}u\right|))=M(\tau _{a}\Phi(\left| u\right|))=M(\Phi(\left| u\right|))$ for all $u\in A$. Besides, as each element of 
	$A$ is uniformly continuous, the group of unitary operators $\{\tau
	_{a}: a\in \mathbb{R}^{d}\times\mathbb{R}^{d} \}$ thus defined is strongly continuous,
	i.e. $\tau _{a}u\rightarrow u$ in $\mathfrak{X}_{A}^{\Phi}$ as $\left| a\right|
	\rightarrow 0$ for all $u\in \mathfrak{X}_{A}^{\Phi}$. Moreover 
	\begin{equation*}
		M(\tau _{a}u)=M(u)\text{ for all }u\in \mathfrak{X}_{A}^{\Phi}\text{ and any }a\in \mathbb{R}^{d}\times\mathbb{R}^{d}\text{.}  
	\end{equation*}%
	
	With this in mind, we begin with the following preliminary results.
	
	\begin{definition}
		A sequence $(u_{\varepsilon})_{\varepsilon>0}$ in $L^{1}(\Omega)$ is said to be uniformly integrable if $(u_{\varepsilon})_{\varepsilon>0}$ is bounded in $L^{1}(\Omega)$ and further $\sup_{ \varepsilon>0} \int_{X} |u_{\varepsilon}| dx \rightarrow 0$ as $|X|\to 0$ ($X$ being an integrable set in $\Omega$ with $|X|$ denoting the Lebesgue measure of $X$).
	\end{definition}	
	
	\begin{lemma}\label{uniinte1}
		Assume the $H$-supralgebras  $A_{y}$ and $A_{z}$ are translation
		invariant and moreover each of element of $A_{y}$ and $A_{z}$ is uniformly continuous. Let $(u_{\varepsilon})_{\varepsilon \in E}$ be a uniformly integrable sequence in $L^{1}(\Omega)$ which weakly reiteratively $\Sigma$-converges towards $u_{0} \in L^{1}(\Omega; \mathcal{X}_{A}^{1})$ (see \cite[Definition 3.1]{gabri}). Let the sequences $(v_{\varepsilon})_{\varepsilon \in E}$ and $(w_{\varepsilon})_{\varepsilon \in E}$ defined below by 
		\begin{equation*}
			v_{\varepsilon}(x) = \int_{B_{r}} \widetilde{u_{\varepsilon }}(x+\varepsilon
			_{1}\rho)d\rho \quad \textup{and} \quad w_{\varepsilon}(x) = \int_{B_{r}} \widetilde{u_{\varepsilon }}(x+\varepsilon
			_{2}\rho)d\rho \text{\ \ }(x\in \Omega )
		\end{equation*}	
		where $\widetilde{u_{\varepsilon }}$ is the zero-extension of $u_{\varepsilon }$ out of $\Omega$ and $B_{r}= r B_{d}$ denotes the open ball (centered at the origin in $\mathbb{R}^{d}$) with radius $r>0$. 
		Then, there is a subsequence $E^{\prime}$ extracted from $E$ such that as $E^{\prime} \ni \varepsilon \to 0$,
		\begin{equation*}
			w_{\varepsilon} \rightharpoonup w_{0} \quad in \; L^{1}(\Omega)\textup{-}weak \; R\,\Sigma,
		\end{equation*}
		where $w_{0}$ is defined by $w_{0}(x,y,z) = \int_{B_{r}} u_{0}(x,y,z+\rho) d\rho $ for $(x,y,z)\in \Omega\times\mathbb{R}^{N}\times\mathbb{R}^{N}$. Moreover, for each $\varphi \in \mathcal{D}(\Omega)$ and each $f\in A_{y}$ we have, as $E^{\prime} \ni \varepsilon \to 0$, 
		\begin{equation}\label{pe}
			\int_{ \Omega } v_{\varepsilon}(x) \varphi(x) f\left(\dfrac{x}{\varepsilon_{1}}\right) dx \rightharpoonup \int_{ \Omega } \pounds_{A}\iint_{\mathbb{R}^{2d}} v_{0}(x,y,z)  f(y) \varphi(x) dx dydz,
		\end{equation} 
		where $v_{0}(x,y,z) = \int_{B_{r}} u_{0}(x,y+\rho,z) d\rho$.
	\end{lemma}
	\begin{proof}
		It is done in \cite[Lemma 2.28]{gabri} with the Besicovitch space $\mathcal{B}_{A}^{1}$ in place of $\mathcal{X}_{A}^{1}$.
	\end{proof}
	
	\begin{remarq}
		\label{rempartcase}\textup{Assume Lemma \ref{uniinte1} holds. Then as } $E^{\prime}\ni \varepsilon
		\rightarrow 0$,
		\begin{equation*}
			\frac{1}{\left\vert B_{\varepsilon_{2}r}\right\vert }\int_{B_{\varepsilon
					_{2}r}}u_{\varepsilon }(x+y )dy\rightarrow \frac{1}{\left\vert
				B_{r}\right\vert } w_{0}\text{\  \emph{in} }L^{1}( \Omega
			)\text{\emph{-weak }} R\, \Sigma \text{\emph{.}}  
		\end{equation*}%
	\end{remarq}	
	
	
	We are now able to state and prove the next capital compactness result for this work. It is our first main result as mentioned in Introduction and whose we recall here for the reader's convenience. \\
	
\noindent \textit{\textup{\textbf{Theorem}} \ref{ti}. 
		Let $\Omega$ be an open subset in $\mathbb{R}^{d}$. Let $\Phi$ be an $d$-function of class $\Delta_{2}$ such that its complementary $\widetilde{\Phi} \in \Delta_{2}$ and let $A= A_{y}\odot A_{z}$ be an $RH$-supralgebra where $A_{y}$ (resp. $A_{z}$) is an ergodic $H$-supralgebra on $\mathbb{R}^{d}_{y}$ (resp. $\mathbb{R}^{d}_{z}$). Assume that  $A_{y}$ and $A_{z}$ are translation invariant, and moreover that their elements are uniformly continuous. Finally, let $(u_{\varepsilon})_{\varepsilon\in E}$ be a bounded sequence in $W^{1}_{0}L^{\Phi}(\Omega)$. There exist a subsequence $E^{\prime}$ from $E$ and a triple $ (u_{0}, u_{1}, u_{2}) \in W^{1}_{0}L^{\Phi}(\Omega) \times L^{1}(\Omega; W^{1}_{\#}\mathcal{X}^{\Phi}_{A_{y}})\times L^{1}(\Omega; \mathcal{X}^{1}_{A_{y}}(\mathbb{R}^{d}_{y} ; W^{1}_{\#}\mathcal{X}^{\Phi}_{A_{z}}))$ such that, as $E^{\prime} \ni \varepsilon \to 0$, }
		\begin{equation*}
			u_{\varepsilon} \rightharpoonup u_{0} \quad in \;\, W^{1}_{0}L^{\Phi}(\Omega)\textup{-}weak 
		\end{equation*}
	\textit{	and }
		\begin{equation*}
			\dfrac{\partial u_{\varepsilon}}{\partial x_{j}} \rightharpoonup \dfrac{\partial u_{0}}{\partial x_{j}} + \dfrac{\overline{\partial} u_{1}}{\partial y_{j}} + \dfrac{\overline{\partial} u_{2}}{\partial z_{j}} \quad in \;\, L^{\Phi}(\Omega)\textup{-}weak \; R\, \Sigma \quad (1 \leq j \leq d).
		\end{equation*}
	\textit{	If in addition $\widetilde{\Phi} \in \Delta^{\prime}$ then $u_{1} \in L^{\Phi}(\Omega; W^{1}_{\#}\mathcal{X}^{\Phi}_{A_{y}})$ and $u_{2} \in  L^{\Phi}(\Omega; \mathcal{X}^{\Phi}_{A_{y}}(\mathbb{R}^{d}_{y} ; W^{1}_{\#}\mathcal{X}^{\Phi}_{A_{z}}))$.
	} 
	
	\begin{proof}
		By reflexivity of the space $ W^{1}_{0}L^{\Phi}(\Omega)$ and also using Theorem \ref{compaorlicz}, there exist a subsequence $E^{\prime }$ from $E$, a
		function $u_{0}\in  W^{1}_{0}L^{\Phi}(\Omega)$ and a vector
		function $\mathbf{v}=(v_{j})_{1\leq j\leq d}\in L^{\Phi}( \Omega ;%
		\mathcal{X}_{A}^{\Phi})^{d}$ such that, as $E^{\prime }\ni \varepsilon
		\rightarrow 0$, we have  \eqref{pa} and $\dfrac{\partial u_{\varepsilon }}{\partial x_{j}} \rightharpoonup
		v_{j}$ in $L^{\Phi}( \Omega )$-weak $R\, \Sigma $ ($1\leq j\leq d$). 
		
		\noindent 
		It remains to check that there exist two functions  $u_{1} \in L^{1}(\Omega; W^{1}_{\#}\mathcal{X}^{\Phi}_{A_{y}})$  and $u_{2}\in L^{1}(\Omega; \mathcal{X}^{1}_{A_{y}}(\mathbb{R}^{N}_{y} ; W^{1}_{\#}\mathcal{X}^{\Phi}_{A_{z}}))$  such that $v_{j} = \dfrac{\partial u_{0}}{\partial x_{j}} + \dfrac{\overline{\partial} u_{1}}{\partial y_{j}} + \dfrac{\overline{\partial} u_{2}}{\partial z_{j}}$  $(1 \leq j \leq d)$. The fact that $u_{1} \in L^{\Phi}(\Omega; W^{1}_{\#}\mathcal{X}^{\Phi}_{A_{y}})$ and $u_{2} \in  L^{\Phi}(\Omega; \mathcal{X}^{\Phi}_{A_{y}}(\mathbb{R}^{d}_{y} ; W^{1}_{\#}\mathcal{X}^{\Phi}_{A_{z}}))$  if $\widetilde{\Phi} \in \Delta'$   will follow from  \cite[Remark 2]{tacha2}.

		We begin by deriving the
		existence of $u_{2}\in L^{1}(\Omega; \mathcal{X}^{1}_{A_{y}}(\mathbb{R}^{d}_{y} ; W^{1}_{\#}\mathcal{X}^{\Phi}_{A_{z}}))$. For
		that purpose, let $r>0$ be freely fixed. Let $B_{\varepsilon _{2}r}$ denote
		the open ball in $\mathbb{\mathbb{R}}^{d}$ centered at the origin and of
		radius $\varepsilon _{2}r$. By the equalities 
		\begin{eqnarray*}
			&&\frac{1}{\varepsilon _{2}}\left( u_{\varepsilon }(x )-\frac{1}{%
				\left\vert B_{\varepsilon _{2}r}\right\vert }\int_{B_{\varepsilon
					_{2}r}}u_{\varepsilon }(x+\rho)d\rho \right) \\
			&=&\frac{1}{\varepsilon _{2}}\frac{1}{\left\vert B_{\varepsilon
					_{2}r}\right\vert }\int_{B_{\varepsilon _{2}r}}\left( u_{\varepsilon
			}(x )-u_{\varepsilon }(x+\rho )\right) d\rho \\
			&=&\frac{1}{\varepsilon _{2}}\frac{1}{\left\vert B_{r}\right\vert }%
			\int_{B_{r}}\left( u_{\varepsilon }(x )-u_{\varepsilon
			}(x+\varepsilon _{2}\rho )\right) d\rho \\
			&=&-\frac{1}{\left\vert B_{r}\right\vert }\int_{B_{r}}d\rho
			\int_{0}^{1}Du_{\varepsilon }(x+t\varepsilon _{2}\rho )\cdot \rho dt
		\end{eqnarray*}%
		where the dot denotes the usual Euclidean inner product in $\mathbb{\mathbb{R%
		}}^{d}$, we deduce from the boundedness of $(u_{\varepsilon })_{\varepsilon
			\in E^{\prime }}$ in $W_{0}^{1}L^{\Phi}(\Omega)$ and Jensen's inequality, that the sequence $%
		(z_{\varepsilon }^{r})_{\varepsilon \in E^{\prime }}$ defined by 
		\begin{equation*}
			z_{\varepsilon }^{r}(x,\omega )=\frac{1}{\varepsilon _{2}}\left(
			u_{\varepsilon }(x )-\frac{1}{\left\vert B_{\varepsilon
					_{2}r}\right\vert }\int_{B_{\varepsilon _{2}r}}u_{\varepsilon }(x+\rho )d\rho \right) \;(x\in  \Omega ,\varepsilon \in
			E^{\prime })
		\end{equation*}%
		is bounded in $L^{\Phi}( \Omega )$. It is important to note that in
		general the function $z_{\varepsilon }^{r}$ is well defined since $%
		u_{\varepsilon }$ and $Du_{\varepsilon }$ can be naturally extended off $\Omega$
		as elements of $L_{\text{loc}}^{\Phi}(\mathbb{R}^{d})$ and $%
		L_{\text{loc}}^{\Phi}(\mathbb{R}^{d})^{d}$, respectively.
		
		Arguing as above one also shows that the sequence $(v_{\varepsilon }^{r})_{\varepsilon\in E^{\prime}}$ defined by 
		\begin{equation*}
			v_{\varepsilon }^{r}(x )=\frac{1}{\varepsilon _{1}}\left(
			u_{\varepsilon }(x )-\frac{1}{\left\vert B_{\varepsilon
					_{1}r}\right\vert }\int_{B_{\varepsilon _{1}r}}u_{\varepsilon }(x+\rho )d\rho \right) \;(x\in  \Omega ,\varepsilon \in
			E^{\prime })	
		\end{equation*}
		is well defined and bounded in $L^{\Phi}( \Omega )$.
		
		Once more, by virtue of Theorem \ref{compaorlicz} we find that there exist a subsequence
		from $E^{\prime }$ (not relabeled) and two functions $v_{r}$ and  $z_{r}$ in $L^{\Phi}(
		\Omega ;\mathcal{X}_{A}^{\Phi})$ such that, as $E^{\prime }\ni \varepsilon
		\rightarrow 0$ 
		\begin{equation*}
			v_{\varepsilon }^{r} \rightharpoonup v_{r}\text{\ in }L^{\Phi}( \Omega
			)\text{-weak }R\,\Sigma, \text{}  
		\end{equation*}
		\begin{equation}
			z_{\varepsilon }^{r} \rightharpoonup z_{r}\text{\ in }L^{\Phi}( \Omega
			)\text{-weak }R\, \Sigma \text{.}  \label{pd}
		\end{equation}%
		
		As $(z_{\varepsilon }^{r})_{\varepsilon \in E^{\prime }}$ is bounded in $%
		L^{\Phi}( \Omega )$ we have (since $\varepsilon _{2}$, $\varepsilon
		_{2}/\varepsilon _{1}\rightarrow 0$ as $E^{\prime }\ni \varepsilon
		\rightarrow 0$) that, 
		\begin{equation}
			\varepsilon _{2}z_{\varepsilon }^{r}\rightarrow 0\text{\ in }L^{\Phi}(
			\Omega )\text{ and }\frac{\varepsilon _{2}}{\varepsilon _{1}}z_{\varepsilon
			}^{r}\rightarrow 0\text{ in }L^{\Phi}( \Omega )\text{ as }E^{\prime }\ni
			\varepsilon \rightarrow 0.  \label{5.9'}
		\end{equation}%
		
		Now, for $\varphi \in \mathcal{D}(\Omega)$, $f\in A_{y}$ and $g\in A_{z}$ we have 
		\begin{equation}
			\begin{array}{l}
				\displaystyle	\int_{ \Omega }\left( \frac{\partial u_{\varepsilon }}{\partial x_{i}}%
				(x )-\frac{1}{\left\vert B_{\varepsilon _{2}r}\right\vert }%
				\int_{B_{\varepsilon _{2}r}}\frac{\partial u_{\varepsilon }}{\partial x_{i}}%
				(x+\rho )d\rho \right) \varphi (x)f\left(  \frac{x}{%
					\varepsilon _{1}} \right) g(\frac{x}{\varepsilon _{2}})dx
				\\ 
				\displaystyle	\;\;=-\int_{ \Omega }\varepsilon _{2}z_{\varepsilon }^{r}(x,\omega
				)f\left(  \frac{x}{\varepsilon _{1}} \right) g\left(\frac{x}{%
					\varepsilon _{2}}\right)\frac{\partial \varphi }{\partial x_{i}}(x)dx \\ 
				\displaystyle	\;\;\;\;\;	-\int_{ \Omega }\frac{\varepsilon _{2}}{\varepsilon _{1}}%
				z_{\varepsilon }^{r}(x)\varphi (x) g\left(\frac{x}{\varepsilon _{2}}%
				\right) \left( \dfrac{\partial f}{\partial x_{i}}\right) \left( \frac{x}{\varepsilon _{1}}\right)
				dx \\ 
				\displaystyle	\;\;\;\;\;-\int_{ \Omega }z_{\varepsilon }^{r}(x )\varphi
				(x)f\left(  \frac{x}{\varepsilon _{1}}\right) \frac{%
					\partial g}{\partial z_{i}}\left( \frac{x}{\varepsilon _{2}}\right) dx.%
			\end{array}       \label{5.10'}
		\end{equation}%
		Passing to the limit in (\ref{5.10'}) (as $E^{\prime }\ni \varepsilon
		\rightarrow 0$) using conjointly (\ref{pd}), (\ref{5.9'}) and Remark \ref{rempartcase} one gets 
		\begin{equation*}
			\begin{array}{l}
				\displaystyle	\int_{ \Omega} \pounds_{A} \iint_{\mathbb{R}^{2d}}  \left( v_{i}(x,\cdot
				,\cdot )-\frac{1}{\left\vert B_{r}\right\vert }\int_{B_{r}}v_{i}(x,\cdot
				,\cdot +\rho )d\rho \right) \varphi (x)\, f \,g dxdy dz
				\\ 
				\displaystyle	\;\;\;=-\int_{ \Omega} \pounds_{A} \iint_{\mathbb{R}^{2d}}  z_{r}(\cdot,\cdot
				,r)\varphi(x) f \partial _{i} g dxd\beta.
			\end{array}%
		\end{equation*}%
		
		Therefore, because of
		the arbitrariness of $\varphi $, $f$ and $g$, we are led to (for $1\leq i\leq d$) 
		
		\begin{equation}\label{231}
			\frac{\overline{\partial }z_{r}}{\partial z_{i}}(x, \cdot, \cdot) =  v_{i}(x,\cdot
			,\cdot )-\frac{1}{\left\vert B_{r}\right\vert }\int_{B_{r}}v_{i}(x,\cdot
			,\cdot +\rho )d\rho \quad \text{a.e. in } x\in \Omega.
		\end{equation}
		
		Set $f_{r}=z_{r}-M_{z}(z_{r}(x,y,\cdot ))$ where here, $%
		z_{r}(x,y ,\cdot )\in \mathcal{X}_{A_{z}}^{\Phi}$ is viewed as its
		representative in $\mathfrak{X}_{A_{z}}^{\Phi}$ and $M_{z}=M$ standing here for the mean value
		on $\mathbb{R}^{d}_{z}$ with respect to $z$. Note that we also have $z_{r}(x, \cdot, z) \in \mathcal{X}_{A_{y}}^{\Phi}$ and $z_{r}(\cdot, \cdot, z) \in L^{1}(\Omega ; \mathcal{X}_{A_{y}}^{\Phi})$.
		Then $%
		M_{z}(f_{r})=0$ and moreover $\overline{D}_{z}f_{r}=\overline{D}_{z}z_{r}$
		so that $f_{r}(x,y,\cdot)\in \mathcal{X}_{A_{z}}^{\Phi}$ with $%
		\overline{\partial }f_{r}(x,y,\cdot)/\partial z_{i}\in \mathcal{X}_{A_{z}}^{\Phi}$ (by \eqref{231}) for a.e. $(x,y) \in \Omega\times\mathbb{R}%
		_{y}^{d}$, that is, 
		\begin{equation*}
			f_{r}(x,y,\cdot)\in W^{1}\mathcal{X}_{A_{z}}^{\Phi}/\mathbb{C}\text{.}%
			\;\;\;\;\;\;\;\;\;\;\;\;\;\;\;\;\;\;\;\;
		\end{equation*}%
		So let $g_{r}=J_{1}^{z} \circ f_{r}$, where $J_{1}^{z}$ denotes the canonical mapping
		of $W^{1}\mathcal{X}_{A_{z}}^{\Phi}/\mathbb{C}$ into its separated completion $%
		W^{1}_{\#}\mathcal{X}_{A_{z}}^{\Phi}$. Then $g_{r}(x,y,\cdot)\in W^{1}_{\#}\mathcal{X}_{A_{z}}^{\Phi}$ for a.e. $(x,y) \in Q\times\mathbb{R}%
		_{y}^{d}$
		and moreover 
		\begin{equation*}
			\frac{\overline{\partial }g_{r}}{\partial y_{i}}(x,y ,\cdot
			)=v_{i}(x,y ,\cdot )-\frac{1}{\left\vert B_{r}\right\vert }%
			\int_{B_{r}}v_{i}(x,y ,\cdot +\rho )d\rho \;\ \ (1\leq i\leq d)
		\end{equation*}%
		since $\frac{\overline{\partial }g_{r}}{\partial z_{i}}(x,y ,\cdot )=%
		\frac{\overline{\partial }f_{r}}{\partial z_{i}}(x,y ,\cdot )=\frac{%
			\overline{\partial }z_{r}}{\partial z_{i}}(x,y ,\cdot )$. Now, we also
		view $v_{i}(x,y ,\cdot )$ as its representative in $\mathfrak{X}_{A_{z}}^{\Phi}$. Taking
		this into account, we have 
		\begin{equation}
			\begin{array}{l}
				\left\Vert g_{r}(x,y ,\cdot )-g_{r^{\prime }}(x,y ,\cdot
				)\right\Vert _{W^{1}_{\#}\mathcal{X}_{A_{z}}^{\Phi}} \\ 
				\leq \left\Vert \overline{D}_{z}g_{r}(x,y ,\cdot )-\mathbf{v}(x,y
				,\cdot )+M_{z}(\mathbf{v}(x,y ,\cdot ))\right\Vert _{\Phi,A_{z}} \\ 
				+\left\Vert \overline{D}_{z}g_{r^{\prime }}(x,y ,\cdot )-\mathbf{v}%
				(x,y ,\cdot )+M_{y}(\mathbf{v}(x,y ,\cdot ))\right\Vert _{\Phi,A_{z}}.%
			\end{array}
			\label{5.11'}
		\end{equation}%
		But,
		\begin{equation*}
			\begin{array}{l}
				\left\Vert \overline{D}_{z}g_{r}(x,y ,\cdot )-\mathbf{v}(x,y
				,\cdot )+M_{z}(\mathbf{v}(x,y ,\cdot ))\right\Vert _{\Phi,A_{z}} \\ 
				\;\;\;=\left\Vert \frac{1}{\left\vert B_{r}\right\vert }\int_{B_{r}}\mathbf{v%
				}(x,y ,\cdot +\rho )d\rho -M_{z}(\mathbf{v}(x,y ,\cdot
				))\right\Vert _{\Phi,A_{z}}.%
			\end{array}%
		\end{equation*}%
		Therefore, since the algebra $A_{z}$ is ergodic, the right-hand side (and hence
		the left-hand side) of (\ref{5.11'}) goes to zero when $r,r^{\prime
		}\rightarrow +\infty $. Thus, the sequence $(g_{r}(x,y ,\cdot ))_{r>0}$
		is a Cauchy sequence in the Banach space $W^{1}_{\#}\mathcal{X}_{A_{z}}^{\Phi}$, whence
		the existence of a unique $u_{2}(x,y ,\cdot )\in W^{1}_{\#}\mathcal{X}_{A_{z}}^{\Phi} $ such that 
		\begin{equation*}
			g_{r}(x,y ,\cdot )\rightarrow u_{2}(x,y ,\cdot )\text{\ in }%
			W^{1}_{\#}\mathcal{X}_{A_{z}}^{\Phi} \text{\ as }r\rightarrow +\infty ,
		\end{equation*}%
		that is 
		\begin{equation*}
			\overline{D}_{z}g_{r}(x,y ,\cdot )\rightarrow \overline{D}%
			_{z}u_{2}(x,y ,\cdot )\text{\ in }(\mathcal{X}_{A_{z}}^{\Phi})^{d}\text{\ as }%
			r\rightarrow +\infty .
		\end{equation*}%
		Once again the ergodicity of $A_{z}$ and the uniqueness of the limit leads at
		once to 
		\begin{equation*}
			\overline{D}_{z}u_{2}(x,y ,\cdot )=\mathbf{v}(x,y ,\cdot )-M_{z}(%
			\mathbf{v}(x,y ,\cdot ))\text{\ a.e. in }\mathbb{R}^{d}_{z}\text{\ and for
				a.e. }(x,y )\in Q\times \mathbb{R}^{d}_{y} \text{.}
		\end{equation*}%
		We deduce the existence of a function $u_{2}:\Omega\times \mathbb{R}^{d}_{y} \rightarrow 
		W^{1}_{\#}\mathcal{X}_{A_{z}}^{\Phi}$, $(x,y )\mapsto u_{2}(x,\omega ,\cdot )$, lying in 
		$L^{1}(\Omega ; \mathcal{X}^{1}_{A_{y}}(\mathbb{R}^{d}_{y} ; W^{1}_{\#}\mathcal{X}_{A_{z}}^{\Phi}))$ such that 
		\begin{equation}
			\mathbf{v}-M(\mathbf{v})=\overline{D}_{z}u_{2}.\;\;\;\;  \label{Equ1}
		\end{equation}%
		
		Let us finally derive the existence of $u_{1}$. 
		Returning to \eqref{5.10'} and taking (in the left-hand side of the equality) there $g \equiv 1$ and replacing $B_{\varepsilon_{2}r}$ by $B_{\varepsilon_{1}r}$, and finally proceeding as we have done to get $u_{2}$ (using this time the convergence result \eqref{pe}), we get the existence of an unique $u_{1} \in L^{1}(\Omega; W^{1}_{\#}\mathcal{X}^{\Phi}_{A_{y}})$ such that 
		\begin{equation}\label{pf}
			\begin{array}{rcl}
				\overline{D}_{y}u_{1}(x,\cdot) &= & M_{z}\big( \mathbf{v}(x,\cdot,\cdot) - M_{y}(\mathbf{v}(x,\cdot,\cdot))\big)  \\
				& = & M_{z}(\mathbf{v}(x,\cdot,\cdot)) - M_{z}\left(M_{y}(\mathbf{v}(x,\cdot,\cdot)) \right).
			\end{array}
		\end{equation}
		Thus, 
		\begin{equation*}
			\begin{array}{rcl}
				\mathbf{v}(x,\cdot,\cdot) &= & M_{z}(\mathbf{v}(x,\cdot,\cdot)) + \overline{D}_{z}u_{2}(x,\cdot,\cdot)  \quad \text{by } \eqref{Equ1}  \\
				& = &  M_{z}\left(M_{y}(\mathbf{v}(x,\cdot,\cdot)) \right) + \overline{D}_{y}u_{1}(x,\cdot) + \overline{D}_{z}u_{2}(x,\cdot,\cdot) \quad \text{by } \eqref{pf}.
			\end{array}
		\end{equation*}
		Finally by the convergence result $Du_{\varepsilon} \rightharpoonup \mathbf{v}$ in $L^{\Phi}( \Omega )^{d}$-weak $R\, \Sigma $ as $E^{\prime} \ni \varepsilon \to 0$, it follows that $Du_{\varepsilon} \rightharpoonup M_{z}(M_{y}(\mathbf{v}))$ in $L^{\Phi}( \Omega )^{d}$-weak $R \Sigma $ as $E^{\prime} \ni \varepsilon \to 0$, so that by \eqref{pa} we get $M_{z}(M_{y}(\mathbf{v})) = Du_{0}$. Whence $\mathbf{v} =  Du_{0} + \overline{D}_{y}u_{1} + \overline{D}_{z}u_{2}$ and the proof is complete.
	\end{proof}
	
	\begin{corollary}
		Let $A= A_{y}\odot A_{z}$ be an $RH$-supralgebra on $\mathbb{R}^{2d}$.
		Under hypotheses of Theorem \ref{ti}, if we also assume that $A_{y}$ and $A_{z}$ are $\Phi$-total, then $A$ is $W^{1}_{0}L^{\Phi}(\Omega)$-proper.
	\end{corollary}
	
	\begin{remarq}
		\textup{	Theorem \ref{ti} easily generalizes to a finite number of separated scales as follows. Let $(A_{y_{i}})_{1\leq i\leq n}$ be an $n$ ergodic $H$-supralgebras that are translation invariant and whose elements are uniformly continuous. For $1\leq i\leq n$, set }
		\begin{equation*}
			\mathbb{R}^{id}_{y_{1},\cdots,y_{i}} = \mathbb{R}^{d}_{y_{1}} \times \cdots \times \mathbb{R}^{d}_{y_{i}} \quad \textup{and} \quad \odot_{j=1}^{i} A_{y_{j}} = A_{y_{1}} \odot \cdots \odot A_{y_{i}}. 
		\end{equation*}
		\textup{	For $f \in L^{\Phi}(\Omega; \odot_{j=1}^{i} A_{y_{j}})$ one defines the trace function $f^{\varepsilon}$ $(\varepsilon>0)$ by }
		\begin{equation*}
			f^{\varepsilon}(x) = f\left(x, \dfrac{x}{\varepsilon_{1}}, \cdots, \dfrac{x}{\varepsilon_{i}}\right) \quad (x\in \Omega),
		\end{equation*}
		\textup{	where the scales $\varepsilon_{i}$ satisfy the following relation: $\varepsilon_{i+1}/\varepsilon_{i} \rightarrow 0$ as $\varepsilon\to 0$ for $1\leq i\leq n-1$. We have the following result.}
	\end{remarq}
	
	\begin{proposition}\label{multiscal}
		Let $\Omega$ be an open subset in $\mathbb{R}^{d}$ and $\Phi$ be an $N$-function of $\Delta_{2}$-class such that $\widetilde{\Phi}$ is also $\Delta_{2}$-class. Let $(u_{\varepsilon})_{\varepsilon\in E}$ be a bounded sequence in $W^{1}_{0}L^{\Phi}(\Omega)$. There exist a subsequence $E^{\prime}$ from $E$ and a $(n+1)$-tuple $\mathbf{u} = (u_{0}, u_{1}, \cdots, u_{n}) \in W^{1}_{0}L^{\Phi}(\Omega) \times L^{1}(\Omega; W^{1}_{\#}\mathcal{X}^{\Phi}_{A_{y_{1}}})\times \prod_{i=1}^{n-1} L^{1}(\Omega; \mathcal{X}^{1}_{A_{y_{1}} \odot \cdots \odot A_{y_{i}} }(\mathbb{R}^{id}_{y_{1},\cdots,y_{i}} ; W^{1}_{\#}\mathcal{X}^{\Phi}_{A_{i+1}}))$ such that, as $E^{\prime} \ni \varepsilon \to 0$,
		\begin{equation*}\label{tf}
			u_{\varepsilon} \rightharpoonup u_{0} \quad in \;\, W^{1}_{0}L^{\Phi}(\Omega)\textup{-}weak 
		\end{equation*}
		and 
		\begin{equation*}\label{tg}
			D u_{\varepsilon} \rightharpoonup D u_{0} + \overline{D}_{y_{1}} u_{1} + \cdots + \overline{D}_{y_{n}} u_{n} \quad in \;\, L^{\Phi}(\Omega)^{d}\textup{-}weak \; R\, \Sigma.
		\end{equation*}
		If in addition $\widetilde{\Phi} \in \Delta^{\prime}$ then
		\begin{equation*}
			(u_{1}, \cdots, u_{n}) \in  L^{\Phi}(\Omega; W^{1}_{\#}\mathcal{X}^{\Phi}_{A_{y_{1}}})\times \prod_{i=1}^{n-1} L^{\Phi}(\Omega; \mathcal{X}^{\Phi}_{A_{y_{1}} \odot \cdots \odot A_{y_{i}} }(\mathbb{R}^{iN}_{y_{1},\cdots,y_{i}} ; W^{1}_{\#}\mathcal{X}^{\Phi}_{A_{i+1}})).
		\end{equation*}
	\end{proposition}
	
	\begin{proof}
		Its is copied on that of Theorem \ref{ti}. We begin by finding $u_{0}$, next $u_{n}$, and inductively, $u_{n-1}, \cdots, u_{1}$.
	\end{proof}
	
	\begin{remarq}
		\textup{	The above proposition is particular interest for the homogenization in Orlicz setting, of deterministic time dependent problems where we will have the $RH$-supralgebra $A= A_{y}\odot A_{z} \odot A_{\tau}$, with $A_{\tau}$ be a further $H$-supralgebra on $\mathcal{R}_{\tau}$. }
	\end{remarq}
	
	In  the next, we apply the results obtained this Section, in particular Theorem \ref{ti}  for the homogenization of problem \eqref{1.1}.
	
	\section{Application to the reiterated homogenization of nonlinear degenerate elliptic operators} \label{labelsect4} 
	
	Before to study homogenization problem in \eqref{1.1}, we have need an supplement hypothesis on the function $a$, that we call \textit{abstract structure hypothesis}.
	Thus, following \cite[Sect. 3]{nnang2014deterministic}, we assume that the function $a$ in \eqref{1.1} satisfy:
	\begin{itemize}
		\item[\textbf{(H$_{7}$)}] (\textit{abstract structure hypothesis})
		\begin{equation}\label{tj}
			a_{i}(\cdot, \cdot, \zeta, \lambda) \in \mathfrak{X}^{\widetilde{\Phi}}_{A}(\mathbb{R}
			_{y}^{d} ; \mathcal{B}(\mathbb{R}
			_{z}^{d}))  \quad \text{for all } (\zeta, \lambda) \in \mathbb{R}\times 
			\mathbb{R}^{d} \;\; (1\leq i \leq d),
		\end{equation}
	\end{itemize}
	where $A$ is an $RH$-supralgebra of class $\mathcal{C}^{\infty}$ on $\mathbb{R}
	_{y}^{d}\times 
	\mathbb{R}
	_{z}^{d}$ and $\Phi$ is the $N$-function of $\Delta_{2}$-class. 
	
	As we have showed in \textbf{Appendix} \ref{appen1} (see, Proposition \ref{tn}), if \eqref{tj} holds, then we have 
	\begin{equation}\label{tk}
		a_{i}(\cdot, \cdot, f, \mathbf{f}) \in  \mathfrak{X}^{\widetilde{\Phi}}_{A}(\mathbb{R}
		_{y}^{d} ; \mathcal{B}(\mathbb{R}
		_{z}^{d}))   \quad (1\leq i \leq d),
	\end{equation}
	for every $f \in A_{\mathbb{R}}$ and $\mathbf{f} \in (A_{\mathbb{R}})^{d}$, where $A_{\mathbb{R}} = A \cap \mathcal{C}(\mathbb{R}^{d}_{y}\cap \mathbb{R}^{d}_{z})$.
	
	For the reader's convenience, we refer to \textbf{Appendix} \ref{appen1} for detailed discussion about the trace results, and \textbf{Appendix} \ref{appen2} for the well posedness of the compositions appearing in \eqref{1.1}.
	
	Our main goal in this section is to investigate the limiting behaviour, as $0<\varepsilon
	\rightarrow 0,$ of a sequence of solutions $u_{\varepsilon }$ of \eqref{1.1} and deduce the homogenized problem, taking into account the \textit{abstract structure hypothesis} \textbf{(H$_{7}$)}, in particular \eqref{tk} and all the hypothesis \textbf{(H$_{1}$)}-\textbf{(H$_{6}$)} stated in Subsection \ref{hypoproblem}.
	
	Let us begin by some auxiliaries convergence results.
	
	\subsection{Preliminary convergence results}
	
	Given $f \in A_{\mathbb{R}}$, $\mathbf{f} \in (A_{\mathbb{R}})^{d}$ and $1\leq i\leq d$, it follows from \textbf{(H$_{1}$)}-\textbf{(H$_{2}$)} in Section \ref{labelintro} combined with \eqref{tk}, that the function $(y,z) \to a_{i}(y,z, f(y,z), \mathbf{f}(y,z))$, denoted by $a_{i}(\cdot,\cdot, f, \mathbf{f})$, belongs to $\mathfrak{X}_{A}^{\widetilde{\Phi }%
		,\infty }\left( \mathbb{R}_{y}^{d},\mathcal{B}\left( 
	\mathbb{R}
	_{z}^{d}\right) \right) = \mathfrak{X}_{A}^{\widetilde{\Phi }%
	}\left( \mathbb{R}_{y}^{d},\mathcal{B}\left( 
	\mathbb{R}
	_{z}^{d}\right) \right) \cap L^{\infty }\left( \mathbb{R}_{y}^{d},\mathcal{B}\left( 
	\mathbb{R}
	_{z}^{d}\right) \right)$ ($\mathfrak{X}_{A}^{\widetilde{\Phi }%
		,\infty }$ equiped with the $L^{\infty}$-norm).
	
	This being so, under the extra \textit{abstract structure hypothesis} on the first two variables of $a$ and for $(f, {\mathbf f}) \in \mathcal{C}(\overline{\Omega} ; A_{\mathbb{R}}\times (A_{\mathbb{R}})^{d})$, using once more \textbf{(H$_{2}$)} , we conclude that, for every $1\leq i \leq d$, the map
	$x\rightarrow a_{i}\left(
	\cdot, \cdot, f\left( x,\cdot,\cdot\right) ,\mathbf{f}\left( x,\cdot,\cdot\right) \right) $ from $\overline{\Omega}$ to $\mathfrak{X}_{A}^{\widetilde{\Phi }%
		,\infty }\left( \mathbb{R}_{y}^{d},\mathcal{B}\left( 
	\mathbb{R}
	_{z}^{d}\right) \right)$  (still denoted by $a_{i}(\cdot,\cdot, f, \mathbf{f})$) belongs to $\mathcal{C}\left( \overline{\Omega };\mathfrak{X}_{A}^{\widetilde{\Phi }%
		,\infty }\left( \mathbb{R}_{y}^{d},\mathcal{B}\left( 
	\mathbb{R}
	_{z}^{d}\right) \right) \right)$,
	where $a_{i}\left( \cdot, \cdot,w \left(
	\cdot,\cdot\right) ,\mathbf{W }\left( \cdot,\cdot\right) \right) $ is the function $%
	\left( y,z\right) \rightarrow a_{i}\left( y,z,w \left( y,z\right),
	\mathbf{W}\left( y,z\right) \right),$ with $\left( w ,\mathbf{W}\right)
	\in $
	$A_{\mathbb{R}}\times (A_{\mathbb{R}})^{d}$ and 
	\begin{equation*}
		\begin{array}{rcl}
			\left\vert a\left( y,z,f ,\mathbf{f} \right) -a\left( y,z,f',\mathbf{f}'\right) \right\vert & \leq & c_{1}\widetilde{\Psi }%
			^{-1}\left( \Phi \left( c_{2}\left\vert f - f^{\prime }\right\vert
			\right) \right)  \\  
			&  & + c_{3}\widetilde{\Phi }^{-1}\left( \Phi \left(
			c_{4}\left\vert \mathbf{f} - \mathbf{f}' \right\vert \right) \right)  
		\end{array}
	\end{equation*}
	in $\overline{\Omega}\times \mathbb{R}^{d}_{y}\times \mathbb{R}^{d}_{z}$ for all $\left( f ,\mathbf{f}\right), (f', \mathbf{f}')
	\in $
	$\mathcal{C}(\overline{\Omega} ; A_{\mathbb{R}}\times (A_{\mathbb{R}})^{d})$. 
	
	Therefore, for fixed $\varepsilon>0$, one defines the function $x \to a\left(\frac{x}{\varepsilon}, \frac{x}{\varepsilon^{2}}, f\left(x, \frac{x}{\varepsilon}, \frac{x}{\varepsilon^{2}}\right), \mathbf{f}\left(x, \frac{x}{\varepsilon}, \frac{x}{\varepsilon^{2}}\right) \right)$ of $\Omega$ into $\mathbb{R}$ (denoted by $a^{\varepsilon}(\cdot,\cdot, f^{\varepsilon}, \mathbf{f}^{\varepsilon})$) as an element of $L^{\infty}_{\mathbb{R}}(\Omega)^{d}$.
	
	Following \cite[Sect. 3]{nnang2014deterministic}, for $\left( g ,\mathbf{g}\right) \in \mathcal{C}_{\mathbb{R}}(\Delta(A)) \times \mathcal{C}_{\mathbb{R}}(\Delta(A))^{d}$ ($\mathcal{C}_{\mathbb{R}}(\Delta(A)) = \mathcal{C}(\Delta(A); \mathbb{R})$), we set 
	\begin{equation*}\label{to}
		\widehat{b_{i}}(g, \mathbf{g}) = \mathcal{G}(a_{i}(\cdot,\cdot, \mathcal{G}^{-1}g, \mathcal{G}^{-1}\mathbf{g})), \quad 1\leq i \leq d,
	\end{equation*}
	where $\mathcal{G}^{-1}\mathbf{g} = (\mathcal{G}^{-1}g_{i})_{1\leq i \leq d}$. Hence, thanks once more to  \eqref{tk} one defines a mapping $\widehat{b_{i}}$ of $\mathcal{C}_{\mathbb{R}}(\Delta(A)) \times \mathcal{C}_{\mathbb{R}}(\Delta(A))^{d}$ to $L^{\infty}_{\mathbb{R}}(\Delta(A))$. When $\left( g ,\mathbf{g}\right) \in \mathcal{C}(\overline{\Omega} ; \mathcal{C}_{\mathbb{R}}(\Delta(A))) \times \mathcal{C}(\overline{\Omega} ; \mathcal{C}_{\mathbb{R}}(\Delta(A)))^{d} \equiv \mathcal{C}(\overline{\Omega} ; \mathcal{C}_{\mathbb{R}}(\Delta(A)) \times \mathcal{C}_{\mathbb{R}}(\Delta(A))^{d})$, we will still denote by $\widehat{b_{i}}(g,\mathbf{g})$ the mapping $x \to b_{i}(g(x), \mathbf{g}(x))$ from $\overline{\Omega}$ into $L^{\infty}_{\mathbb{R}}(\Delta(A))$. Now, for $u \in A_{\mathbb{R}}$ and $\mathbf{u} \in (A_{\mathbb{R}})^{d}$, letting 
	\begin{equation}\label{tp}
		b_{i}(u, \mathbf{u}) = \mathcal{G}_{1}^{-1}(\widehat{b_{i}}(\mathcal{G}u, \mathcal{G}\mathbf{u})), \quad 1\leq i \leq d.
	\end{equation}
	
	Hence, we are in position to state our first convergence result reasoning exactly as in the proof of \cite[Proposition 3.4]{tacha4}.
	\color{black}
	\begin{proposition}\label{prop3.3}
		For $\left( f,\mathbf{f }\right) \in 
			\mathcal{C}\left( \overline{%
				\Omega }; A_{\mathbb{R}}\times (A_{\mathbb{R}})^{d} \right)$, let $a= (a_i)_{1\leq i \leq d}: \mathbb R^d \times \mathbb R^d \times \mathbb R \times \mathbb R^d \to \mathbb R^d$ satisfy  \textbf{(H$_{1}$)}-\textbf{(H$_{5}$)} and \textbf{(H$_{7}$)}, we have that 
			\begin{itemize}
				\item[\textbf{(i)}] for each $1\leq i\leq
				d,$
				\begin{align*} b_{i}\left( f  ,\mathbf{f}%
					\right) \in \mathcal{C}\left( \overline{\Omega };\mathcal{X}_{A}^{\widetilde{\Phi }}\left( 
					\mathbb{R}
					_{y}^{d},\mathcal{B}\left( 
					\mathbb{R}
					_{z}^{d}\right) \right) \right) 
					\;\hbox{ and}
					\\
					a_{i}^{\varepsilon }\left( \cdot, \cdot,f ^{\varepsilon }(\cdot, \cdot,\cdot),\mathbf{f }%
					^{\varepsilon }(\cdot, \cdot,\cdot)\right) \rightharpoonup b_{i}\left( f  ,%
					\mathbf{f } \right) %
					\nonumber\\
					\text{ in }L^{\widetilde{\Phi }}(\Omega)\text{-weak } R\, \Sigma, \text{ as }\varepsilon \rightarrow 0.  \nonumber 
				\end{align*}
				
				\item[\textbf{(ii)}]
				The mapping $\left( f ,\mathbf{f }\right)
				\rightarrow b\left(f ,\mathbf{f }\right) = (b_{i}\left(f ,\mathbf{f }\right))_{1\leq i \leq d} $ from $\mathcal{C}%
				\left( \overline{\Omega } ; A_{\mathbb{R}}\right) \times \mathcal{C}%
				\left( \overline{\Omega } ; A_{\mathbb{R}}\right)^{d}$ into 
				$L^{\widetilde{\Phi }}_{\mathbb{R}}\left( \Omega ; \mathcal{X}_{A}^{\widetilde{\Phi }}\left( 
				\mathbb{R}
				_{y}^{d},\mathcal{B}\left( 
				\mathbb{R}
				_{z}^{d}\right) \right) \right) ^{d}$ extends by continuity to a (not relabeled)
				mapping from $L^{\Phi }_{\mathbb{R}}\left( \Omega ; \mathcal{X}_{A}^{\Phi}\left( 
				\mathbb{R}
				_{y}^{d},\mathcal{B}\left( 
				\mathbb{R}
				_{z}^{d}\right) \right) \right) ^{d+1}$ to  $L^{\widetilde{\Phi }}_{\mathbb{R}}\left( \Omega ; \mathcal{X}_{A}^{\widetilde{\Phi }}\left( 
				\mathbb{R}
				_{y}^{d},\mathcal{C}_b\left( 
				\mathbb{R}
				_{z}^{d}\right) \right) \right) ^{d}$, and there exist positive constants $c, c'$ such that: 
				\begin{equation*}
					\begin{array}{rcl}
						\left\Vert b\left(u ,\mathbf{u }\right) - b\left(v ,\mathbf{v }\right)
						\right\Vert _{L^{\widetilde{\Phi }}_{\mathbb{R}}\left( \Omega ; \mathcal{X}_{A}^{\widetilde{\Phi }}\left( 
							\mathbb{R}
							_{y}^{d},\mathcal{B}\left( 
							\mathbb{R}
							_{z}^{d}\right) \right) \right) ^{d}} & \leq  &	c\left\Vert u-v\right\Vert _{L^{\Phi }_{\mathbb{R}}\left( \Omega ; \mathcal{X}_{A}^{\Phi}\left( 
							\mathbb{R}
							_{y}^{d},\mathcal{B}\left( 
							\mathbb{R}
							_{z}^{d}\right) \right) \right) }^{\alpha }
						\\ 
						& & + c^{\prime }\left\Vert \mathbf{u -v}%
						\right\Vert _{L^{\Phi }_{\mathbb{R}}\left( \Omega ; \mathcal{X}_{A}^{\Phi}\left( 
							\mathbb{R}
							_{y}^{d},\mathcal{B}\left( 
							\mathbb{R}
							_{z}^{d}\right) \right) \right)^{d}}^{\beta },
					\end{array}
				\end{equation*}
			\end{itemize}
			and 
			\begin{equation*}
				b\left(u,\mathbf{u} \right) \cdot\mathbf{u}\geq \theta \Phi \left(
				\left\vert \mathbf{u}\right\vert \right) \hbox{ a.e. in } \Omega \times \mathbb{R}^{d}\times\mathbb{R}
				^{d} \label{3.5}
			\end{equation*}
			and 
			\begin{align}\label{tq}
				\left(b\left(u,\mathbf{u} \right) -b\left(v,\mathbf{v} \right) \right)\cdot\left( 
				\mathbf{u}-\mathbf{v}\right) \geq 0
				\text{ a.e in }\Omega \times \mathbb{R}^{d}\times \mathbb{R}^{d}, 
			\end{align}%
			for all $u,v\in L^{\Phi }_{\mathbb{R}}\left( \Omega ; \mathcal{X}_{A}^{\Phi}\left( 
			\mathbb{R}
			_{y}^{d},\mathcal{B}\left( 
			\mathbb{R}
			_{z}^{d}\right) \right) \right) ,\mathbf{u},\mathbf{v}\in L^{\Phi }_{\mathbb{R}}\left( \Omega ; \mathcal{X}_{A}^{\Phi}\left( 
			\mathbb{R}
			_{y}^{d},\mathcal{B}\left( 
			\mathbb{R}
			_{z}^{d}\right) \right) \right) ^{d},$ (with $\mathbf{u}\neq \mathbf{v}$ in \eqref{tq}), where $\alpha ,\beta$ 
			and $\theta $ are defined as in Proposition \ref{propalphabeta} and in \eqref{theta}, respectively.
		\end{proposition}

		In the same spirit of \cite[Corollary 4.2 and 4.3]{tacha4}, replacing the periodic algebras $\mathcal{C}_{per}(Y\times Z), \mathcal{C}_{per}(Y), \mathcal{C}_{per}(Z)$ by the $H$-algebras $A, A_{y}$ and $A_{z}$, respectively, and referring to \cite[Corollary 3.5 and 3.6]{nnang2014deterministic},  we have the following two corollaries.
		
		\begin{corollary}\label{cor3.2}
			Under the same assumptions of Proposition \ref{prop3.3}, take $\left(
			u_{\varepsilon }\right)_\varepsilon \subset L
			^{\Phi }_{\mathbb{R}}\left( \Omega \right) $ such that $u_{\varepsilon} \to u_0$
			in
			$L^{\Phi }\left( \Omega \right) $ as $ \varepsilon \rightarrow 0.$ Then,  for each $1\leq i\leq d$ and 
			for every $\mathbf{f}\in \mathcal{C}\left( \overline{\Omega }; A_{\mathbb{R}} ^{d}\right) $ we
			have 
			\begin{equation*}
				a_{i}^{\varepsilon }\left( \cdot, \cdot,u_{\varepsilon },\mathbf{f}^{\varepsilon
				}\right) \rightharpoonup b_{i}\left( u_{0},\mathbf{f}\right) \text{ in }L^{%
					\widetilde{\Phi }}(\Omega)\text{-weak } R\,\Sigma, \text{as } \varepsilon \rightarrow 0.
			\end{equation*}
		\end{corollary}

		\begin{corollary}\label{cor3.4}
			Under the assumptions of Proposition \ref{prop3.3}, consider $\left(
			u_{\varepsilon}\right)_\varepsilon \subset L^{\Phi }\left( \Omega \right) $ such that $u_{\varepsilon
			}\rightarrow u_{0}$ in $L^{\Phi }\left( \Omega \right) $ as $\varepsilon \rightarrow 0.$ Let $\phi _{\varepsilon
			}\left( x\right): =\varphi _{0}\left( x\right) +\varepsilon \varphi
			_{1}\left( x,\frac{x}{\varepsilon }\right) +\varepsilon
			^{2}\varphi _{2}\left( x,\frac{x}{\varepsilon },\frac{x}{\varepsilon ^{2}}\right) ,$ $x\in \Omega ,\varphi _{0}\in \mathcal{D}_{\mathbb{R}}\left(
			\Omega \right) ,\varphi _{1}\in \mathcal{D}_{\mathbb{R}}\left( \Omega \right) \otimes 
			\,	_{\mathbb{R}}A_{y}^{\infty} ,\varphi _{2}\in \mathcal{D}_{\mathbb{R}}\left( \Omega
			\right) \otimes\, _{\mathbb{R}}A_{y}^{\infty} \otimes\, _{\mathbb{R}}A_{z}^{\infty}$ ($
			\phi _{\varepsilon }=\varphi _{0}+\varepsilon \varphi _{1}^{\varepsilon
			}+\varepsilon ^{2}\varphi _{2}^{\varepsilon },$  shortly and with $_{\mathbb{R}}A_{y}^{\infty} = A_{y}^{\infty} \cap \mathcal{C}(\mathbb{R}^{d};\mathbb{R})$ and a similar definition for $_{\mathbb{R}}A_{z}^{\infty}$).  Then, 
			\begin{align*}
				a_{i}^{\varepsilon }\left( \cdot, \cdot,u_{\varepsilon },D\phi _{\varepsilon }\right)
				\rightharpoonup b_{i}\left( u_{0},D\varphi _{0}+ \overline{D}_{y} \varrho \varphi
				_{1} +\overline{D}_{z} \varrho \varphi
				_{2} \right) \text{ in }L^{\widetilde{\Phi }}(\Omega)
				\text{-weak } R\, \Sigma
			\end{align*}%
			as $\varepsilon \rightarrow 0$, for any $1\leq i\leq d$ and where	$\overline{D}_{y} \varrho \varphi
			_{1}= (\overline{\partial} \varrho \varphi_{1}/\partial y_{i})_{1\leq i\leq d}$,  (the same definition for $\overline{D}_{z} \varrho \varphi
			_{2}$).
			
			Furthermore, given a sequence $\left(v_{\varepsilon }\right)_\varepsilon \subset L^{\Phi }\left( \Omega \right) ,$ such that $v_{\varepsilon }\rightharpoonup
			v_{0}$ in $L^{\Phi }\left( \Omega \right)$-weak $R\,\Sigma$ as $\varepsilon
			\rightarrow 0,$ one has
			\begin{align*}
				\underset{\varepsilon \rightarrow 0}{\lim }\int_{\Omega
				}a_{i}^{\varepsilon }\left( \cdot, \cdot,u_{\varepsilon },D\phi _{\varepsilon
				}\right) v_{\varepsilon }dx = \\ 
				\int_{\Omega } \pounds_{A}\iint_{\mathbb{R}^{2d}} b_{i}\left( u_{0},D\varphi _{0}+ \overline{D}_{y} \varrho \varphi
				_{1} +\overline{D}_{z} \varrho \varphi
				_{2} \right) v_{0}dx dy dz
			\end{align*}
			for any $1\leq i\leq d.$%
		\end{corollary}
		
		The next section is devoted to the main homogenization result for problem \eqref{1.1}.

		\subsection{Astract homogenization Result for \eqref{1.1} }\label{sub32}
		
		Let the basic notation be as above. Let $\Phi\cap \Delta'$ be an $N$-function such that $\widetilde{\Phi} \in \Delta_{2}$ and let $A= A_{y}\odot A_{z}$ be an $H$-supralgebra where $A_{y}$ (resp. $A_{z}$) is an ergodic $H$-supralgebra on $\mathbb{R}^{d}_{y}$ (resp. $\mathbb{R}^{d}_{z}$). Assume that  $A_{y}$ and $A_{z}$ are translation invariant, and moreover that their elements are uniformly continuous. We also assume that  $A$ is of class $\mathcal{C}^{\infty}$. Note that all the above assumptions implies that the $RH$-supralgebra $A$ is $W^{1}_{0}L^{\Phi}(\Omega)$-proper (see Definition \ref{defproperness}).
		We set 
		\begin{equation*}
			\mathbb{F}_{0}^{1,\Phi }:=W_{0}^{1}L^{\Phi}\left( \Omega
			\right) \times L^\Phi_{\mathbb{R}}\left( \Omega ;W_{\#}^{1}\mathcal{X}^{\Phi}_{A_{y}} \right) \times L^\Phi_{\mathbb{R}}\left( \Omega ; \mathcal{X}_{A_{y}}^{\Phi}\left(
			\mathbb{R}^{d}_{y} ; W_{\#}^{1}\mathcal{X}^{\Phi}_{A_{z}} \right) \right)
		\end{equation*}
		and 
		\begin{eqnarray*}
			F^{\infty }_{0}:=\mathcal{D}_{\mathbb{R}}\left( \Omega \right)
			\times \left[ \mathcal{D}_{\mathbb{R}}\left( \Omega \right) \otimes J_{1}^{y}\circ \varrho(_{\mathbb{R}}A^{\infty}_{y})/\mathbb{C}\right] \times \nonumber \\
			\left[ \mathcal{D}_{\mathbb{R}}\left( \Omega \right) \otimes
			_{\mathbb{R}}A^{\infty}_{y} \otimes J_{1}^{z}\circ \varrho(_{\mathbb{R}}A^{\infty}_{z})/\mathbb{C}\right],
		\end{eqnarray*}
		where $J_{1}^{y}$ (resp. $\varrho $)
		denotes the canonical mapping of $W^{1}\mathcal{X}_{A_{y}}^{\Phi}/\mathbb{C}$ (resp. $%
		\mathfrak{X}_{A_{y}}^{\Phi}$) into its separated completion $%
		W^{1}_{\#}\mathcal{X}^{\Phi}_{A_{y}}$ (resp. $\mathcal{X}_{A_{y}}^{\Phi}$) and subscript $_{\mathbb{R}}$ stands the functions with real values (of course, all that remaining valid with $z$ in place of index $y$).
		
		We  provide $\mathbb{F}_{0}^{1,\Phi }$ with the norm
		\begin{eqnarray*}
			\left\Vert u\right\Vert _{\mathbb{F}_{0}^{1,\Phi }}:=\left\Vert
			D_{x}u_{0}\right\Vert _{\Phi ,\Omega }+\left\Vert \overline{D}_{y}u_{1}\right\Vert
			_{L^\Phi\left( \Omega ;\mathcal{X}^{\Phi}_{A_{y}} \right)}+\left\Vert \overline{D}_{z}u_{2}\right\Vert _{L^\Phi\left( \Omega ; \mathcal{X}_{A_{y}}^{\Phi}\left(
				\mathbb{R}^{d}_{y} ; \mathcal{X}^{\Phi}_{A_{z}} \right) \right)}, \nonumber \\
			\text{ } u=\left( u_{0},u_{1},u_{2}\right) \in \mathbb{F}%
			_{0}^{1,\Phi },
		\end{eqnarray*}
		which makes it a Banach space, and since we have $A_{y}$ and  $A_{z}$ $\Phi$-total, one can easily show that the space $%
		F^{\infty }_{0}$ is dense in $\mathbb{F}_{0}^{1,\Phi }$ (see, \cite[Sect. 3.2]{nnang2014deterministic} for nonreiterated case), where, in terms of notation, if  
		\begin{align}
			\left( \psi _{0},\psi _{1},\psi _{2}\right) \in
			\mathcal{D}_{\mathbb{R}}\left( \Omega \right)
			\times \left[ \mathcal{D}_{\mathbb{R}}\left( \Omega \right) \otimes J_{1}^{y}\circ \varrho(_{\mathbb{R}}A^{\infty}_{y})/\mathbb{C}\right] \times \nonumber \\
			\left[ \mathcal{D}_{\mathbb{R}}\left( \Omega \right) \otimes
			_{\mathbb{R}}A^{\infty}_{y} \otimes J_{1}^{z}\circ \varrho(_{\mathbb{R}}A^{\infty}_{z})/\mathbb{C}\right], \label{3.57}
		\end{align}
		then $F^{\infty }_{0}=\left\{ \phi , \;\phi :=\left( \psi _{0},\psi _{1},\psi
		_{2}\right) \text{ in \eqref{3.57}}\right\}.$ \\
		
		We are in position of proving the existence Theorem \ref{mainresult}, which we restate for the reader's convenience. \\
		
		\noindent \textbf{Theorem  \ref{mainresult}}
		{\it 		Let \eqref{1.1} be the problem under hypotheses in Subsection \ref{hypoproblem}, with $a$ and $f$ satisfying \textbf{(H$_{1}$)}-\textbf{(H$_{5}$)}. Suppose that $A=A_{y}\odot A_{z}$ is an $RH$-supralgebra on $\mathbb{R}^{d}_{y}\times \mathbb{R}^{d}_{z}$ such that the \textit{abstract structure hypothesis} \textbf{(H$_{7}$)} holds.
			We also assume that $A_{y}$ (resp. $A_{z}$) is ergodic, translation invariant, and moreover that their elements are uniformly continuous.
			
			For each $\varepsilon
			>0$, let $u_{\varepsilon }$ be a solution of \eqref{1.1}. Then there exists a not relabeled subsequence  and $\mathbf{u}:=\left( u_{0},u_{1},u_{2}\right)\in \mathbb{F}_{0}^{1,\Phi }:= \in W^{1}_{0}L^{\Phi}(\Omega) \times L^{\Phi}(\Omega; W^{1}_{\#}\mathcal{X}^{\Phi}_{A_{y}})\times L^{\Phi}(\Omega; \mathcal{X}^{\Phi}_{A_{y}}(\mathbb{R}^{d}_{y} ; W^{1}_{\#}\mathcal{X}^{\Phi}_{A_{z}}))$ such that
			\begin{equation*}
				u_{\varepsilon} \rightharpoonup u_{0} \quad in \;\, W^{1}_{0}L^{\Phi}(\Omega)\textup{-}weak 
			\end{equation*}
			\begin{equation*}
				\dfrac{\partial u_{\varepsilon}}{\partial x_{j}} \rightharpoonup \dfrac{\partial u_{0}}{\partial x_{j}} + \dfrac{\overline{\partial} u_{1}}{\partial y_{j}} + \dfrac{\overline{\partial} u_{2}}{\partial z_{j}} \quad in \;\, L^{\Phi}(\Omega)\textup{-}weak \; R\, \Sigma \quad (1 \leq j \leq d).
			\end{equation*}
			and $\mathbf{u}$ 
			solves the variational homogenized problem
			\begin{equation*}
				\left\{ 
				\begin{tabular}{l}
					$ \displaystyle \int_{\Omega } \pounds_{A}\iint_{\mathbb{R}^{2d}} b\left( u_{0},Du_{0}+ \overline{D}_{y} u
					_{1} +\overline{D}_{z} u
					_{2} \right)\cdot \left(Dv_{0}+ \overline{D}_{y} v
					_{1} +\overline{D}_{z} v
					_{2} \right) dx dy dz
					$ \\
					\\ 
					$ \displaystyle =\int_{\Omega }fv_{0}dx,$ for all $v=\left( v_{0},v_{1},v_{2}\right) \in 
					\mathbb{F}_{0}^{1,\Phi }$. \\ 
				\end{tabular}%
				\right.  
			\end{equation*}
			where $b=(b_{i})$, $1\leq i\leq d$ is defined as in \eqref{tp}.
		}
		
		\begin{proof}
			Observing that $\underset{0<\varepsilon \leq 1}{\sup }\left\Vert u_{\varepsilon
			}\right\Vert _{W_{0}^{1}L^{\Phi }\left( \Omega
				\right) }<\infty$. Thus, given an arbitrary fundamental sequence $E$, using Theorem \ref{ti} one can extract a
			subsequence $E'$ from $E$ and some triple $\mathbf{u}:=\left( u_{0},u_{1},u_{2}\right) \in \mathbb{F}_{0}^{1,\Phi }$ such that \eqref{3.58} and \eqref{3.59} hold.
			
			We show that $\mathbf{u}:=\left( u_{0},u_{1},u_{2}\right) $ defined by \eqref{3.58} and 
			\eqref{3.59} is a solution of \eqref{3.60}. 
			To this end, for  $\phi :=\left( \psi
			_{0}, J_{1}^{y}(\varrho \psi _{1}), J_{1}^{z}(\varrho \psi _{2})\right) \in F_{0}^{\infty },$ consider $\phi
			_{\varepsilon }\left( x\right) =\psi _{0}\left( x\right) +\varepsilon \psi
			_{1}\left( x,\frac{x}{\varepsilon }\right) +\varepsilon
			^{2}\psi _{2}\left( x,\frac{x}{\varepsilon },\frac{x}{%
				\varepsilon ^{2}}\right) $. We have 
			\begin{align*}
				\int_{\Omega }a^{\varepsilon }\left(
				\cdot, \cdot,u_{\varepsilon },Du_{\varepsilon }\right) \cdot Du_{\varepsilon
				}dx=\int_{\Omega }fu_{\varepsilon }dx;
				\\
				\int_{\Omega }a^{\varepsilon }\left( \cdot, \cdot,u_{\varepsilon },Du_{\varepsilon
				}\right)  \cdot D\phi _{\varepsilon }dx=\int_{\Omega }f\phi _{\varepsilon }dx.
			\end{align*}
			Thus, by \textbf{(H$_{4}$)},
			\begin{align*}0\leq \int_{\Omega }\left( a^{\varepsilon }\left( \cdot, \cdot,u_{\varepsilon
				},Du_{\varepsilon }\right) -a^{\varepsilon }\left( \cdot, \cdot,u_{\varepsilon
				},D\phi _{\varepsilon }\right) \right) \cdot\left( Du_{\varepsilon }-D\phi
				_{\varepsilon }\right) dx=\int_{\Omega }f\left( u_{\varepsilon }-\phi
				_{\varepsilon }\right) dx
				\\
				-\int_{\Omega }a^{\varepsilon }\left( \cdot, \cdot, u_{\varepsilon}, D\phi
				_{\varepsilon }\right) \cdot\left( Du_{\varepsilon }-D\phi _{\varepsilon
				}\right) dx,
			\end{align*}
			with $f\in L^{\widetilde{\Phi }}\left( \Omega \right) \subset L^{%
				\widetilde{\Phi }}\left( \Omega ; A
			\right)$ and
			$\phi _{\varepsilon }\rightarrow \psi _{0}$ in $L^{\Phi }\left( \Omega
			\right) $, as $\varepsilon\to 0$ (and so the convergence to $\psi_0$ is also in the strong reiterated $\Sigma$-convergence setting). 
			Moreover, since $u_{\varepsilon }\rightharpoonup u_{0}$ in $W_{0}^{1}L^{\Phi }\left( \Omega \right)$-weak as $\varepsilon\to 0$, and the embedding $W_{0}^{1}L^{\Phi }\left( \Omega
			\right) \hookrightarrow L^{\Phi }\left( \Omega\right) $ is compact, then $u_{\varepsilon }\rightarrow u_{0}$ in $L^{\Phi
			}\left( \Omega\right) $ strongly as $\varepsilon\to 0$ and in the reiterated $\Sigma$-convergence setting. 
			
			Therefore, 
			\begin{align*}
				\int_{\Omega }f\left( u_{\varepsilon }-\phi _{\varepsilon }\right)
				dx\rightarrow \int_{\Omega }\pounds_{A}\iint_{\mathbb{R}^{2d}} f\left( u_{0}-\psi _{0}\right)
				dxdydz =\int_{\Omega }f\left( u_{0}-\psi _{0}\right) dx.
			\end{align*}
			
			By Corollary \ref{cor3.4}, it results that:
			\begin{align*}
				\int_{\Omega }a^{\varepsilon }\left( \cdot, \cdot, u
				_{\varepsilon} ,D\phi _{\varepsilon }\right) \cdot \left( Du_{\varepsilon }-D\phi
				_{\varepsilon }\right) dx\rightarrow\\
				\int_{\Omega }\pounds_{A}\iint_{\mathbb{R}^{2d}} b\left( u_{0},D\psi _{0}+\overline{D}_{y}\varrho \psi
				_{1}+\overline{D}_{z}\varrho \psi _{2}\right)\cdot\\
				\left( D\left( u_{0}-\psi _{0}\right) +\overline{D}_{y}\left( u_{1}- \varrho\psi _{1}\right)
				+\overline{D}_{z}\left( u_{2}- \varrho\psi _{2}\right) \right) dxdydz,
			\end{align*}
			as $\varepsilon \to 0$.
			Thus 
			\begin{align*}
				0\leq \int_{\Omega
				}f\left( u_{0}-\psi_{0}\right) dx -\int_{\Omega }\pounds_{A}\iint_{\mathbb{R}^{2d}} b\left( u_{0},D\psi _{0}+\overline{D}_{y}\varrho\psi _{1}+\overline{D}_{z}\varrho\psi
				_{2}\right)\cdot
				\\
				\left( D\left( u_{0}-\psi _{0}\right) +\overline{D}_{y}\left( u_{1}- \varrho\psi _{1}\right)
				+\overline{D}_{z}\left( u_{2}-\varrho\psi _{2}\right) \right) dxdydz.
			\end{align*} 
			Using density of $
			F_{0}^{\infty }$ in $\mathbb{F}_{0}^{1,\Phi }$  the result remains valid
			for all $\phi \in \mathbb{F}_{0}^{1,\Phi }$. 
			In particular, choosing $\phi
			:=\mathbf{u}-t\mathbf{v},\, \mathbf{v}=\left( v_{0},v_{1},v_{2}\right) \in \mathbb{F}_{0}^{1,\Phi }$, and dividing
			by $t>0$ we get:
			\begin{align*}
				0\leq \int_{\Omega }fv_{0}dx-\int_{\Omega }\pounds_{A}\iint_{\mathbb{R}^{2d}} b\left(
				u_{0},\mathbb{D}\mathbf{u}-t\mathbb{D}\mathbf{v}\right) \cdot \mathbb{D}vdxdydz,
			\end{align*} where 
			$\mathbb{D}\mathbf{v}=Dv_{0}+\overline{D}_{y}v_{1}+\overline{D}_{z}v_{2}.$ 
			
			Using the continuity of $a$ in its
			last argument and passing to the limit as $t\rightarrow 0,$ we are led
			to 
			\begin{align*}0\leq \int_{\Omega }fv_{0}dx-\int_{\Omega }\pounds_{A}\iint_{\mathbb{R}^{2d}} b\left(
				u_{0},\mathbb{D}\mathbf{u}\right) \cdot \mathbb{D}\mathbf{v} dxdydz,
			\end{align*} that is 
		for all $\mathbf{v}\in \mathbb{F}_{0}^{1,\Phi }.$
		Thus, if we replace $%
		\mathbf{v}=\left( v_{0},v_{1},v_{2}\right) \in \mathbb{F}_{0}^{1,\Phi }$ by $%
		\mathbf{v}^{1}=\left( -v_{0},-v_{1},-v_{2}\right) $ we deduce:
		\begin{align*}
			-\int_{\Omega }fv_{0}dx\geq \int_{\Omega }\pounds_{A}\iint_{\mathbb{R}^{2d}} b\left(
			u_{0},\mathbb{D}\mathbf{u}\right) \cdot\mathbb{D}\mathbf{v}^1dxdydz, \quad \text{i.e.};
			\\
			-\int_{\Omega }fv_{0}dx\geq -\int_{\Omega }\pounds_{A}\iint_{\mathbb{R}^{2d}} b\left(
			u_{0},\mathbb{D}\mathbf{u}\right) \cdot\mathbb{D}\mathbf{v} dxdydz
		\end{align*}
		thus the equality follows, i.e.
		$\mathbf{u}=\left( u_{0},u_{1},u_{2}\right) $ verifies \eqref{3.60}. 
	\end{proof}
	
	The following remark discuss about the uniqueness of solution for \eqref{3.60}.
	
	\begin{remarq}\label{remuniq}
		\textup{	Assuming \textbf{(H$_{6}$)}. 
			Let $\mathbf{w}=\left(
			w_{0},w_{1},w_{2}\right) \in \mathbb F^{1, \Phi}_0$ another solution of \eqref{3.60}, then }
		\begin{align*}
			\int_{\Omega }fw_{0}dx=\int_{\Omega }\pounds_{A}\iint_{\mathbb{R}^{2d}} b\left( u_{0},%
			\mathbb{D}\mathbf{u}\right) \cdot \mathbb{D}\mathbf{w}dxdydz,
			\\
			\int_{\Omega }fu_{0}dx=\int_{\Omega }\pounds_{A}\iint_{\mathbb{R}^{2d}} b\left( w_{0},%
			\mathbb{D}\mathbf{w}\right) \cdot \mathbb{D}\mathbf{u}dxdydz,
			\\
			-\int_{\Omega }fu_{0}dx=-\int_{\Omega }\pounds_{A}\iint_{\mathbb{R}^{2d}} b\left(u_{0},%
			\mathbb{D}\mathbf{u}\right) \cdot \mathbb{D}\mathbf{u}dxdydz,
			\\
			-\int_{\Omega }fw_{0}dx=-\int_{\Omega }\pounds_{A}\iint_{\mathbb{R}^{2d}} b\left(w_{0},%
			\mathbb{D}\mathbf{w}\right) \cdot \mathbb{D}\mathbf{w}dxdydz,
		\end{align*}
		\textup{	Consequently, by \textbf{(H$_{6}$)} }
		\begin{align*}
			0=\int_{\Omega }\pounds_{A}\iint_{\mathbb{R}^{2d}} \left( b\left(w_{0},\mathbb{D}\mathbf{w}\right) -b\left(
			u_{0},\mathbb{D}\mathbf{u}\right) \right) .\left( \mathbb{D}\mathbf{w}-\mathbb{D}\mathbf{u}\right)
			dxdydz\geq \\
			\gamma \int_{\Omega }\pounds_{A}\iint_{\mathbb{R}^{2d}} \Phi \left( \left\vert \mathbb{D}\mathbf{w}-
			\mathbb{D}\mathbf{u}\right\vert \right) dxdydz.
		\end{align*}
		\textup{	The fact that $\mathbf{w}= \mathbf{u}$ follows by mere routine (see, e.g., \cite[Remark 4.4]{tacha4}).}
	\end{remarq}
	
	\begin{remarq}
		\textup{	Since the variational problem \eqref{3.60} admits a unique solution, then as mentioned in the proof of \cite[Theorem 3.7]{nnang2014deterministic} it will turn out that \eqref{3.58} and \eqref{3.59} hold not only when $E'\ni \varepsilon \to 0$ but merely when $0< \varepsilon \to 0$. }
	\end{remarq}
	
	It is worth to point out that equation  \eqref{3.60} is referred as \textit{global
		homogenization problem} for \eqref{1.1} and as in reiterated periodic case (see \cite[Sect. 4]{tacha4}), it is equivalent to the
	following three systems:
	\begin{align}
		\int_{\Omega }\pounds_{A}\iint_{\mathbb{R}^{2d}}  b\left(
		u_{0},Du_{0}+\overline{D}_{y}u_{1}+\overline{D}_{z}u_{2}\right) \cdot \overline{D}_{z}v_{2}dxdydz=0 \nonumber \\ 
		\hbox{for all }v_{2}\in L^\Phi\left( \Omega ;\mathcal{X}_{A_{y}}^\Phi\left(
		\mathbb{R}^{d}_{y};W_{\#}^{1}\mathcal{X}_{A_{z}}^{\Phi} \right) \right),
		\label{3.64}
		\\
		\int_{\Omega }\pounds_{A_{y}}\int_{\mathbb{R}^{d}} \left( \pounds_{A_{z}}\int_{\mathbb{R}^{d}} b\left(
		u_{0},Du_{0}+\overline{D}_{y}u_{1}+\overline{D}_{z}u_{2}\right) dz\right) \cdot \overline{D}_{y}v_{1}dxdy=0 \nonumber \\ 
		\hbox{ for all }v_{1}\in L^\Phi\left( \Omega ;W_{\#}^{1}\mathcal{X}^{\Phi}_{A_{z}} \right), 
		\label{3.65}
		\\
		\int_{\Omega }\left( \pounds_{A}\iint_{\mathbb{R}^{2d}}  b\left(u_{0},Du_{0}+\overline{D}_{y}u_{1}+\overline{D}_{z}u_{2}\right) dydz\right)\cdot
		Dv_{0} dx \nonumber \\ 
		=\int_{\Omega }fv_{0}dx, \hbox{ for all }v_{0}\in W_{0}^{1}L^{\Phi }\left( \Omega
		\right).
		\label{3.66}
	\end{align}

	Now we are in position to derive a \textit{macroscopic homogenized problem}.
	
	Hence, let  $r \in \mathbb R$ and $\xi \in
	\mathbb{R}
	^{d}$ be arbitrarily fixed. For a.e.  $y \in \mathbb{R}
	^{d}_{y}$, consider the \textit{variational cell problem} in \eqref{3.67}, whose solution is denoted by $\pi_2(y, r,\xi)$, i.e.
	\begin{equation}\label{1.12}
		\left\{ 
		\begin{tabular}{l}
			$\displaystyle \hbox{find } \pi _{2}\left(y,r, \xi \right) \in W_{\#}^{1}\mathcal{X}^{\Phi}_{A_{z}} $ \hbox{such that} \\ 
			$\displaystyle \pounds_{A_{z}} \int_{\mathbb{R}^{d}} b\left( r,\xi +\overline{D}_{z}\pi _{2}\left( y,r,\xi \right) \right)
			\cdot \overline{D}_{z}\theta dz=0$, for all $\theta \in W_{\#}^{1}\mathcal{X}^{\Phi}_{A_{z}}.$%
		\end{tabular}%
		\right. 
	\end{equation}%
	It is a fact that, this problem has a solution (arguing as in \cite[Theorem 3.2]{Y}), unique under assumption \textbf{(H$_{6}$)}.
	
	\noindent	Comparing \eqref{3.66} with (\ref{1.12}) for  $r= u_0(x)$ and $\xi =Du_{0}\left( x\right)
	+\overline{D}_{y}u_{1}\left( x,y\right) $ we can consider 
	$$
	\Omega \times 
	\mathbb{R}
	_{y}^{d} \ni \left( x,y\right) \rightarrow \pi _{2}\left(y, u_0(x), Du_{0}\left( x\right)
	+\overline{D}_{y}u_{1}\left( x,y\right) \right)  \in W_{\#}^{1}\mathcal{X}^{\Phi}_{A_{z}}, $$
	Hence, defining  for a.e. $y \in \mathbb{R}^{d}_{y}$, and for any $(r, \xi) \in \mathbb R \times \mathbb R^d$, $h$ as in \eqref{h}, namely  
	\begin{equation*}h\left(y,r, \xi \right) :=\pounds_{A_{z}}\int_{\mathbb{R}^{d}} b_{i}\left(r ,\xi
		+\overline{D}_{z}\pi _{2}\left(y,r, \xi \right) \right) dz,
	\end{equation*}
	\eqref{3.65} becomes
	\begin{align*}
		\int_{\Omega }\pounds_{A_{y}}\int_{\mathbb{R}^{d}}  h\left(y, u_0(x), Du_{0}(x)+\overline{D}_{y}u_{1}(x,y)\right) \cdot \overline{D}_{y}v_{1}dxdy=0,
	\end{align*}
	for all  $v_{1}\in L^{\Phi}\left( \Omega ;W_{\#}^{1}\mathcal{X}^{\Phi}_{A_{y}} \right).$
	
	\noindent Consequently, in analogy with the previous steps, one can consider, for any $(r,\xi) \in \mathbb R\times
	\mathbb{R}
	^{d},$ the function $\pi _{1}\left(r, \xi \right) \in
	W_{\#}^{1}\mathcal{X}^{\Phi}_{A_{y}} $ solution of the variational problem in \eqref{3.69b},  (unique by \textbf{(H$_{6}$)}) i.e.
	\begin{equation*}
		\left\{ 
		\begin{tabular}{l}
			\hbox{ find} $\pi _{1}\left(r, \xi \right) \in W_{\#}^{1}\mathcal{X}^{\Phi}_{A_{y}} $ \hbox{ such that } \\ 
			$\pounds_{A_{y}} \int_{\mathbb{R}^{d}} h\left(y, r, \xi +\overline{D}_{y}\pi _{1}\left(r, \xi \right) \right) \cdot \overline{D}_{y}\theta
			dy=0$ for all $\theta \in W_{\#}^{1}\mathcal{X}^{\Phi}_{A_{y}}. $%
		\end{tabular}%
		\right.  \label{3.69}
	\end{equation*}
	Note also that (\ref{3.65}) lead us to $u_{1}=\pi _{1}\left(u_0,
	Du_{0}\right) .$ Set, again, for $(r,\xi) \in \mathbb R\times \mathbb R^d$
	\begin{align*}
		q\left( r,\xi \right) =\pounds_{A_{y}}\int_{\mathbb{R}^{d}} h\left(y,r, \xi +\overline{D}_{y}\pi
		_{1}\left( r,\xi \right) \right) dy,
	\end{align*}
	the function $q$ in \eqref{q} is well defined.
	Moreover, it results from \eqref{3.66} and the above cell problems that 
	\begin{align}
		\int_{\Omega }q\left( u_0, Du_{0}\right) \cdot Dv_{0}dx=\int_{\Omega }f\cdot v_{0}dx\text{
			for all }v_{0}\in W_{0}^{1}L^{\Phi }\left( \Omega \right) .  \label{3.71}
	\end{align}
	
	Thus we are in position to establish  the generalization to deterministic setting of macroscopic problem in \cite[Theorem 1.2]{tacha4} and whose the proof is entirely similar. This is Theorem \ref{maincor} that we summarize here for the reader's convenience. \\
	
	\noindent {\bf Theorem \ref{maincor}} {\it 	Under hypothesis of Theorem \ref{mainresult}, for every $\varepsilon >0$, let \eqref{1.1} be  such that $a$ and $f$ satisfy \textbf{(H$_{1}$)}-\textbf{(H$_{6}$)}.  
		
		Let $u_0 \in W_{0}^{1}L^{\Phi }(\Omega)$ be the solution  defined by means of \eqref{3.60}. Then, it is the unique solution of the
		macroscopic homogenized problem 
		\begin{equation*}
			\left\{ \begin{array}{l}
				-{\rm div}\, q\left(u_0, Du_{0}\right) =f\text{ in }\Omega,  \\
				
				\\
				u_{0}\in W_{0}^{1}L^{\Phi
				}\left( \Omega \right),  
			\end{array}\right.
		\end{equation*}
		where $q$ is defined by \eqref{3.71}, taking into account \eqref{h}, \eqref{3.67} and \eqref{3.69b}.} \\

	Next, we explore some concrete homogenization problem for \eqref{1.1}, i.e. under different \textit{abstract structure hypothesis} using $RH$-algebras of Example \ref{cb}. 
	
	
	\section{Concrete reiterated homogenization results for \eqref{1.1}} \label{labelsect5}
	
	In the present section, we consider a few examples of homogenization problems for \eqref{1.1} under different \textit{abstract structure hypothesis}. We will show how their study leads naturally to the abstract setting of Section \ref{labelsect4}  and so we may conclude by merely applying Theorem \ref{mainresult} (thanks to Remark \ref{algegood}).
	
	\subsection{Problem I: $a_{i}(\cdot,\cdot, \zeta,\lambda)\in \mathcal{C}_{per}(Y\times Z)$}
	
	We assume here that for each fixed $(\zeta,\lambda) \in \mathbb{R}\times\mathbb{R}^{d}$, the function $(y,z) \to a_{i}(y,z, \zeta,\lambda)$ satisfies the following condition, commonly know as the \textit{periodicity hypothesis}:
	\begin{equation}\label{ca}
		\begin{array}{l}
			\text{ For each } k \in \mathbb{Z}^{d} \text{ and each } l \in \mathbb{Z}^{d}, \text{ we have } \\
			a_{i}(y+k, z+l, \zeta,\lambda) = a_{i}(y,z,\zeta,\lambda) \text{ a.e. in } (y,z) \in \mathbb{R}^{d}\times\mathbb{R}^{d}.  
		\end{array}
	\end{equation}
	One also expresses \eqref{ca} by saying that $a_{i}(y,z,\zeta,\lambda)$ (for fixed $(\zeta,\lambda) \in \mathbb{R}\times\mathbb{R}^{d}$) is $Y$-periodic in $y$ and $Z$-periodic in $z$, or simply that $a_{i}(y,z,\zeta,\lambda)$ is $Y\times Z$-periodic in $(y,z)$, where $Y= (0,1)^{d}$ and $Z= (0,1)^{d}$ (see Example \ref{cb}).
	
	Of course, the right $RH$-supralgebra for this is the periodic $RH$-supralgebra $A = A_{y}\odot A_{z} = \mathcal{C}_{per}(Y\times Z)$  with $A_{y} = \mathcal{C}_{per}(Y)$ and $A_{z} = \mathcal{C}_{per}(Z)$  (see Example \ref{cb}). 
	
	our purpose is to study the homogenization of \eqref{1.1} under the periodicity hypothesis \eqref{ca} and $\widetilde{\Phi} \in \Delta'$. In this case, we will show that this problem becomes to that studied in \cite{tacha4}.
	
	In this setting, \eqref{3.60} takes a rather simple form:
	\begin{equation}
		\left\{ 
		\begin{tabular}{l}
			$\displaystyle \int_{\Omega } \int_{Y} \int_{Z} a\left(y,z, u_{0},Du_{0}+ D_{y} u
			_{1} +D_{z} u
			_{2} \right)\cdot \left(Dv_{0}+ D_{y} v
			_{1} +D_{z} v
			_{2} \right) dx dy dz
			$ \\
			\\ 
			$\displaystyle =\int_{\Omega }fv_{0}dx,$ for all $v=\left( v_{0},v_{1},v_{2}\right) \in 
			\mathbb{F}_{0}^{1,\Phi }$, \\ 
		\end{tabular}%
		\right.  \label{cc}
	\end{equation}
	where $ \mathbb{F}_{0}^{1,\Phi }:=W_{0}^{1}L^{\Phi}\left( \Omega
	\right) \times L^\Phi_{per}\left( \Omega ;W_{\#}^{1}L^{\Phi}\left( Y
	\right) \right) \times L^\Phi\left( \Omega ;L_{per}^{\Phi}\left(
	Y;W_{\#}^{1}L^{\Phi}\left( Z
	\right) \right) \right)$ with 
	$W_{\#}^{1}L^{\Phi}\left( Y
	\right) = \left\{ v \in W^{1}L^{\Phi}\left( Y
	\right), \;\; \int_{Y} v(y) dy =0 \right\}$ and \\ $L_{per}^{\Phi}\left(
	Y;W_{\#}^{1}L^{\Phi}\left( Z
	\right) \right) = \left\{ v \in L^{\Phi}_{per}(Y\times Z), \;\; v(y,\cdot) \in W_{\#}^{1}L^{\Phi}\left( Z
	\right) \text{ for a.e. } y\in Y \right\}$ and \\ $L^\Phi_{per}\left( \Omega ;W_{\#}^{1}L^{\Phi}\left( Y
	\right) \right) = \left\{ v \in L^\Phi\left( \Omega ;L^{\Phi}_{per}\left( Y
	\right) \right), \;\; u(x,\cdot) \in W_{\#}^{1}L^{\Phi}\left( Y
	\right), \text{ for a.e. } x\in \Omega \right\}$. 
	
	Given an $N$-function $\Phi \in \Delta_{2}$ such that $\widetilde{\Phi} \in \Delta'$, the aim is to show that as $\varepsilon\to 0$, we have $u_{\varepsilon }\rightharpoonup u_{0}\text{ in }W_{0}^{1}L^{\Phi }\left( \Omega
	\right)\text{-weak}$ and $	D_{x_{i}}u_{\varepsilon }\rightharpoonup
	D_{x_{i}}u_{0}+D_{y_{i}}u_{1}+D_{z_{i}}u_{2} 
	\text{ in }L^{\Phi }(\Omega)\text{-weak} R\,\Sigma, (1\leq i\leq N)$, where $u_{\varepsilon}$ is the solution of \eqref{1.1} (for fixed $\varepsilon>0$) and $\mathbf{u}= (u_{0}, u_{1}, u_{2})$ is (uniquely) defined by \eqref{cc}. 
	
	But, since $A_{y} = \mathcal{C}_{per}(Y)$ and $A_{z} = \mathcal{C}_{per}(Z)$ satisfy all the assumption of Theorem \ref{ti}, we see by Theorem \ref{mainresult} that the whole problem reduces to verifying that \eqref{tj} holds.
	\begin{proof}[Proof that \eqref{tj} holds]
		We arguing as in \cite[Sect. 4.1]{luka}. Let $1\leq i\leq d$ and $\psi\in (A_{\mathbb{R}})^{d}$ be freely fixed. Let $z\in \mathbb{R}^{d}$ be arbitrarily fixed. In view of (\textbf{H}$_{1}$) and \eqref{1.3}, the function $h: \mathbb{R}^{d}\times \mathbb{R}\times \mathbb{R}^{d} \rightarrow \mathbb{R}$ given by $h(y,\zeta,\lambda) = a_{i}(y,z,\zeta,\lambda)$, $(y,\zeta,\lambda) \in \mathbb{R}^{d}\times \mathbb{R}\times \mathbb{R}^{d}$ has the Carath\'{e}odory property. However, if $(u, \mathbf{u})$ is any measurable function from $\mathbb{R}^{d}\times \mathbb{R}^{d}$ into $\mathbb{R}\times \mathbb{R}^{d}$, then the function $y \to h(y, u(y), \mathbf{u}(y))$ is measurable from $\mathbb{R}^{d}$ into $\mathbb{R}$. Choosing in particular $u(y)= \phi(y,z)$ and $\mathbf{u}(y)= \psi(y,z)$, ($y\in \mathbb{R}^{d}$), we see that the function $y \rightarrow a_{i}(y,z, \phi(y,z), \psi(y,z))$ is measurable from $\mathbb{R}^{d}$ into $\mathbb{R}$, and that for any arbitrary $z\in \mathbb{R}^{d}$. On the other hand, using [part (ii) of] \textbf{(H$_{1}$)} and \eqref{1.3} one can easily show that the function $z \rightarrow a_{i}(y,z, \phi(y,z), \psi(y,z))$ (for fixed $y$) is continuous on $\mathbb{R}^{d}$. Taking account \eqref{ca}, we deduce that, a.e. in $y \in \mathbb{R}^{d}$, the function $z \rightarrow a_{i}(y,z, \phi(y,z), \psi(y,z))$ of $\mathbb{R}^{d}$ into $\mathbb{R}$, denoted below by $a_{i}(y,\cdot, \phi(y,\cdot), \psi(y,\cdot))$, lies in $\mathcal{C}_{per}(Z)$. Hence it follows by \cite[Lemma 2.1]{luka}  that the function $y \rightarrow a_{i}(y,\cdot, \phi(y,\cdot), \psi(y,\cdot))$, denoted by $a_{i}(\cdot,\cdot, \phi, \psi)$, is measurable from $\mathbb{R}^{d}$ to  $\mathcal{C}_{per}(Z)$. From this we deduce using [part (ii) of ] \textbf{(H$_{1}$)}, \eqref{1.3} and \eqref{ca}, that $a_{i}(\cdot,\cdot, \phi, \psi)$ lies in $L^{\infty}_{per}(Y; \mathcal{C}_{per}(Z))$.
		
		But $L^{\infty}_{per}(Y; \mathcal{C}_{per}(Z)) \subset L^{\widetilde{\Phi}}_{per}(Y; \mathcal{C}_{per}(Z))$. Hence \eqref{tj} follows by the facts that $A= \mathcal{C}_{per}(Y\times Z)$ is dense in $ L^{\widetilde{\Phi}}_{per}(Y; \mathcal{C}_{per}(Z))$ and the latter space is continuously embedded in $\Xi^{\widetilde{\Phi}}(\mathbb{R}^{d}_{y}; \mathcal{B}(\mathbb{R}^{d}_{z}))$.
	\end{proof}
	\begin{remarq}
		\textup{	In the classical frame of Lebesgue spaces, the corresponding $N$-function $\widetilde{\Phi}$ satisfy the $\Delta'$-condition (see, e.g. \cite{luka}). }
	\end{remarq}
	
	\subsection{Problem II: $a_{i}(\cdot,\cdot, \zeta,\lambda) \in \mathcal{B}_{\infty}(\mathbb{R}^{d}_{z}; \mathcal{C}_{per}(Y))$}
	
	We study here the homogenization of \ref{1.1} under the structure hypothesis:
	\begin{equation}\label{ce}
		a_{i}(\cdot,\cdot, \zeta,\lambda) \in \mathcal{B}_{\infty}(\mathbb{R}^{d}_{z}; \mathcal{C}_{per}(Y)) \text{ for any } (\zeta,\lambda) \in \mathbb{R}\times \mathbb{R}^{d} \;\; (1\leq i \leq d),
	\end{equation}
	where $\mathcal{B}_{\infty}(\mathbb{R}^{d}_{z}; \mathcal{C}_{per}(Y))$ denotes the space of continuous functions $\psi: \mathbb{R}^{d}_{z} \rightarrow \mathcal{C}_{per}(Y)$ such that $\psi(z)= \psi(\cdot,z)$ has a limit in $\mathcal{B}(\mathbb{R}^{d}_{z})$ when $|z|\to \infty$.
	
	But, since $A= \mathcal{B}_{\infty}(\mathbb{R}^{d}_{z}; \mathcal{C}_{per}(Y)) =  \mathcal{C}_{per}(Y) \odot \mathcal{B}_{\infty}(\mathbb{R}^{d}_{z})$ satisfy all the assumption of Theorem \ref{ti}, we see by Theorem  \ref{mainresult} that the whole problem reduces to verifying that \eqref{tj} holds.
	
	This is complete once we have established that
	\begin{equation}\label{cd}
		a_{i}(\cdot,\cdot, f, \mathbf{f}) \in \mathcal{B}_{\infty}(\mathbb{R}^{d}_{z}; \mathcal{C}_{per}(Y)) \text{ for all } ( f, \mathbf{f}) \in A_{\mathbb{R}}\times (A_{\mathbb{R}})^{d} \;\; (1\leq i \leq d).
	\end{equation}
	\begin{proof}[Proof of \eqref{cd}]
		To do this, let $1\leq i \leq d$ and $( f, \mathbf{f}) \in A_{\mathbb{R}}\times (A_{\mathbb{R}})^{d}$ be freely fixed. Let $K$ (resp. $L$) be a compact sets in $\mathbb{R}$ (resp. $\mathbb{R}^{d}$) such that $f(y,z) \in K$ (resp. $\mathbf{f}(y,z) \in L$) for all  $(y,z) \in \mathbb{R}^{d}\times \mathbb{R}^{d}$. According to \eqref{ce}, we may view $a_{i}$ as a function $(\zeta,\lambda) \rightarrow a_{i}(\cdot,\cdot, \zeta,\lambda)$ of $\mathbb{R}\times \mathbb{R}^{d}$ into $A_{\mathbb{R}}$, which function lies in $\mathcal{C}(\mathbb{R}\times \mathbb{R}^{d} ; A_{\mathbb{R}})$ (use \eqref{1.3}). Still calling $a_{i}$ the restriction of the latter function to $K\times L$, we have therefore $a_{i} \in \mathcal{C}(K\times L ; A_{\mathbb{R}})$. Recalling that $\mathcal{C}(K\times L)\otimes A_{\mathbb{R}}$ is dense in $\mathcal{C}(K\times L ; A_{\mathbb{R}})$ (see, e.g., |6, p.46), we see that we may consider a sequence $(q_{n})_{n\geq 1}$ in $\mathcal{C}(K\times L)\otimes A_{\mathbb{R}}$ such that 
		\begin{equation*}
			\sup_{(y,z,\zeta,\lambda)\in \mathbb{R}^{d}\times \mathbb{R}^{d}\times K\times L} |q_{n}(y,z,\zeta,\lambda) - a_{i}(y,z,\zeta,\lambda)| \rightarrow 0 \text{ as } n \to \infty.
		\end{equation*}
		Hence $q_{n}(\cdot,\cdot,f,\mathbf{f}) \rightarrow a_{i}(\cdot,\cdot,f,\mathbf{f})$ in $\mathcal{B}(\mathbb{R}^{d}_{y}\times \mathbb{R}^{d}_{z})$ as $n\to \infty$.
		
		Thus, \eqref{cd} is proved if we can check that each $q_{n}(\cdot,\cdot,f,\mathbf{f})$ lies in $A$. However, it is enough to verify that we have $q_{n}(\cdot,\cdot,f,\mathbf{f}) \in A$ for any function $q: \mathbb{R}^{d}_{y}\times \mathbb{R}^{d}_{z} \times \mathbb{R}_{\zeta}\times \mathbb{R}^{d}_{\lambda} \rightarrow \mathbb{R}$ of the form 
		\begin{eqnarray*}
			q(y,z,\zeta,\lambda) = \chi(\zeta,\lambda) \phi(y,z) \;\; (y,z,\lambda \in \mathbb{R}^{d}  \nonumber \\ \text{ and } \zeta\in \mathbb{R}) \text{ with } \chi \in \mathcal{C}(K\times L; \mathbb{R}) 
			\text{ and } \phi \in A_{\mathbb{R}}.
		\end{eqnarray*}
		
		Given one such $q$, by the Stone-Weierstrass theorem we may consider a sequence $(f_{n})$ of polynomials in $(\zeta,\lambda) = (\zeta, \lambda_{1}, \cdots, \lambda_{d}) \in K\times L$ such that $f_{n} \rightarrow \chi$ in $\mathcal{C}(K\times L)$ as $n\to \infty$, hence $f_{n}(f, \mathbf{f}) \rightarrow \chi(f, \mathbf{f})$ in $\mathcal{B}(\mathbb{R}^{d}_{y}\times \mathbb{R}^{d}_{z})$ as $n\to \infty$. We deduce that $\chi(f, \mathbf{f})$ lies in $A_{\mathbb{R}}$, since the same is true of each $f_{n}(f, \mathbf{f})$ ($A_{\mathbb{R}}$ being an algebra). Whence the result.
	\end{proof}

	\subsection{Problem III: $a_{i}(\cdot,\cdot, \zeta,\lambda) \in \mathcal{B}_{\infty}(\mathbb{R}^{d}_{z}; AP(\mathbb{R}^{d}_{y}))$}
	
	We study here the homogenization of \ref{1.1} under the structure hypothesis:
	\begin{equation*}\label{cf}
		a_{i}(\cdot,\cdot, \zeta,\lambda) \in \mathcal{B}_{\infty}(\mathbb{R}^{d}_{z}; AP(\mathbb{R}^{d}_{y})) \text{ for any } (\zeta,\lambda) \in \mathbb{R}\times \mathbb{R}^{d} \;\; (1\leq i \leq d).
	\end{equation*}
	
	Note that, $A= \mathcal{B}_{\infty}(\mathbb{R}^{d}_{z}; AP(\mathbb{R}^{d}_{y})) = AP(\mathbb{R}^{d}_{y}) \odot \mathcal{B}_{\infty}(\mathbb{R}^{d}_{z})$ satisfy all the assumption of Theorem \ref{ti}, we see by Theorem  \ref{mainresult} that the whole problem reduces to verifying that \eqref{tj} holds.
	
	This is complete once we have established that
	\begin{equation*}\label{cg}
		a_{i}(\cdot,\cdot, f, \mathbf{f}) \in \mathcal{B}_{\infty}(\mathbb{R}^{d}_{z}; AP(\mathbb{R}^{d}_{y})) \text{ for all } ( f, \mathbf{f}) \in A_{\mathbb{R}}\times (A_{\mathbb{R}})^{d} \;\; (1\leq i \leq d).
	\end{equation*}
	However, the proof of this is entirely identical to the case of Problem II and therefore is not worth repeating.
	
	\subsection{Problem IV: $a_{i}(\cdot,\cdot, \zeta,\lambda) \in  AP(\mathbb{R}^{d}_{y}\times \mathbb{R}^{d}_{z})$}
	
	We study here the homogenization of \ref{1.1} under the structure hypothesis:
	\begin{equation*}\label{ch}
		a_{i}(\cdot,\cdot, \zeta,\lambda) \in  AP(\mathbb{R}^{d}_{y}\times \mathbb{R}^{d}_{z}) \text{ for any } (\zeta,\lambda) \in \mathbb{R}\times \mathbb{R}^{d} \;\; (1\leq i \leq d).
	\end{equation*}
	
	Note that, $A=  AP(\mathbb{R}^{d}_{y}\times \mathbb{R}^{d}_{z}) = AP(\mathbb{R}^{d}_{y}) \odot  AP(\mathbb{R}^{d}_{z})$ satisfy all the assumption of Theorem \ref{ti}, we see by Theorem  \ref{mainresult} that the whole problem reduces to verifying that \eqref{tj} holds.
	
	This is complete once we have established that
	\begin{equation*}\label{ci}
		a_{i}(\cdot,\cdot, f, \mathbf{f}) \in  AP(\mathbb{R}^{d}_{y}\times \mathbb{R}^{d}_{z}) \text{ for all } ( f, \mathbf{f}) \in A_{\mathbb{R}}\times (A_{\mathbb{R}})^{d} \;\; (1\leq i \leq d).
	\end{equation*}
	However, the proof of this is \textit{verbatim} copy of case of Problem II and therefore is not worth repeating.
	
	\subsection{Problem V: $a_{i}(\cdot,\cdot, \zeta,\lambda) \in  \mathcal{B}_{\infty}(\mathbb{R}^{d}_{z}; \mathcal{B}_{\infty,\mathbb{Z}^{d}}(\mathbb{R}^{d}_{y}))$}
	
	
	We intend to study here the homogenization of \ref{1.1} under the structure hypothesis:
	\begin{equation*}\label{cj}
		a_{i}(\cdot,\cdot, \zeta,\lambda) \in  \mathcal{B}_{\infty}(\mathbb{R}^{d}_{z}; \mathcal{B}_{\infty,\mathbb{Z}^{d}}(\mathbb{R}^{d}_{y})) \text{ for any } (\zeta,\lambda) \in \mathbb{R}\times \mathbb{R}^{d} \;\; (1\leq i \leq d).
	\end{equation*}
	
	Here, $A=  \mathcal{B}_{\infty}(\mathbb{R}^{d}_{z}; \mathcal{B}_{\infty,\mathbb{Z}^{d}}(\mathbb{R}^{d}_{y}))  = \mathcal{B}_{\infty,\mathbb{Z}^{d}}(\mathbb{R}^{d}_{y}) \odot  \mathcal{B}_{\infty}(\mathbb{R}^{d}_{z})$ satisfy all the assumption of Theorem \ref{ti}, we see by Theorem  \ref{mainresult} that the whole problem reduces to verifying that \eqref{tj} holds.
	
	This is complete once we have established that
	\begin{equation*}\label{cl}
		a_{i}(\cdot,\cdot, f, \mathbf{f}) \in \mathcal{B}_{\infty}(\mathbb{R}^{d}_{z}; \mathcal{B}_{\infty,\mathbb{Z}^{d}}(\mathbb{R}^{d}_{y})) \text{ for all } ( f, \mathbf{f}) \in A_{\mathbb{R}}\times (A_{\mathbb{R}})^{d} \;\; (1\leq i \leq d).
	\end{equation*}
	But this follows by repeating word for word the proof in the case of Problem II.
	
	\begin{corollary}
		Suppose we have $a_{i}(\cdot,\cdot, \zeta,\lambda) \in  \mathcal{B}_{\infty}(\mathbb{R}^{d}_{z}; \mathcal{B}_{\infty,\mathbb{Z}^{d}}(\mathbb{R}^{d}_{y})) \text{ for all } (\zeta,\lambda) \in \mathbb{R}\times \mathbb{R}^{d} \text{ and for all } 1\leq i \leq d.$ Since $\mathcal{B}_{\infty}(\mathbb{R}^{d}_{z}; \mathcal{B}_{\infty}(\mathbb{R}^{d}_{y})) \subset \mathcal{B}_{\infty}(\mathbb{R}^{d}_{z}; \mathcal{B}_{\infty,\mathbb{Z}^{d}}(\mathbb{R}^{d}_{y}))$, then we have the same conclusion for $A= \mathcal{B}_{\infty}(\mathbb{R}^{d}_{z}; \mathcal{B}_{\infty}(\mathbb{R}^{d}_{y}))$.
	\end{corollary}
	
	\subsection{Problem VI: $a_{i}(\cdot,\cdot, \zeta,\lambda) \in  \mathcal{C}_{per}(Z; \mathcal{B}_{\infty,\mathbb{Z}^{d}}(\mathbb{R}^{d}_{y}))$}

	We intend to study here the homogenization of \ref{1.1} under the structure hypothesis:
	\begin{equation*}\label{cm}
		a_{i}(\cdot,\cdot, \zeta,\lambda) \in  \mathcal{C}_{per}(Z; \mathcal{B}_{\infty,\mathbb{Z}^{d}}(\mathbb{R}^{d}_{y})) \text{ for any } (\zeta,\lambda) \in \mathbb{R}\times \mathbb{R}^{d} \;\; (1\leq i \leq d),
	\end{equation*}
	where $\mathcal{C}_{per}(Z; \mathcal{B}_{\infty,\mathbb{Z}^{d}}(\mathbb{R}^{d}_{y})) $ is the space of $Z$-periodic continuous functions of $\mathbb{R}^{d}_{z}$ into $\mathcal{B}_{\infty,\mathbb{Z}^{d}}(\mathbb{R}^{d}_{y})$.
	
	Here, $A=  \mathcal{C}_{per}(Z; \mathcal{B}_{\infty,\mathbb{Z}^{d}}(\mathbb{R}^{d}_{y}))  = \mathcal{B}_{\infty,\mathbb{Z}^{d}}(\mathbb{R}^{d}_{y}) \odot  \mathcal{C}_{per}(Z)$ satisfy all the assumption of Theorem \ref{ti}, we see by Theorem  \ref{mainresult} that the whole problem reduces to verifying that \eqref{tj} holds.
	
	This is complete once we have established that
	\begin{equation*}\label{cn}
		a_{i}(\cdot,\cdot, f, \mathbf{f}) \in \mathcal{C}_{per}(Z; \mathcal{B}_{\infty,\mathbb{Z}^{d}}(\mathbb{R}^{d}_{y}))  \text{ for all } ( f, \mathbf{f}) \in A_{\mathbb{R}}\times (A_{\mathbb{R}})^{d} \;\; (1\leq i \leq d).
	\end{equation*}
	But, it is obtained by repeating also  the proof in the case of Problem II.
	
	\subsection{Problem VII: $a_{i}(\cdot,\cdot, \zeta,\lambda) \in AP(\mathbb{R}^{d}_{y}; WAP(\mathbb{R}^{d}_{z}))$}

	We intend to study here the homogenization of \ref{1.1} under the structure hypothesis:
	\begin{equation*}\label{co}
		a_{i}(\cdot,\cdot, \zeta,\lambda) \in  AP(\mathbb{R}^{d}_{y}; WAP(\mathbb{R}^{d}_{z})) \text{ for any } (\zeta,\lambda) \in \mathbb{R}\times \mathbb{R}^{d} \;\; (1\leq i \leq d),
	\end{equation*}
	where $AP(\mathbb{R}^{d}_{y}; WAP(\mathbb{R}^{d}_{z})) $ is the space of all almost periodic continuous functions of $\mathbb{R}^{d}_{y}$ into $ WAP(\mathbb{R}^{d}_{z})$.
	
	Here, $A= AP(\mathbb{R}^{d}_{y}; WAP(\mathbb{R}^{d}_{z})) = AP(\mathbb{R}^{d}_{y}) \odot  WAP(\mathbb{R}^{d}_{z})$ satisfy all the assumption of Theorem \ref{ti}, we see by Theorem  \ref{mainresult} that the whole problem reduces to verifying that \eqref{tj} holds.
	
	Hence, one can solve the problem in this case by establishing that
	\begin{equation*}\label{cp}
		a_{i}(\cdot,\cdot, f, \mathbf{f}) \in AP(\mathbb{R}^{d}_{y}; WAP(\mathbb{R}^{d}_{z})) \text{ for all } ( f, \mathbf{f}) \in A_{\mathbb{R}}\times (A_{\mathbb{R}})^{d} \;\; (1\leq i \leq d).
	\end{equation*}

	\subsection{Problem VIII: $a_{i}(\cdot,\cdot, \zeta,\lambda) \in WAP(\mathbb{R}^{d}_{y}\times\mathbb{R}^{d}_{z})$}

	We intend to study here the homogenization of \ref{1.1} under the structure hypothesis:
	\begin{equation*}\label{cq}
		a_{i}(\cdot,\cdot, \zeta,\lambda) \in  WAP(\mathbb{R}^{d}_{y}\times\mathbb{R}^{d}_{z}) \text{ for any } (\zeta,\lambda) \in \mathbb{R}\times \mathbb{R}^{d} \;\; (1\leq i \leq d).
	\end{equation*}
	
	Here, $A=WAP(\mathbb{R}^{d}_{y}\times\mathbb{R}^{d}_{z}) = WAP(\mathbb{R}^{d}_{y}) \odot  WAP(\mathbb{R}^{d}_{z})$ satisfy all the assumption of Theorem \ref{ti}, we see by Theorem  \ref{mainresult} that the whole problem reduces to verifying that \eqref{tj} holds.
	
	Hence, one can solve the problem in this case by establishing that
	\begin{equation*}\label{cr}
		a_{i}(\cdot,\cdot, f, \mathbf{f}) \in WAP(\mathbb{R}^{d}_{y}\times\mathbb{R}^{d}_{z}) \text{ for all } ( f, \mathbf{f}) \in A_{\mathbb{R}}\times (A_{\mathbb{R}})^{d} \;\; (1\leq i \leq d).
	\end{equation*}

	\subsection{Problem IX: $a_{i}(\cdot,\cdot, \zeta,\lambda) \in FS(\mathbb{R}^{d}_{y}\times\mathbb{R}^{d}_{z})$}

	We intend to study here the homogenization of \ref{1.1} under the structure hypothesis:
	\begin{equation*}\label{cs}
		a_{i}(\cdot,\cdot, \zeta,\lambda) \in FS(\mathbb{R}^{d}_{y}\times\mathbb{R}^{d}_{z}) \text{ for any } (\zeta,\lambda) \in \mathbb{R}\times \mathbb{R}^{d} \;\; (1\leq i \leq d).
	\end{equation*} 
	
	However, $A=FS(\mathbb{R}^{d}_{y}\times\mathbb{R}^{d}_{z}) = FS(\mathbb{R}^{d}_{y})\odot FS(\mathbb{R}^{d}_{z})$ satisfy all the assumption of Theorem \ref{ti}, we see by Theorem  \ref{mainresult} that the whole problem reduces to verifying that \eqref{tj} holds.
	
	Hence, one can solve the problem in this case by establishing that
	\begin{equation*}\label{ct}
		a_{i}(\cdot,\cdot, f, \mathbf{f}) \in FS(\mathbb{R}^{d}_{y}\times\mathbb{R}^{d}_{z}) \text{ for all } ( f, \mathbf{f}) \in A_{\mathbb{R}}\times (A_{\mathbb{R}})^{d} \;\; (1\leq i \leq d).
	\end{equation*}

	\subsection{Problem X: $a_{i}(\cdot,\cdot, \zeta,\lambda) \in FS(\mathbb{R}^{d}_{y})\odot A_{z}$}

	We intend to study here the homogenization of \ref{1.1} under the structure hypothesis:
	\begin{equation*}\label{cu}
		a_{i}(\cdot,\cdot, \zeta,\lambda) \in FS(\mathbb{R}^{d}_{y})\odot A_{z} \text{ for any } (\zeta,\lambda) \in \mathbb{R}\times \mathbb{R}^{d} \;\; (1\leq i \leq d),
	\end{equation*} 
	where $A_{z}$ is an $H$-algebra as in Theorem \ref{ti}.
	
	However, $A= FS(\mathbb{R}^{d}_{y})\odot A_{z}$ satisfy all the assumption of Theorem \ref{ti}, we see by Theorem  \ref{mainresult} that the whole problem reduces to verifying that \eqref{tj} holds.
	
	Hence, one can solve the problem in this case by establishing that
	\begin{equation*}\label{cv}
		a_{i}(\cdot,\cdot, f, \mathbf{f}) \in FS(\mathbb{R}^{d}_{y})\odot A_{z} \text{ for all } ( f, \mathbf{f}) \in A_{\mathbb{R}}\times (A_{\mathbb{R}})^{d} \;\; (1\leq i \leq d).
	\end{equation*}

	\section{Appendix} \label{labelappendix}

	\subsection{Traces results}\label{appen1}
	
	Our goal here is to prove result \eqref{tk} under the \textit{abstract structure hypothesis} \textbf{(H$_{7}$)}. We use the same method to that used in \cite[Sect. 3.1]{tacha4}.
	
	For that, let $u\in L^{\Phi}_{\textrm{loc}}(\mathbb{R}^{d}_{x}) \times L^{\widetilde{\Phi}}_{\textrm{loc}}(\mathbb{R}^{d}_{y} ; \mathcal{B}(\mathbb{R}^{d}_{z}))$, that is 
	\begin{equation*}
		u = \sum_{i\in I} u_{i} \otimes v_{i} \quad \text{with } u_{i}\in L^{\Phi}_{\textrm{loc}}(\mathbb{R}^{d}_{x}) \;\; \text{and } \; v_{i} \in L^{\widetilde{\Phi}}_{\textrm{loc}}(\mathbb{R}^{d}_{y} ; \mathcal{B}(\mathbb{R}^{d}_{z})),
	\end{equation*}
	where $I$ is a finite set (depending to $u$). Letting 
	\begin{equation*}
		u^{\varepsilon}(x) = \sum_{i\in I} u_{i}(x) \otimes v_{i}^{\varepsilon}(x), \quad x \in \mathbb{R}^{d}, 
	\end{equation*}
	where $v_{i}^{\varepsilon}(x) = v_{i}\left(\frac{x}{\varepsilon_{1}}, \frac{x}{\varepsilon_{2}}\right)$, one has $u^{\varepsilon} \in L^{1}_{\textrm{loc}}(\mathbb{R}^{d}_{x})$ with 
	\begin{equation*}
		u^{\varepsilon}(x) = u\left(x,\dfrac{x}{\varepsilon_{1}}, \dfrac{x}{\varepsilon_{2}}\right), \quad x \in \mathbb{R}^{d}. 
	\end{equation*}
	This define a linear operator $u \to u^{\varepsilon}$ from $L^{\Phi}_{\textrm{loc}}(\mathbb{R}^{d}_{x}) \times L^{\widetilde{\Phi}}_{\textrm{loc}}(\mathbb{R}^{d}_{y} ; \mathcal{B}(\mathbb{R}^{d}_{z}))$ into $L^{1}_{\textrm{loc}}(\mathbb{R}^{d}_{x})$, and we have the following lemma whose proof is identically to that of \cite[Lemma 3.1]{tacha4}.
	
	\begin{lemma}\label{tl}
		The restriction on $\mathcal{B}(\mathbb{R}^{d}_{x})\otimes \Xi^{\widetilde{\Phi}}(\mathbb{R}^{d}_{y}; \mathcal{B}(\mathbb{R}^{d}_{z})$ of the linear operator $u \to u^{\varepsilon}$ defined above extends by continuity to a mapping, still denoted by $u \to u^{\varepsilon}$, of $\mathcal{B}(\mathbb{R}^{d}_{x} ; \Xi^{\widetilde{\Phi}}(\mathbb{R}^{d}_{y}; \mathcal{B}(\mathbb{R}^{d}_{z}))$ into $\Xi^{\widetilde{\Phi}}(\mathbb{R}^{d}_{x})$ with 
		\begin{equation*}
			\|u^{\varepsilon}\|_{\Xi^{\widetilde{\Phi}}(\mathbb{R}^{d}_{x})} \leq \|u\|_{\mathcal{B}(\mathbb{R}^{d}_{x} ; \Xi^{\widetilde{\Phi}}(\mathbb{R}^{d}_{y}; \mathcal{B}(\mathbb{R}^{d}_{z}))}.
		\end{equation*}
	\end{lemma}
	
	As in \cite[Lemma 3.2]{tacha4}, we also have the following lemma whose proof can be deduce from the above lemma.
	
	\begin{lemma}\label{tm}
		Let $A$ be an $RH$-supralgebra of class $\mathcal{C}^{\infty}$ and let $f$ be a function of $\mathbb{R}^{d}_{y}\times \mathbb{R}^{d}_{z}\times\mathbb{R}^{m}_{\mu}$ into $\mathbb{C}$ with the following property:
		\begin{equation*}
			f \in \mathcal{B}(\mathbb{R}^{m}_{\mu}; \mathfrak{X}_{A}^{\widetilde{\Phi}}(\mathbb{R}^{d}_{y}; \mathcal{B}(\mathbb{R}^{d}_{z})).
		\end{equation*}
		Then for all $\mathbf{v} \in (A)^{m}$, the (trace) function $(y,z) \to f(y, z, \mathbf{v}(y,z))$ from $\mathbb{R}^{d}_{y}\times \mathbb{R}^{d}_{z}$ into $\mathbb{C}$, denoted by $f(\cdot, \cdot, \mathbf{v})$, is well defined and belongs to $\mathfrak{X}_{A}^{\widetilde{\Phi}}(\mathbb{R}^{d}_{y}; \mathcal{B}(\mathbb{R}^{d}_{z})$.
	\end{lemma}
	
	Now, we are in position to establish \eqref{tk} for all  $f \in A_{\mathbb{R}}$ and $\mathbf{f} \in (A_{\mathbb{R}})^{d}$. As in the proof of \cite[Proposition 3.3]{tacha4}, we use Lemma \ref{tm} to  have the following.
	
	\begin{proposition}\label{tn}
		Let $A$ be an $RH$-supralgebra of class $\mathcal{C}^{\infty}$ such that \eqref{tj} holds. Then for every $f \in A_{\mathbb{R}}$ and $\mathbf{f} \in (A_{\mathbb{R}})^{d}$, we have \eqref{tk}.
	\end{proposition}
	
	Using the proof of \cite[Lemma 2.10]{nguetseng2025homogenization} replacing integer $q$ by any $N$-function $\Phi \in \Delta_{2}$, we have the following corollary.
	\begin{corollary}
		Let \eqref{tk} be satisfied. Then, the following hold for any arbitrary $f \in A_{\mathbb{R}}$ and $\mathbf{f} \in (A_{\mathbb{R}})^{d}$.
		\begin{itemize}
			\item[(i)] One has $a_{i}(\cdot, \cdot, f, \mathbf{f}) \in L^{\infty}(\mathbb{R}^{d}_{y}; \mathcal{B}(\mathbb{R}^{d}_{z}))$ $(1\leq i\leq d)$. 
			\item[(ii)] For a.e. $y \in \mathbb{R}^{d}_{y}$, we have $a_{i}(y, \cdot, f(y,\cdot), \mathbf{f}(y,\cdot)) \in A_{z}$ $(1\leq i\leq d)$.
		\end{itemize}
	\end{corollary}
	
	In the following proposition, we recall the property of function $a(\cdot, \cdot, w, \mathbf{v})$ in \eqref{1.1} when $\left( v,%
	\mathbf{v}\right) \in \mathcal{C}\left( \overline{\Omega };\mathcal{B}\left( 
	\mathbb{R}_{y}^{d}\times 
	\mathbb{R}_{z}^{d}\right) \right)^{d+1}$
	
	\begin{proposition}\cite{tacha4}\label{cor42} \\
		\bigskip Let $a:= (a_i)_{1\leq i \leq d}:\mathbb R^d \times\mathbb R^d \times \mathbb R \times \mathbb R^d\to \mathbb R^d $ satisfy \textbf{(H$_{1}$)}-\textbf{(H$_{4}$)} and let $\left( v,\mathbf{v}\right) \in \mathcal{B}\left( 
		\mathbb{R}_{y}^{d}\times 
		\mathbb{R}
		_{z}^{d}\right)^{d+1}.$ For $1\leq i\leq d,$ the function $\left(
		y,z\right) \rightarrow a_{i}\left( y,z,v\left( y,z\right) ,\mathbf{v}\left(
		y,z\right) \right) $ from $\mathbb{R}_{y}^{d}\times 
		\mathbb{R}_{z}^{d}$ into $\mathbb{R}$ is an element of $L^{\infty }\left( 
		\mathbb{R}_{y}^{d},\mathcal{B}\left( 
		\mathbb{R}
		_{z}^{d}\right) \right) ,$ denoted as $a_{i}\left( \cdot, \cdot,v,\mathbf{v}\right)$.
		
		Moreover, the function $%
		x \in \Omega \rightarrow a_{i}\left( \frac{x}{\varepsilon },\frac{x}{\varepsilon ^{2}}%
		,w\left( x,\frac{x}{\varepsilon },\frac{x}{\varepsilon ^{2}}\right) ,\mathbf{%
			w}\left( x,\frac{x}{\varepsilon },\frac{x}{\varepsilon ^{2}}\right) \right) \in
		\mathbb{R}
		,$ denoted  by $a_{i}\left( \cdot, \cdot, w^{\varepsilon },\mathbf{w}^{\varepsilon }\right)
		,$ is well defined and belongs to $L^{\infty }\left( \Omega \right) .$
	\end{proposition}

	\subsection{Weak solution of problem \eqref{1.1}}\label{appen2}
	
	\noindent The aim of this section consists to establish of the existence of a weak solution to \eqref{1.1}, neglecting the \textit{abstract structure hypothesis} on $a$. It is done in \cite[Sect. 3 ]{tacha4}, but for the reader's convenience we recall it here.
	
	\begin{proposition}\label{propalphabeta}\cite{tacha4}  \\
		Let $(a_i)_{1\leq i \leq d}$ be the functions in \eqref{1.1} and assume that they satisfy the assumptions \textbf{(H$_{1}$)}-\textbf{(H$_{4}$)}. Let $a^{\varepsilon }\left( \cdot, \cdot,v,%
		\mathbf{v}\right) =\left( a_{i}^{\varepsilon }\left( \cdot, \cdot,v,\mathbf{v}\right)
		\right) _{1\leq i\leq d},$ then
		the transformation $\left( v,\mathbf{v}\right) \rightarrow a^{\varepsilon
		}\left( \cdot, \cdot,v,\mathbf{v}\right) $ of $\ \mathcal{C}\left( \overline{\Omega }
		\right) \times \mathcal{C}\left( \overline{\Omega }
		\right) ^{d}$ into $L^{\infty }(\Omega)^{d}$ extends by continuity to a mapping, still denoted by $\left( v,%
		\mathbf{v}\right) \rightarrow a^{\varepsilon }\left( \cdot, \cdot,v,\mathbf{v}\right)
		,$ from $L^{\Phi }\left( \Omega
		\right) \times L^{\Phi }\left( \Omega
		\right) ^{d}$ into $L^{\widetilde{\Phi }}\left( \Omega
		\right) ^{d}$ such that
		\begin{equation*}\label{Lest}
			\left\Vert a^{\varepsilon }\left( \cdot, \cdot,v,\mathbf{v}\right) -a^{\varepsilon
			}\left( \cdot, \cdot,w,\mathbf{w}\right) \right\Vert _{\widetilde{\Phi },\Omega }\leq
			c\left\Vert v-w\right\Vert _{\Phi ,\Omega }^{\alpha }+c^{\prime }\left\Vert
			\bf v-\bf w\right\Vert _{\Phi ,\Omega }^{\beta },  
		\end{equation*}
		for all $\left( v,\mathbf{v}\right) ,\left( w,\mathbf{w}\right) \in L^{\Phi
		}\left( \Omega
		\right) \times L^{\Phi }\left( \Omega
		\right) ^{d}$ where $c,c^{\prime }>0$  \\ and $\alpha ,\beta \in \left\{ \frac{%
			\rho _{1}}{\rho _{0}}\left( \rho _{0}-1\right) ,\frac{\rho _{2}}{\rho _{0}}%
		\left( \rho _{0}-1\right) ,\right.$ 	$\left.\rho _{1}-1,\frac{\rho _{2}}{\rho _{1}}\left(
		\rho _{1}-1\right) \right\},$
		with the constants $\rho_0,\rho_1$ and $\rho_2$ as in \eqref{1.2}.
	\end{proposition}
	
	On the other hand, by (i) and (ii) in \textbf{(H$_{1}$)}, the assumptions on $f$, \eqref{1.3} in \textbf{(H$_{2}$)}, \textbf{(H$_{3}$)} and \textbf{(H$_{4}$)} it follows that for every $\varepsilon >0$, the assumptions of \cite[Theorem 3.2]{Y} are satisfied, hence there exists $%
	u_{\varepsilon }\in W_{0}^{1}L^{\Phi }\left( \Omega
	\right) \cap L^{\infty }\left( \Omega\right) $ weak solution of \eqref{1.1}, i.e. 
	\begin{equation}\label{eqvarepsilon}
		\left\{	\begin{array}{l}
			\displaystyle 	\exists\;  u_{\varepsilon }\in W_{0}^{1}L^{\Phi }\left( \Omega
			\right) \cap L^{\infty }\left( \Omega\right) \hbox{ such that } \\
			\displaystyle 	\int_{\Omega }a\left( \frac{x}{\varepsilon },\frac{x}{%
				\varepsilon ^{2}},u_{\varepsilon },Du_{\varepsilon }\right)
			Dvdx=\int_{\Omega }fvdx, \hbox{ for every }v\in  W_{0}^{1}L^{\Phi }\left( \Omega
			\right). 
		\end{array}\right.
	\end{equation} Thus, by \eqref{1.4} in \textbf{(H$_{4}$)} and \textbf{(H$_{5}$)}, it follows that 
	\begin{align*}\int_{\Omega }\theta \Phi \left( \left\vert
		Du_{\varepsilon }\right\vert \right) dx\leq \int_{\Omega }a\left( \frac{x}{%
			\varepsilon },\frac{x}{\varepsilon ^{2}},u_{\varepsilon },Du_{\varepsilon
		}\right) Du_{\varepsilon }dx=\int_{\Omega }fu_{\varepsilon }dx\leq
		2\left\Vert f\right\Vert _{\widetilde{\Phi },\Omega }\left\Vert
		u_{\varepsilon }\right\Vert _{\Phi ,\Omega },
	\end{align*}
	where 
	\begin{align}
		\label{theta}
		\theta:= {\widetilde \Phi^{-1}}(\Phi (\min_{t\geq 0} h(t)) )> 0.
	\end{align}
	If $\left\Vert \left\vert Du_{\varepsilon }\right\vert \right\Vert _{\Phi
		,\Omega }\geq 1,$ we have 
	\begin{align*}\theta \left\Vert \left\vert Du_{\varepsilon
		}\right\vert \right\Vert _{\Phi ,\Omega }^{\sigma }\leq \int_{\Omega }\theta
		\Phi \left( \left\vert Du_{\varepsilon }\right\vert \right) dx\leq
		2\left\Vert f\right\Vert _{\widetilde{\Phi },\Omega }\left\Vert
		u_{\varepsilon }\right\Vert _{\Phi ,\Omega }\leq c\left\Vert \left\vert
		Du_{\varepsilon }\right\vert \right\Vert _{\Phi ,\Omega };\sigma >1.
	\end{align*} 
	We deduce therefore that: $$\underset{0<\varepsilon }{\sup }\left\Vert
	u_{\varepsilon } \right\Vert _{W_{0}^{1}L^{\Phi
		}\left( \Omega
		\right) }<+\infty .$$
	
	It is easily observed (as in Remark \ref{remuniq}) that if \textbf{(H$_{6}$)} is satisfied the solution in \eqref{eqvarepsilon} is unique.

	\vspace{0.5cm}
	
	\noindent	\textbf{Acknowledgments.}
	\emph{The authors would like to thank the anonymous referee for his/her pertinent remarks, comments and suggestions.}

	\section*{Declarations}
	
	\begin{itemize}
		\item Funding : No funding was received to to assist with the preparation of this manuscript.
		\item Conflict of interest/Competing interests : We have no conflicts of interest to disclose. 
		\item Consent to participate : All authors consented to participate in this work.
		\item Consent for publication : We are enclosing herewith a manuscript entitled ``Reiterated $\Sigma$-convergence in Orlicz setting and Applications" submitted to the journal ``++++++" for possible evaluation. 
		\item Ethics approval : With the submission of this manuscript we would like to undertake that the above mentioned manuscript has not been published elsewhere, accepted for publication elsewhere or under editorial review for publication elsewhere. 
		\item Availability of data and materials : `Not applicable'
		\item Code availability : `Not applicable'
		\item Authors contributions : The authors contributed equally to this work. The first draft of the manuscript was written by \textsc{Tchinda Takougoum Franck Arnold} and all authors commented on previous versions of the manuscript. All authors read and approved the final manuscript.
	\end{itemize}
	

	\bibliographystyle{abbrv}
	\bibliography{stoc_sigma_biblio1}
	
\end{document}